\documentclass[12pt]{article}

\usepackage{multirow} 
\usepackage{graphicx} 

\usepackage{amsmath,bm}
\usepackage{amsthm} 
\usepackage{hyperref}
\usepackage{amssymb}
\usepackage{fullpage}
\usepackage{placeins}

\newtheorem{thm}{Theorem}[section]
\newtheorem{prop}{Proposition}[section]

\newtheorem{lem}{Lemma}[section]

\newtheorem{cor}{Corollary}[section]

\newcommand{\be}{\begin{equation}}
\newcommand{\ee}{\end{equation}}

\newcommand{\ur}{u_{r}}
\newcommand{\uthe}{u_{\theta}}

\newcommand{\dive}{\mbox{div}}

\newcommand{\pt}{\partial}

\newcommand{\ctthe}{\cot\theta}
\newcommand{\sthe}{\sin\theta}
\newcommand{\cthe}{\cos\theta}

\newcommand{\er}{e_r}
\newcommand{\ethe}{e_{\theta}}
\newcommand{\ephi}{e_{\phi}}

\renewenvironment{proof}{\noindent{\bf Proof.}}{\qed}

\begin{document}
\title{Homogeneous solutions of stationary Navier-Stokes equations with isolated singularities on the unit sphere. III. Two singularities}
\author{Li Li\footnote{Department of Mathematics, Harbin Institute of Technology, Harbin 150080, China. Email: lilihit@126.com}, 
	YanYan Li\footnote{Department of Mathematics, Rutgers University, 110 Frelinghuysen Road, Piscataway, NJ 08854, USA. Email: yyli@math.rutgers.edu}, 
	Xukai Yan\footnote{School of Mathematics, Georgia Institute of Technology, 686 Cherry St NW, Atlanta, GA 30313, USA. Email: xukai.yan@math.gatech.edu}}
\date{}
\maketitle

\emph{Dedicated to Luis Caffarelli on his 70th birthday, with admiration and friendship.}

\abstract{All $(-1)$-homogeneous axisymmetric no-swirl solutions
of incompressible stationary Navier-Stokes equations in three dimension
which are smooth on the unit sphere minus north and south poles
have been classified in our earlier work as a four dimensional surface 
with boundary.
In this paper, we establish near the no-swirl solution surface
existence, non-existence and uniqueness results on
$(-1)$-homogeneous axisymmetric solutions with nonzero swirl
which are smooth on the unit sphere minus north and south poles.}

\setcounter{section}{0}

\section{Introduction}\label{sec:intro}
Consider the incompressible stationary Navier-Stokes equations (NSE) in $\mathbb{R}^3$: 
\begin{equation}\label{NS}
\left\{
\begin{split}
	& -\triangle u + u\cdot \nabla u +\nabla p = 0, \\
	& \dive\textrm{ } u=0.
\end{split}
\right. 
\end{equation} 

The equations are invariant under the scaling $u(x)\to \lambda u(\lambda x)$ and $p(x)\to \lambda^2 p(\lambda x)$, $\lambda>0$. We study solutions which are invariant under the scaling. For such solutions $u$ is $(-1)$-homogeneous and $p$ is $(-2)$-homogeneous.  We call them $(-1)$-homogeneous solutions.

The NSE can be reformulated in spherical coordinates $(r,\theta,\phi)$. A vector field $u$ can be written as
\[
	u = u_r \er + u_\theta \ethe + u_\phi \ephi, 
\]
where 
\[
	\er = \left(
	\begin{matrix}
		\sthe\cos\phi \\
		\sthe\sin\phi \\
		\cthe
	\end{matrix} \right),  \hspace{0.5cm}
	\ethe = \left(
	\begin{matrix}
		\cthe\cos\phi  \\
		\cthe\sin\phi   \\
		-\sthe	
	\end{matrix} \right), \hspace{0.5cm}
	\ephi = \left(
	\begin{matrix}
		-\sin\phi \\  \cos\phi \\ 0
	\end{matrix} \right).  
\]
A vector field $u$ is called axisymmetric if $u_r$, $u_{\theta}$ and $u_{\phi}$ are independent of $\phi$, and is called {\it no-swirl} if $u_{\phi}=0$.  

Landau discovered in \cite{Landau} a three parameter family of explicit $(-1)$-homogeneous solutions of the stationary NSE \eqref{NS}, which are axisymmetric and with no swirl. These solutions are now called Landau solutions. Tian and Xin prove in \cite{TianXin} that all $(-1)$-homogeneous, axisymmetric nonzero solutions of \eqref{NS} in $C^{2}(\mathbb{R}^3\setminus\{0\})$ are Landau solutions. A classification of all $(-1)$-homogeneous solutions in $C^{2}(\mathbb{R}^3\setminus\{0\})$  was given by \v{S}ver\'{a}k in 2006:

\addtocounter{thm}{-1}
\renewcommand{\thethm}{A}%
\begin{thm}{\rm(\cite{Sverak})}
All (-1)-homogeneous nonzero solutions of \eqref{NS} in $C^2(\mathbb{R}^3\setminus\{0\})$ are Landau solutions.
\end{thm}
\renewcommand{\thethm}{\thesection.\arabic{thm}}%

We are interested in analyzing solutions which are smooth on $\mathbb{S}^2$ minus finite points.
In our first paper \cite{LLY1}, we have classified all axisymmtric no-swirl solutions with one singularity, as a two dimensional surface with boundary, and have proved the existence of a curve of axisymmetric solutions with nonzero swirl emanating from the interior and one part of the boundary of the surface of no-swirl solutions. We have also proved that there is no such curve of solutions for any point on the other part of the boundary. Uniqueness results of nonzero swirl solutions near the surface were also given in this paper.  In our second paper \cite{LLY2}, we classified  all $(-1)$-homogeneous, axisymmetric no-swirl solutions of \eqref{NS} which are smooth on $\mathbb{S}^2\setminus\{S,N\}$ as a four parameter family of solutions, where $S$ is the south pole and $N$ is the north pole. Our main result in this paper is to prove the existence of $(-1)$-homogeneous axisymmetric solutions with nonzero swirl emanating from the surface of no-swirl solutions, as well as nonexistence and uniqueness results.

%
%
%
%
%
%

For each $c_1\ge -1, c_2\ge -1$, let 
\[
    \bar{c}_3 (c_1,c_2) := -\frac{1}{2} \left( \sqrt{1+c_1} + \sqrt{1+c_2}  \right) \left( \sqrt{1+c_1} + \sqrt{1+c_2}  + 2 \right).
    \]
     Let  $c:=(c_1,c_2,c_3)$, define
\begin{equation*}
   J:=\{c\in \mathbb{R}^3 | c_1\geq -1, c_2\geq -1,c_3\ge \bar{c}_3(c_1,c_2)\}. 
\end{equation*}
In \cite{LLY2} we have proved that there exist $\gamma^-,\gamma^+\in C^0(J, \mathbb{R})$ such that  all $(-1)$-homogeneous axisymmetric no-swirl solutions of \eqref{NS} in $C^2(\mathbb{R}^3\setminus\{x_3-axis\})$ are a four parameter family $\{(u^{c,\gamma},p^{c,\gamma})\}$, where $c:=(c_1,c_2,c_3)$, $(c,\gamma)\in I$  and
\begin{equation*}
	I:= \{(c,\gamma)\in \mathbb{R}^4 \mid c_1\geq -1, c_2\geq -1, c_3\geq \bar{c}_3(c_1,c_2), \gamma^-(c) \leq \gamma\leq \gamma^+(c) \}. 
\end{equation*}

For an axisymmetric no-swirl solution $(u^{c,\gamma}, p^{c,\gamma})$ of \eqref{NS} in $C^{2}(\mathbb{S}^2\setminus\{N,S\})$, the linearized equation of \eqref{NS} at $(u^{c,\gamma}, p^{c,\gamma})$ is 
\begin{equation}\label{sec1:eq:LNS}
  \left\{
\begin{split}
	& -\triangle v  + u^{c,\gamma}\cdot \nabla v+ v\cdot \nabla u^{c,\gamma}+\nabla q = 0, \\
	& \dive\textrm{ } v=0.
\end{split}
\right.
\end{equation}
Define
\begin{equation}\label{eq_ab_1}
	a_{c,\gamma}(\theta)=-\int_{\frac{\pi}{2}}^{\theta} \left( 2\cot t+u_\theta^{c,\gamma} \right)dt, \quad     
	b_{c,\gamma}(\theta)=-\int_{\frac{\pi}{2}}^{\theta}u_\theta^{c,\gamma}dt,
\end{equation}
and

\begin{equation*}
  v_{c,\gamma}^{1}= -\frac{1}{\sin\theta}\int_{\frac{\pi}{2}}^{\theta}e^{-b_{c,\gamma}(t)}\sin t dt\cdot \vec{e}_{\phi},\quad v_{c,\gamma}^2=\frac{1}{\sin\theta}\vec{e}_{\phi}. 
\end{equation*}
Then, as explained towards the end of Section \ref{sec_2},  $\{v_{c,\gamma}^1, v_{c,\gamma}^2\}$ are linearly independent solutions of  \eqref{sec1:eq:LNS}, in spherical coordinates,  on $\mathbb{S}^2\setminus\{N,S\}$. 

 
 We define the following subsets of $J$ and $I$:
 \begin{equation}\label{eqJ_1}
   \begin{split}
      & J_1:=\{c\in J \mid c_1>-1,c_2>-1,c_3>\bar{c}_3\},\quad J_2:=\{c\in J \mid c_1=-1,c_2>-1,c_3>\bar{c}_3\},\\
      & J_3:=\{c\in J \mid c_1>-1,c_2=-1,c_3>\bar{c}_3\},\quad J_4:=\{c\in J \mid c_1=-1,c_2=-1,c_3>\bar{c}_3\},\\
      & J_5:=\{c\in J \mid c_1>-1,c_2>-1,c_3=\bar{c}_3\},\quad J_6:=\{c\in J \mid c_1=-1,c_2>-1,c_3=\bar{c}_3\},\\
      & J_7:=\{c\in J \mid c_1>-1,c_2=-1,c_3=\bar{c}_3\},\quad J_8:=\{c\in J \mid c_1=-1,c_2=-1,c_3=\bar{c}_3\},
   \end{split}
\end{equation}
and for $1\le k\le 4$,
	\begin{equation}\label{eq1_3}
	\begin{split}
	& I_{k,1}:=\{(c,\gamma)\in I \mid c\in J_k,\gamma^-(c)<\gamma<\gamma^+(c)\}, \\
	& I_{k,2}:=\{(c,\gamma)\in I \mid c\in J_k, c_1<-\frac{3}{4}, \gamma=\gamma^+(c)\},\\
	& I_{k,3}:=\{(c,\gamma)\in I \mid c\in J_k, c_2<-\frac{3}{4}, \gamma=\gamma^-(c)\}, 
  \end{split}
\end{equation}
and for $5\le k\le 8$, let $I_{k,l}:=\{(c,\gamma)\in J_k\times\{\gamma^+(c)\} \mid c_1,c_2<-\frac{3}{4}\}$, $l=1,2,3$.

Moreover, let 
\begin{equation*}
   \hat{I}:=\{(c,\gamma)\in I\setminus\cup_{1\le k\le 8, 1\le l\le 3} I_{k,l} | c_1> -\frac{3}{4} \textrm{ or } c_2> -\frac{3}{4}\}.
\end{equation*}

\begin{thm}\label{thm1}
	Let $K$ be a compact subset of one of the sets $I_{k,l}$, $1\le k\le 8$, $1\le l\le 3$. Then there exist  $\delta=\delta(K)>0$, and $(u,p)\in C^{\infty}(K\times B_{\delta}(0)\times (\mathbb{S}^2\setminus\{N,S\}))$ such that for every $(c,\gamma, \beta)\in K\times  B_{\delta}(0)$, $\beta=(\beta_1,\beta_2)$, $(u,p)(c,\gamma,\beta; \cdot)\in C^{\infty}(\mathbb{S}^2\setminus\{N,S\})$ satisfies \eqref{NS} in $\mathbb{R}^3\setminus\{(0,0,x_3)|x_3\in\mathbb{R}\}$, with nonzero swirl if $\beta\ne 0$, and $\| \sin\theta \left( u(c,\gamma,\beta) - u^{c,\gamma} \right) \|_{L^{\infty}(\mathbb{S}^2\setminus\{N,S\})}\to 0$ as $\beta\to 0$. Moreover, $ \frac{\partial }{\partial \beta_i}u(c,\gamma,\beta)|_{\beta=0}=v_{c,\gamma}^i$, $i=1,2$. 
  
	On the other hand, for $(c,\gamma)\in \hat{I}$, if there exist a sequence of solutions $\{u^i\}$ of \eqref{NS} in $C^{\infty}(\mathbb{S}^2\setminus\{N,S\})$, such that $\| \sin\theta (u^i-u^{c,\gamma})\|_{L^{\infty}(\mathbb{S}^2\setminus\{N,S\})}\to 0$ as $i\to \infty$, then for sufficiently large $i$, $u^i=u^{c_i,\gamma_i}+\frac{C_i}{\sin\theta}\vec{e}_{\phi}$ for some constants $c_i,\gamma_i, C_i$ satisfying $(c_i,\gamma_i)\to (c,\gamma)$ and $C_i\to 0$ as $i\to \infty$. 
\end{thm}

In the above theorem, $(u,p)\in C^{\infty}(\mathbb{S}^2\setminus\{N,S\})$ is understood to have been extended to $\mathbb{R}^3\setminus\{(0,0,x_3)|x_3 \in\mathbb{R}\}$ so that $u$ is $(-1)$-homogeneous and $p$ is $(-2)$-homogeneous. We use this convention throughout the paper unless otherwise stated.\\

%


In Section 2, we will state some properties of $(-1)$-homogeneous axisymmetric no-swirl solutions in $C^2(\mathbb{S}^2\setminus\{S,N\})$ that we have obtained in \cite{LLY2} and will use them in later sections. We will also introduce some notations in that section.  We will then prove the existence, nonexistence and uniqueness results of $(-1)$-homogeneous axisymmetric solutions with nonzero swirl in $C^2(\mathbb{S}^2\setminus\{S,N\})$ in three different cases in Section \ref{sec_3}, \ref{sec33:sec} and \ref{sec_5}. The proof of Theorem \ref{thm1} is given at the end of Section \ref{sec_5}, utilizing the results in Section \ref{sec_3}-\ref{sec_5}.


\bigskip

\noindent
{\bf Acknowledgment}. 
The work of the first named author is partially supported by NSFC grants No. 11871177. The work of the second named author is partially supported by NSF grants DMS-1501004.   The work of the third named author is partially supported by AMS-Simons Travel Grant and AWM-NSF Travel Grant.

\section{Preliminary}\label{sec_2}

An $(-1)$-homogeneous axisymmetric vector field $u$ is divergence free if and only if 
\begin{equation}\label{eq_divefree}
    \ur =-\frac{d \uthe}{d \theta} - \ctthe \uthe .
    \end{equation}

    Define new unknown functions and a different independent variable: 
\begin{equation*}
	x:=\cthe, \quad U_\theta := u_\theta \sthe, \quad U_\phi:= u_\phi \sthe, 
\end{equation*}
and let $U:=(U_{\theta},U_{\phi})$.

All axisymmetric no-swirl solutions on $\mathbb{S}^2\setminus\{N,S\}$ are given by the family $U^{c,\gamma}:=(U^{c,\gamma}_{\theta}, 0)$ with $(c,\gamma)\in I$, where $U^{c,\gamma}_{\theta}$ are solutions given by Theorem 1.2 in \cite{LLY2}. 

Denote $\bar{U}:=U^{c,\gamma}$ for convenience.

As explained in \cite{LLY1}, $(u,p)$ is a $(-1)$-homogeneous axisymmetric solution of \eqref{NS} if and only if $u_r$ is given by \eqref{eq_divefree}, $p$ is given by 
\begin{equation}\label{eq1_2}
      p=-\frac{1}{2}\left(\frac{d^2 \ur}{d\theta^2} + (\ctthe - \uthe) \frac{d \ur}{d\theta} + \ur^2 +\uthe^2+u^2_{\phi}\right),
\end{equation}
and $U:=(U_{\theta},U_{\phi})$ satisfies 

\begin{equation}\label{sec31:eq:NSE}
\left\{
\begin{split}
	& (1-x^2) U_\theta' + 2x U_\theta +\frac{1}{2} U_\theta^2 + \int_{0}^{x} \int_{0}^{l} \int_{0}^{t} \frac{2 U_\phi(s) U_\phi'(s)}{1-s^2} ds dt dl 
	 = P_{\hat{c}}(x),  \\
	& (1-x^2) U_\phi'' + U_\theta U_\phi' = 0. 
\end{split}
\right. 
\end{equation}
where 
\[
   P_{\hat{c}}= \hat{c}_1(1-x)+ \hat{c}_2(1+x)+ \hat{c}_3 (1-x^2),
\]
and $\hat{c}_1, \hat{c}_2, \hat{c}_3$ are constants. 

Moreover, for each $(c,\gamma)\in I$,  $\bar{U}=U^{c,\gamma}$ satisfies $\bar{U}_\phi \equiv 0$,
\begin{equation}\label{eq:UthP}
 (1-x^2)\bar{U}_{\theta}'+2x\bar{U}_{\theta}+\frac{1}{2}\bar{U}_{\theta}^2=c_1(1-x)+c_2(1+x)+c_3(1-x^2),
\end{equation}
and $\bar{U}_{\theta}(0)=\gamma$.

We first introduce the implicit function theorem (IFT) which we use: 
\addtocounter{thm}{-1}
\renewcommand{\thethm}{B}%
\begin{thm}{\bf (Implicit Function Theorem)}{\rm(\cite{Nirenberg})}\label{thm:IFT}
	Let $\mathbf{X},\mathbf{Y},\mathbf{Z}$ be Banach spaces and $f$ a continuous mapping of an open set $U\subset \mathbf{X}\times \mathbf{Y}\to \mathbf{Z}$. Assume that $f$ has a Fr\'{e}chet derivative with respect to $x$, $f_{x}(x,y)$ which is continuous in $U$. Let $(x_0,y_0)\in U$ and $f(x_0,y_0)=0$. If $A=f_{x}(x_0,y_0)$ is an isomorphism of $\mathbf{X}$ onto $\mathbf{Z}$ then
	
	(1) There is a ball $\{y:\|y-y_0\|<r\}=B_r(y_0)$ and a unique continuous map $u:B_r(y_0)\to \mathbf{X}$ such that $u(y_0)=x_0$ and $f(u(y),y)\equiv 0$.
	
	(2) If $f$ is of class $C^1$ then $u(y)$ is of class $C^1$ and $u_y(y)=-(f_{x}(u(y),y))^{-1}\circ f_{y}(u(y),y)$.
	
	(3) $u_{y}(y)$ belongs to $C^k$ if $f$ is in $C^k$, $k>1$. 
\end{thm}
\renewcommand{\thethm}{\thesection.\arabic{thm}}%

By the asymptotic behavior studies in \cite{LLY1}, if $\bar{U}_{\theta}(-1)=2$, $\eta_1:=\lim_{x\to -1}(\bar{U}_{\theta}-2)\ln(1+x)$ exists and $\eta_1=0$ or $4$. If $\bar{U}_{\theta}(1)=-2$, $\eta_2:=\lim_{x\to 1}(\bar{U}_{\theta}+2)\ln(1-x)$ exists and $\eta_2=0$ or $-4$.

To prove the existence of axisymmetric, with swirl solutions near $\bar{U}$ using the implicit function theorem, we construct function spaces according to the singular behaviors of $\bar{U}_{\theta}$ near the poles, which are determined by the values of $\bar{U}_{\theta}(-1)$, $\bar{U}_{\theta}(1)$, $\eta_1$ and $\eta_2$. 

Our proof of existence will be carried out in following separate cases:

Case 1:  $\big( \bar{U}_{\theta}(-1)<3$, $\bar{U}_{\theta}(-1)\ne 2$ or $\bar{U}_{\theta}(-1)=2$ with $\eta_1=0\big)$, AND $\big( \bar{U}_{\theta}(1)>-3$, $\bar{U}_{\theta}(1)\ne -2$ or $\bar{U}_{\theta}(1)=-2$ with $\eta_2=0\big)$.

Case 2: $\big( \bar{U}_{\theta}(-1)=2$ with $\eta_1=4$, and $(\bar{U}_{\theta}(1)>-3$, $\bar{U}_{\theta}(1)\ne -2$ or $\bar{U}_{\theta}(1)=-2$ with $\eta_2=0) \big)$ OR $\big(\bar{U}_{\theta}(1)=-2$ with $\eta_2=-4$, and $(\bar{U}_{\theta}(-1)<3$, $\bar{U}_{\theta}(-1)\ne 2$ or $\bar{U}_{\theta}(-1)=2$ with $\eta_1=0 ) \big)$.

Case 3: $\bar{U}_{\theta}(-1)=2$ with $\eta_1=4$, and $\bar{U}_{\theta}(1)=-2$ with $\eta_2=-4$.

Case 4: $\bar{U}_{\theta}(-1)\ge 3$ or $\bar{U}_{\theta}(1)\le -3$.\\


The axisymmetric no-swirl solution $\{U_{\theta}^{c,\gamma}\}$ of \eqref{eq:UthP} satisfies the following properties: 
\begin{prop}{\rm(Theorem 1.3 in \cite{LLY1} and Theorem 1.3 in \cite{LLY2})}\label{prop2_1}
Suppose $(c,\gamma)\in I$, then

	
	
	(i) For any $(c,\gamma)\in I$, $U_\theta^{c,\gamma}(\pm 1)$ both exist and are finite. Moreover,	
	\begin{equation*}
	\begin{split}
		& U_\theta^{c,\gamma}(-1) := \left\{
		\begin{array}{ll}
			2+2\sqrt{1+c_1}, &  \mbox{if } \gamma=\gamma^+(c), \\
			2-2\sqrt{1+c_1}, &  \mbox{otherwise},
		\end{array}
		\right. \\ 
		&U_\theta^{c,\gamma}(1) := \left\{
		\begin{array}{ll}
			-2-2\sqrt{1+c_2}, &  \mbox{if } \gamma=\gamma^-(c), \\
			2+2\sqrt{1+c_2}, &  \mbox{otherwise}. 
		\end{array}
		\right. 
		\end{split}
	\end{equation*}
	
	(ii) If $c_1=-1$, then $\eta_1:=\lim_{x\to -1}(U^{c,\gamma}_{\theta}-2)\ln(1+x)$ exists. Moreover, $\eta_1=0$ if $\gamma=\gamma^+(c)$ and $\eta_1=4$ if $\gamma<\gamma^+(c)$.
	
	If $c_2=-1$, then $\eta_2:=\lim_{x\to 1}(U^{c,\gamma}_{\theta}+2)\ln(1-x)$ exists. Moreover, $\eta_2=0$ if $\gamma=\gamma^-(c)$ and $\eta_2=-4$ if $\gamma>\gamma^-(c)$.
\end{prop}

Let $I_{k,l}$ be the sets defined in Section \ref{sec:intro} by \eqref{eqJ_1} and \eqref{eq1_3}. By Proposition \ref{prop2_1} we have the following relations:

$U^{c,\gamma}$ satisfies Case 1 if and only if $(c,\gamma)\in I_{k,l}$ with $(k,l)\in A_1:=\{(k,l)\in \mathbb{Z}^2| k=1 \textrm{ or } 5\le k\le 8, 1\le l\le 3\}\cup\{(2,2), (3,3)\}$.

$U^{c,\gamma}$ satisfies Case 2 if and only if  $(c,\gamma)\in I_{k,l}$ with $(k,l)\in A_2:=\{(2,1), (2,3), (4,3)\}$ or $A_3:=\{(3,1), (3,2), (4,2)\}$.

$U^{c,\gamma}$ satisfies Case 3 if and only if $(c,\gamma)\in I_{k,l}$ with $(k,l)=(4,1)$.

\medskip

According to Theorem 1.1 in \cite{LLY2}, if $c_3=\bar{c}_3$, i.e. $(c,\gamma)\in I_{k,l}$ for $5\le k\le 8$,  then there is only one solution of \eqref{eq:UthP}, which is 
\[
   U_{\theta}^{c,\gamma}= (1+\sqrt{1+c_1})(1-x)+(-1-\sqrt{1+c_2})(1+x).
\]
By this fact and Theorem 1.5 in \cite{LLY2}, we have 


\begin{prop}\label{propA_1}{\rm(Theorem 1.1 and Theorem 1.5 in \cite{LLY2})}
Let $K$ be a compact set contained in one of $I_{k,l}$, $1\le k\le 8$, $l=1,2,3$. Then the solution $U^{c,\gamma}_{\theta}$ of equation \eqref{eq:UthP} is $C^{\infty}(K\times(-1,1))$. Moreover,	

	(1) If $k=1$, or $5\le k\le 8$ and $l=1,2,3$, or $(k,l)=(2,2)$ or $(3,3)$, then for $-1<x<1$
	\begin{equation*}
		|\pt_c^{\alpha} \pt_\gamma^j \bar{U}_\theta|\le C(m,K), \quad \mbox{for any } 0\le |\alpha|+j\leq m,
	\end{equation*}
	where $j=0$ if $l=2,3$, $\alpha_3=0$ if $5\le k\le 8$; $\alpha_1=0$ if $k=2,4$; and $\alpha_2=0$ if $k=3,4$. 
	
	(2) If $(k,l)=(2,1)$, $(2,3)$ or $(4,3)$, then for $-1<x<1$
	\begin{equation*}
		\left(\ln \frac{1+x}{3}\right)^2|\pt_c^{\alpha} \pt_\gamma^j \bar{U}_\theta|\le C(m,K), \quad \mbox{for any } 1\le |\alpha|+j\leq m, \alpha_1=0,  
	\end{equation*}
	where $j=0$ if $l=3$, and $\alpha_2=0$ if $k=4$. 

	(3) If $(k,l)=(3,1)$, $(3,2)$ or $(4,2)$, then for $-1<x<1$
	\begin{equation*}
		\left(\ln \frac{1-x}{3} \right)^2 |\pt_c^{\alpha} \pt_\gamma^j \bar{U}_\theta|\le C(m,K), \quad \mbox{for any } 1\le |\alpha|+j\leq m, \alpha_2=0, 
	\end{equation*}
	where $j=0$ if $l=2$, and $\alpha_1=0$ if $k=4$. 

	(4) If $(k,l)=(4,1)$, then for $-1<x<1$ and for any $1\le |\alpha|+j\leq m, \alpha_1=\alpha_2=0$,
	\begin{equation*}
		\left(\ln \frac{1+x}{3}\right)^2\left(\ln \frac{1-x}{3}\right)^2  |\pt_c^{\alpha} \pt_\gamma^j \bar{U}_\theta|\le C(m,K). 
	\end{equation*}
\end{prop}


We will work with $\tilde{U}:=U-\bar{U}$, a calculation gives
\[
	(1-x^2)U'_{\theta}+2xU_{\theta}+\frac{1}{2}U^2_{\theta} - c_1(1-x) - c_2 (1+x) - c_3 (1-x^2) = (1-x^2)\tilde{U}_{\theta}'+(2x+\bar{U}_{\theta})\tilde{U}_{\theta}+\frac{1}{2}\tilde{U}_{\theta}^2,
\]
where $\tilde{U}_{\phi}=U_{\phi}$. 
Denote
\begin{equation}\label{sec31:eq:psi}
	\psi[\tilde{U}_{\phi}, \tilde{V}_{\phi}](x):=\int_{0}^{x} \int_{0}^{l} \int_{0}^{t} \frac{2\tilde{U}_{\phi}\tilde{V}'_{\phi}}{1-s^2} ds dt dl, 
\end{equation}
and
\begin{equation*}
	\varphi_{c,\gamma}[\tilde{U}_{\theta}](x) := (1-x^2) \tilde{U}_{\theta}'+(2x+\bar{U}_{\theta})\tilde{U}_{\theta}+\frac{1}{2}\tilde{U}_{\theta}^2.
\end{equation*}
For convenience write $\psi[\tilde{U}_{\phi}](x):=\psi[\tilde{U}_{\phi}, \tilde{U}_{\phi}](x)$.
Define a map $G$ on $(c,\gamma,\tilde{U})$ by
\begin{equation}\label{sec31:eq:G}
	G(c,\gamma,\tilde{U}) =
	\left(
	\begin{matrix}
		(1-x^2)\tilde{U}'_{\theta}+(2x+\bar{U}_{\theta})\tilde{U}_{\theta}+\frac{1}{2}\tilde{U}^2_{\theta} + \psi[\tilde{U}_{\phi}](x) - \tilde{P}(x)\\
		(1-x^2)\tilde{U}''_{\phi}+(\tilde{U}_{\theta}+\bar{U}_{\theta})\tilde{U}'_{\phi}
	\end{matrix}
	\right),
\end{equation}
where 
$$
	\tilde{P}(x) = \frac{1}{2}\psi[\tilde{U}_{\phi}](-1)(1-x) + \frac{1}{2}\psi[\tilde{U}_{\phi}](1)(1+x) 
	- \frac{1}{2} ( \varphi_{c,\gamma} [ \tilde{U}_\theta ] )''(0) \cdot (1-x^2). 
$$
In the above expression $\psi[\tilde{U}_{\phi}](-1)$ and $\psi[\tilde{U}_{\phi}](1)$  are finite if $|\tilde{U}_{\phi}\tilde{U}'_{\phi}|\le C(1-x^2)^{-1-\epsilon}$ for some $\epsilon<1$, which will be satisfied in later applications with $\tilde{U}$ in  the spaces we constructed.

If $\tilde{U}$ satisfies $G(c,\gamma,\tilde{U})=0$, then $U=\tilde{U}+\bar{U}$ gives a solution of (\ref{sec31:eq:NSE}) with 
$$
	\hat{c}_1 = c_1 - \frac{1}{2}\psi[\tilde{U}_{\phi}](-1), \quad \hat{c}_2 = c_2 - \frac{1}{2}\psi[\tilde{U}_{\phi}](1),
$$
$$
	\hat{c}_3 = c_3 + \frac{1}{2} ( \varphi_{c,\gamma} [ \tilde{U}_\theta ] )''(0), 
$$
satisfying $U_{\theta}(-1)=\bar{U}_{\theta}(-1)$, $U_{\theta}(1)=\bar{U}_{\theta}(1)$. 

Denote 
\begin{equation}\label{sec31:eq:l}
	l_{c,\gamma}[\tilde{U}_{\theta}](x):=(1-x^2)\tilde{U}'_{\theta}(x)+(2x+\bar{U}_{\theta})\tilde{U}_{\theta}(x).
\end{equation}

Let $A$ and $Q$ be maps of  the form
\begin{equation}\label{sec31:eq:A}
	A(c,\gamma,\tilde{U}) =\left(
	\begin{matrix}
		A_\theta\\
		A_\phi
	\end{matrix}
	\right) 	 :=
	\left(
	\begin{matrix}
		l_{c,\gamma}[\tilde{U}_{\theta}](x) + \frac{1}{2} ( l_{c,\gamma} [ \tilde{U}_\theta ] )''(0) \cdot (1-x^2)\\
		(1-x^2)\tilde{U}''_\phi+\bar{U}_{\theta}\tilde{U}'_{\phi}
	\end{matrix}
	\right) ,
\end{equation}
and 
\begingroup
\begin{equation}\label{sec31:eq:Q}
\begin{split}
	& Q(\tilde{U},\tilde{V}) = \left(
	\begin{matrix}
		Q_\theta\\
		Q_\phi
	\end{matrix}
	\right) 	 \\
	& := 
	\left(
	\begin{matrix}
		\frac{1}{2}\tilde{U}_{\theta}\tilde{V}_{\theta} 
		+ \psi[\tilde{U}_{\phi}, \tilde{V}_{\phi}](x) 
		- \frac{1-x}{2} \psi[\tilde{U}_{\phi}, \tilde{V}_{\phi}](-1)
		- \frac{1+x}{2} \psi[\tilde{U}_{\phi}, \tilde{V}_{\phi}](1) 
		+ \frac{1}{4} (\tilde{U}_{\theta} \tilde{V}_{\theta})''(0)(1-x^2) \\
		\tilde{U}_{\theta}\tilde{V}'_{\phi} 
	\end{matrix}
	\right) .	
\end{split}
\end{equation}
\endgroup
Then $G(c,\gamma,\tilde{U})=A(c,\gamma,\tilde{U})+Q(\tilde{U}, \tilde{U})$. 

By computation, the linearized operator of $G$ with respect to $\tilde{U}$ at $(c,\gamma,\tilde{U})$ is given by
\begin{equation}\label{sec31:eq:Linear}
\begin{split}
	& L_{\tilde{U}}^{c,\gamma} \tilde{V} 
	:=A(c,\gamma,\tilde{V}) \\
	& +
	\left(
	\begin{matrix}
		\tilde{U}_{\theta} \tilde{V}_{\theta}
		+ \Psi_{\tilde{U}_{\phi}}[\tilde{V}_{\phi}](x) 
		- \frac{1-x}{2}\Psi_{\tilde{U}_{\phi}}[\tilde{V}_{\phi}](-1)
		- \frac{1+x}{2}\Psi_{\tilde{U}_{\phi}}[\tilde{V}_{\phi}](1)
		+ \frac{1}{2}(\tilde{U}_{\theta} \tilde{V}_{\theta})''(0)(1-x^2) \\
		\tilde{U}_{\theta} \tilde{V}'_{\phi}+ \tilde{V}_{\theta} \tilde{U}'_{\phi}
	\end{matrix}
	\right), 
\end{split}
\end{equation}
where 
\[
	\Psi_{\tilde{U}_{\phi}}[\tilde{V}_{\phi}](x):=\int_0^x \int_0^l \int_0^t \frac{\displaystyle 2(\tilde{U}_{\phi}(s) \tilde{V}'_{\phi}(s)+ \tilde{V}_{\phi}(s)\tilde{U}'_{\phi}(s))}{\displaystyle 1-s^2} ds dt dl. 
\]

In particular, at $\tilde{U}=0$, the linearized operator of $G$ with respect to $\tilde{U}$ is
\begin{equation}\label{sec31:eq:Linear0}
	L^{c,\gamma}_{0} \tilde{V}=A(c,\gamma,\tilde{V})=
		\left(
		\begin{matrix}
			l_{c,\gamma}[\tilde{V}_{\theta}](x)+\frac{1}{2}(l_{c,\gamma}[\tilde{V}_{\theta}])''(0)(1-x^2) \\
			(1-x^2) \tilde{V}_\phi'' + \bar{U}_\theta \tilde{V}_\phi'
		\end{matrix}
		\right). 
\end{equation}
After a change of variable $s=\cos t$,  we derive from (\ref{eq_ab_1}) that
\begin{equation}\label{sec31:eq:ab}
	a_{c,\gamma}(x)=\int_{0}^{x}\frac{\displaystyle 2s+\bar{U}_{\theta}}{\displaystyle 1-s^2}ds,\quad
	b_{c,\gamma}(x)=\int_{0}^{x}\frac{\displaystyle \bar{U}_{\theta}}{\displaystyle 1-s^2}ds,\quad -1<x<1.
\end{equation}

From the discussion in \cite{LLY2}, for all $(c,\gamma)\in I$, $\bar{U}_{\theta}$ is smooth in $(-1,1)$. So $a_{c,\gamma}, b_{c,\gamma}\in C^{\infty}(-1,1)$. 

By observation $a_{c,\gamma}(x)=-\ln(1-x^2)+b_{c,\gamma}(x)$. A calculation gives
\begin{equation}\label{sec31:eq:diff:a}
	a_{c,\gamma}'(x)=\frac{2x+\bar{U}_{\theta}(x)}{1-x^2}, \quad a_{c,\gamma}''(x)=\frac{2+\bar{U}'_{\theta}(x)}{1-x^2}+\frac{4x^2+2x\bar{U}_{\theta}(x)}{(1-x^2)^2}.
\end{equation}

Next, formally define the maps $W^{c,\gamma,i}_{\theta}$, $i=1,2a,2b,3$, on $\xi=(\xi_{\theta},\xi_{\phi})$ by 
\begin{equation}\label{sec31:eq:Wthei}
\begin{split}
	& W^{c,\gamma,1}_{\theta}(\xi)(x):=e^{-a_{c,\gamma}(x)}\int_{0}^{x}e^{a_{c,\gamma}(s)}\frac{\xi_{\theta}(s)}{1-s^2}ds,\\
	& W^{c,\gamma,2a}_{\theta}(\xi)(x):=e^{-a_{c,\gamma}(x)}\int_{-1}^{x}e^{a_{c,\gamma}(s)}\frac{\xi_{\theta}(s)}{1-s^2}ds,\\
	& W^{c,\gamma,2b}_{\theta}(\xi)(x):=e^{-a_{c,\gamma}(x)}\int_{1}^{x}e^{a_{c,\gamma}(s)}\frac{\xi_{\theta}(s)}{1-s^2}ds,\\
	& W^{c,\gamma,3}_{\theta}(\xi)(x):=e^{-a_{c,\gamma}(x)}\int_{-1}^{x}e^{a_{c,\gamma}(s)}\left(\frac{\xi_{\theta}(s)}{1-s^2}-C^{c,\gamma}_W(\xi)\right)ds,
\end{split}
\end{equation}
where 
\begin{equation}\label{eq:CW}
	C^{c,\gamma}_W(\xi)=\frac{1}{\int_{-1}^{1}e^{a_{c,\gamma}(s)}ds}\int_{-1}^{1}e^{a_{c,\gamma}(s)}\frac{\xi_{\theta}(s)}{1-s^2}ds.
\end{equation}

Define a map $W^{c,\gamma}_{\phi}$ on $\xi$ by
\begin{equation}\label{sec31:eq:Wphi}
	W^{c,\gamma}_\phi(\xi)(x) :=\int_{0}^{x}e^{-b_{c,\gamma}(t)}\int_{0}^{t}e^{b_{c,\gamma}(s)}\frac{\xi_{\phi}(s)}{1-s^2}dsdt.
\end{equation}

A calculation gives
\begin{equation}\label{sec31:eq:diff:Wthe}
\begin{split}
	&(W^{c,\gamma,i}_{\theta}(\xi))'(x)=-a_{c,\gamma}'(x)W^{c,\gamma,i}_{\theta}(x)+\frac{\xi_{\theta}(x)}{1-x^2}, \quad i=1,2a,2b, \\
	& (W^{c,\gamma,3}_{\theta}(\xi))'(x)=-a_{c,\gamma}'(x)W^{c,\gamma,3}_{\theta}(x)+\frac{\xi_{\theta}(x)}{1-x^2}-C^{c,\gamma}_W(\xi),
	\end{split}
\end{equation}
\begin{equation}\label{sec31:eq:diff1:Wphi}
	(W^{c,\gamma}_{\phi}(\xi))'(x)= e^{-b_{c,\gamma}(x)}\int_0^x e^{b_{c,\gamma}(s)}\frac{\xi_{\phi}(s)}{1-s^2} ds, 
\end{equation}
\begin{equation}\label{sec31:eq:diff2:Wphi}
	(W^{c,\gamma}_{\phi}(\xi))''(x)=  - b_{c,\gamma}'(x) (W^{c,\gamma}_{\phi}(\xi))'(x) + \frac{\xi_{\phi}(x)}{1-x^2}. 
\end{equation}
We will prove in the following subsections that $W^{c,\gamma} = (W_\theta^{c,\gamma},W_\phi^{c,\gamma})$ is, roughly speaking, a right inverse of $L_0^{c,\gamma}$.

Consider the following system of ordinary differential equations in $(-1,1)$:
\[
	\left\{
	\begin{matrix}
		(1-x^2) V_\theta' + (2x  + \bar{U}_\theta ) V_\theta + \frac{1}{2} ( l_{c,\gamma} [\tilde{V}_\theta] )'' (0) (1-x^2) = 0, \\
		(1-x^2) V_\phi'' + \bar{U}_\theta V_\phi' =0.
	\end{matrix}
	\right.
\]
All solutions $V\in C^1((-1,1),\mathbb{R}^2)$ are given by
\begin{equation}\label{sec31:eq:ker}
	V = d_1 V_{c,\gamma}^1 + d_2 V_{c,\gamma}^2 + d_3 V_{c,\gamma}^3 + d_4 V_{c,\gamma}^4
\end{equation}
where $d_1, d_2, d_3, d_4\in\mathbb{R}$, and 
\begin{equation}\label{eq_basis}
\begin{aligned}
	& V_{c,\gamma}^{1} := \left(
	\begin{matrix}
		e^{-a_{c,\gamma}(x)}  \\  0
	\end{matrix}\right),
	\quad 
	&& V_{c,\gamma}^{2} := \left(
	\begin{matrix}
		e^{-a_{c,\gamma}(x)} \int_0^x e^{a_{c,\gamma}(s)} ds  \\  0
	\end{matrix}\right),  
	\\
	& V_{c,\gamma}^{3} := \left(
	\begin{matrix}
		0  \\  \int_0^x e^{-b_{c,\gamma}(t)}dt
	\end{matrix}\right),
	\quad 
	&&V_{c,\gamma}^{4} := \left(
	\begin{matrix}
		0  \\  1
	\end{matrix}\right).    
	\end{aligned}
\end{equation}
Moreover, denote  
\begin{equation}\label{sec32:eq:V2a}
	V_{c,\gamma}^{2a} := \left(
	\begin{matrix}
		e^{-a_{c,\gamma}(x)} \int_{-1}^x e^{a_{c,\gamma}(s)} ds   \\  0
	\end{matrix}\right), 
\end{equation}
and 
\begin{equation}\label{sec32:eq:V2b}
	V_{c,\gamma}^{2b} := \left(
	\begin{matrix}
		e^{-a_{c,\gamma}(x)} \int_{1}^x e^{a_{c,\gamma}(s)} ds  \\  0
	\end{matrix}\right) . 
\end{equation}
It can be seen that $V_{c,\gamma}^{2a} = V_{c,\gamma}^{2} + V_{c,\gamma}^{1} \int_{-1}^0 e^{a_{c,\gamma}(s)} ds$, and $V_{c,\gamma}^{2b}= V_{c,\gamma}^{2} - V_{c,\gamma}^{1} \int_0^{1} e^{a_{c,\gamma}(s)} ds $.
Next, introduce the linear functionals $l_i$, $1\le i\le 4$ and $l_{2a}, l_{2b}$ acting on vector-valued functions $V(x)=(V_{\theta}(x), V_{\phi}(x))$ by
\begin{equation}\label{sec31:eq:fcnal:l}
	l_1(V) :=  V_\theta(0), \quad l_2(V) :=V_\theta'(0),\quad l_3(V)=V'_{\phi}(0), \quad l_4(V)=V_{\phi}(0).
\end{equation}
and 
\begin{equation}\label{eq_l2ab}
  l_{2a}(V):=V_{\theta}'(-1),\quad l_{2b}(V):=V_{\theta}'(1).
\end{equation}

By computation it can be checked that 
\[
	(l_i(V_{c,\gamma}^j)) = 
	\begin{pmatrix}
		1 & 0 & 0 & 0 \\
		l_2(V_{c,\gamma}^1) & 1 & 0 & 0 \\
		0 & 0 & 1 & 0 \\
		0 & 0 & 0 & 1 
	\end{pmatrix}, \qquad 1\le i,j\le 4. 
\]
So $(l_i(V_{c,\gamma}^j))$ is an invertible matrix. Similar results hold if we replace $l_2$ by $l_{2a}$ or $l_{2b}$, replace $V_{c,\gamma}^2$ by $V_{c,\gamma}^{2a}$ or $V_{c,\gamma}^{2b}$, respectively.

\section{Existence of axisymmetric, with swirl solutions around \texorpdfstring{$\bm{U^{c,\gamma}}$}{}, when \texorpdfstring{$\bm{(c,\gamma)\in I_{k,l}}$}{} with \texorpdfstring{$\bm{(k,l)\in A_1}$}{} }\label{sec_3}



Denote $\bar{U}_{\theta}=U^{c,\gamma}$. By the assumption of $(c,\gamma)$ in this case, we have $\big( \bar{U}_{\theta}(-1)<3$, $\bar{U}_{\theta}(-1)\not=2$ or $\bar{U}_{\theta}(-1)=2$ with $\eta_1=0\Big)$ and $\big(\bar{U}_{\theta}(1)> -3$, $\bar{U}_{\theta}(1)\not= -2$ or $\bar{U}_{\theta}(1)=-2$ with $\eta_2=0\Big)$. Choose $0<\epsilon<\frac{1}{2}$ satisfying
\[
\epsilon>\left\{
   \begin{split}
    &  \frac{\bar{U}_{\theta}(-1)}{4}, \quad \textrm{ if }\bar{U}_{\theta}(-1)<2,\\
    & \frac{\bar{U}_{\theta}(-1)}{2}-1, \quad  \textrm{ if }\bar{U}_{\theta}(-1)\ge 2,
   \end{split}
   \right.
\textrm{ and } \quad 
   \epsilon>\left\{
   \begin{split}
    &  -\frac{\bar{U}_{\theta}(1)}{4}, \quad  \textrm{ if }\bar{U}_{\theta}(1)>-2,\\
    & -\frac{\bar{U}_{\theta}(1)}{2}-1, \quad  \textrm{ if }\bar{U}_{\theta}(1)\le -2.
   \end{split}
   \right.
\]
For this fixed $\epsilon$, define
\[
	\begin{split}
	\mathbf{M}_1 = & \mathbf{M}_1(\epsilon)\\
		:= & \Big\{ \tilde{U}_{\theta}\in C^3\left(-1/2, 1/2 \right)\cap C^1(-1,1)\cap C[-1,1] \mid \tilde{U}_\theta(-1) = \tilde{U}_\theta(1) = 0, \\
		 & \| (1-x^2)^{-1+2\epsilon} \tilde{U}_{\theta} \|_{L^\infty(-1,1)} < \infty, 
		 \| (1-x^2)^{2\epsilon} \tilde{U}'_{\theta} \|_{L^\infty(-1,1)}<\infty, \\
		 & \|  \tilde{U}''_{\theta} \|_{L^\infty(-\frac{1}{2}, \frac{1}{2})}<\infty, \|  \tilde{U}'''_{\theta} \|_{L^\infty(-\frac{1}{2}, \frac{1}{2})}<\infty \Big\}, \\
	\mathbf{M}_2=& \mathbf{M}_2(\epsilon)\\
		:= & \Big\{ \tilde{U}_{\phi} \in C^2(-1,1) \mid 
		\| (1-x^2)^{\epsilon}\tilde{U}_{\phi} \|_{L^\infty(-1,1)}<\infty, \\
		& \|(1-x^2)^{1+\epsilon}\tilde{U}'_{\phi} \|_{L^\infty(-1,1)}<\infty, \|(1 - x^2)^{2+\epsilon} \tilde{U}''_{\phi} \|_{L^\infty(-1,1)}<\infty \Big\}, 
	\end{split}
\]
with the following norms accordingly
\[
	\begin{split}
	& \| \tilde{U}_{\theta} \|_{\mathbf{M}_1} := 
	\| (1-x^2)^{-1+2\epsilon} \tilde{U}_{\theta} \|_{L^\infty(-1,1)} 
	+ \| (1-x^2)^{2\epsilon} \tilde{U}'_{\theta} \|_{L^\infty(-1,1)} \\
	& \hspace{1.9cm} + \|  \tilde{U}''_{\theta} \|_{L^\infty(-\frac{1}{2}, \frac{1}{2})} 
	+ \|  \tilde{U}'''_{\theta} \|_{L^\infty(-\frac{1}{2}, \frac{1}{2})},\\
	& \| \tilde{U}_{\phi} \|_{\mathbf{M}_2} := \|(1-x^2)^{\epsilon}\tilde{U}_{\phi}\|_{L^\infty(-1,1)} + \|(1-x^2)^{1+\epsilon}\tilde{U}'_{\phi} \|_{L^\infty(-1,1)} + \| (1-x^2)^{2+\epsilon}\tilde{U}''_{\phi} \|_{L^\infty(-1,1)}. \\
	\end{split}
\]
Next define the following function spaces: 
\[
	\begin{split}
	& \mathbf{N}_1 = \mathbf{N}_1(\epsilon) := \{ \xi_{\theta}\in C^2(-1/2,1/2)\cap C[-1,1] \mid \xi_\theta(-1) = \xi_\theta(1) =\xi''_{\theta}(0) = 0, \\
	& \hspace{2.8 cm} \| (1-x^2)^{-1+2\epsilon} \xi_{\theta} \|_{L^\infty(-1,1)}<\infty, \|  \xi'_{\theta} \|_{L^\infty(-\frac{1}{2}, \frac{1}{2})}<\infty, \| \xi''_{\theta} \|_{L^\infty(-\frac{1}{2}, \frac{1}{2})}<\infty \}, \\
	& \mathbf{N}_2 = \mathbf{N}_2(\epsilon) := \{\xi_{\phi}\in C(-1,1) \mid \| (1-x^2)^{1+\epsilon}\xi_{\phi} \|_{L^\infty(-1,1)}<\infty\}, \\
	\end{split}
\]
with the following norms accordingly: 
\[
\begin{split}
	& \| \xi_{\theta} \|_{\mathbf{N}_1} := \| (1-x^2)^{-1+2\epsilon} \xi_{\theta} \|_{L^\infty(-1,1)}+\|  \xi'_{\theta} \|_{L^\infty(-\frac{1}{2}, \frac{1}{2})}+\| \xi''_{\theta} \|_{L^\infty(-\frac{1}{2}, \frac{1}{2})}, \\
	& \| \xi_{\phi} \|_{\mathbf{N}_2} := \|(1-x^2)^{1+\epsilon}\xi_{\phi} \|_{L^\infty(-1,1)}. 
	\end{split}
\]
Then let $\mathbf{X} := \{ \tilde{U} = (\tilde{U}_\theta, \tilde{U}_\phi ) \mid \tilde{U}_\theta \in \mathbf{M}_1, \tilde{U}_{\phi} \in \mathbf{M}_2\}$ with norm $ \| \tilde{U} \|_{\mathbf{X}} = \| \tilde{U}_\theta \|_{\mathbf{M}_1} + \| \tilde{U}_\phi \|_{\mathbf{M}_2}$, $\mathbf{Y} := \{ \xi = ( \xi_\theta, \xi_\phi ) \mid \xi_\theta \in \mathbf{N}_1, \xi_\phi \in \mathbf{N}_2 \}$, with norm $ \| \xi \|_{\mathbf{Y}} = \| \xi_\theta \|_{\mathbf{N}_1} + \| \xi_\phi \|_{\mathbf{N}_2}$. It can be proved that $\mathbf{M}_1$, $\mathbf{M}_2$, $\mathbf{N}_1$, $\mathbf{N}_2$, $\mathbf{X}$ and $\mathbf{Y}$ are Banach spaces.


Let $l_i:\mathbf{X}\to \mathbb{R}$, $1\le i\le 4$, be the bounded linear functionals defined by (\ref{sec31:eq:fcnal:l}) for each $V\in \mathbf{X}$. Define 
\begin{equation}\label{sec32:eq:X}
	 \mathbf{X}_1:= \ker l_1 \cap \ker l_2 \cap \ker l_3 \cap \ker l_4.\\
\end{equation}
It can be seen that $\mathbf{X}_1$ is independent of $(c,\gamma)$.

\begin{thm}\label{sec32:thm:1}
	For every compact subset $K\subset I_{1,1}$ ,  for every $(c,\gamma)\in K$, there exist $\delta = \delta(K) > 0$, and $V\in C^{\infty}( K\times B_{\delta}(0), \mathbf{X}_1)$ satisfying $V(c,\gamma,0)=0$ and $\displaystyle \frac{\partial V}{\partial \beta_i}\big |_{\beta=0}=0$, $1\le i\le 4$, $\beta=(\beta_1,\beta_2,\beta_3,\beta_4)$, such that 
	\begin{equation}\label{sec32:eq:U1}
		U = U^{c,\gamma} + \sum_{i=1}^{4} \beta_i V_{c,\gamma}^i + V(c,\gamma,\beta)
	\end{equation}
	satisfies equation (\ref{sec31:eq:NSE}) with $\hat{c}_1= c_1 - \frac{1}{2}\psi[\tilde{U}_{\phi}](-1)$, $\hat{c}_2 = c_2 - \frac{1}{2}\psi[\tilde{U}_{\phi}](1)$, $\hat{c}_3 = c_3 +\frac{1}{2}(\varphi_{c,\gamma}[\tilde{U}_{\theta}])''(0)$.

	Moreover, there exists some $\delta'=\delta'(K)>0$, such that if $\| U - U^{c,\gamma } \|_{\mathbf{X}} < \delta'$, $(c,\gamma) \in K$,  and $U$ satisfies  equation (\ref{sec31:eq:NSE}) with some constants $\hat{c}_1, \hat{c}_2, \hat{c}_3$, then (\ref{sec32:eq:U1}) holds for some $|\beta|< \delta$. 
\end{thm}

Let $l_{2a}, l_{2b}: \mathbf{X}\to \mathbb{R}$ be the bounded linear functionals defined by (\ref{eq_l2ab}) for each $V\in \mathbf{X}$. Define
\begin{equation}\label{sec32:eq:Xab}
  \begin{split}
    & \mathbf{X}_{2a}:= \ker l_{2a} \cap \ker l_3 \cap \ker l_4,\\
    &\mathbf{X}_{2b}:= \ker l_{2b} \cap \ker l_3 \cap \ker l_4.
  \end{split}
\end{equation}
The spaces $\mathbf{X}_{2a}$, $\mathbf{X}_{2b}$ are independent of $(c,\gamma)$. 
\begin{thm}\label{sec32:thm:2}
	For every compact subset $K$ of $I_{1,2}$ or $I_{2,2}$,  for every $(c,\gamma)\in K$, there exist $\delta = \delta(K) > 0$, and $V\in C^{\infty}( K\times B_{\delta}(0), \mathbf{X}_{2a})$ satisfying $V(c,\gamma,0)=0$ and $\displaystyle \frac{\partial V}{\partial \beta_i}\big |_{\beta=0}=0$, $i=2,3,4$, $\beta=(\beta_2,\beta_3,\beta_4)$, such that 
	\begin{equation}\label{sec32:eq:U2}
		U = U^{c,\gamma} + \beta_2 V_{c,\gamma}^{2a} + \beta_3 V_{c,\gamma}^3 +\beta_4 V_{c,\gamma}^4+ V(c,\gamma,\beta)
	\end{equation}
	satisfies equation (\ref{sec31:eq:NSE}) with $\hat{c}_1= c_1 - \frac{1}{2}\psi[\tilde{U}_{\phi}](-1)$, $\hat{c}_2 = c_2 - \frac{1}{2}\psi[\tilde{U}_{\phi}](1)$, $\hat{c}_3 = c_3 +\frac{1}{2}(\varphi_{c,\gamma}[\tilde{U}_{\theta}])''(0)$.

	Moreover, there exists some $\delta'=\delta'(K)>0$, such that if $\| U - U^{c,\gamma } \|_{\mathbf{X}} < \delta'$, $(c,\gamma) \in K$,  and $U$ satisfies  equation (\ref{sec31:eq:NSE}) with some constants $\hat{c}_1, \hat{c}_2, \hat{c}_3$, then (\ref{sec32:eq:U2}) holds for some $|(\beta_2,\beta_3,\beta_4)|< \delta$. 
\end{thm}

\addtocounter{thm}{-1}
\renewcommand{\thethm}{\thesection.\arabic{thm}'}%
\begin{thm}\label{sec32:thm:2'}
	For every compact subset $K$ of $I_{1,3}$ or $I_{3,3}$,  for every $(c,\gamma)\in K$, there exist $\delta = \delta(K) > 0$, and $V\in C^{\infty}( K\times B_{\delta}(0), \mathbf{X}_{2b})$ satisfying $V(c,\gamma,0)=0$ and $\displaystyle \frac{\partial V}{\partial \beta_i}\big |_{\beta=0}=0$, $i=2,3,4$, $\beta=(\beta_2,\beta_3,\beta_4)$, such that 
	\begin{equation}\label{sec32:eq:U2'}
		U = U^{c,\gamma}  + \beta_2 V_{c,\gamma}^{2b} + \beta_3 V_{c,\gamma}^3+\beta_4 V_{c,\gamma}^4 + V(c,\gamma,\beta)
	\end{equation}
	satisfies equation (\ref{sec31:eq:NSE}) with $\hat{c}_1= c_1 - \frac{1}{2}\psi[\tilde{U}_{\phi}](-1)$, $\hat{c}_2 = c_2 - \frac{1}{2}\psi[\tilde{U}_{\phi}](1)$, $\hat{c}_3 = c_3 +\frac{1}{2}(\varphi_{c,\gamma}[\tilde{U}_{\theta}])''(0)$.

	Moreover, there exists some $\delta'=\delta'(K)>0$, such that if $\| U - U^{c,\gamma } \|_{\mathbf{X}} < \delta'$, $(c,\gamma) \in K$,  and $U$ satisfies  equation (\ref{sec31:eq:NSE}) with some constants $\hat{c}_1, \hat{c}_2, \hat{c}_3$, then (\ref{sec32:eq:U2'}) holds for some $|(\beta_2,\beta_3,\beta_4)|< \delta$. 
\end{thm}
\renewcommand{\thethm}{\thesection.\arabic{thm}}%

Define
\begin{equation}\label{sec32:eq:X3}
  \mathbf{X}_3:=  \ker l_3 \cap \ker l_4.
\end{equation}
The spaces $\mathbf{X}_3$ is also independent of $(c,\gamma)$. We have
\begin{thm}\label{sec32:thm:3}
	Let $K$ be a compact set contained in one of $I_{k,l}$ with $5\le k\le 8$ and $1\le l\le 3$, there exist $\delta = \delta(K) > 0$, and $V\in C^{\infty}( K\times B_{\delta}(0), \mathbf{X}_3)$ satisfying $V(c,\gamma,0)=0$ and $\displaystyle \frac{\partial V}{\partial \beta_i}\big |_{\beta=0}=0$, $i=3,4$, $\beta=(\beta_3,\beta_4)$, such that 
	\begin{equation}\label{sec32:eq:U3}
		U = U^{c,\gamma} + \beta_3 V_{c,\gamma}^3+\beta_4 V_{c,\gamma}^4 + V(c,\gamma,\beta)
	\end{equation}
	satisfies equation (\ref{sec31:eq:NSE}) with $\hat{c}_1= c_1 - \frac{1}{2}\psi[\tilde{U}_{\phi}](-1)$, $\hat{c}_2 = c_2 - \frac{1}{2}\psi[\tilde{U}_{\phi}](1)$, $\hat{c}_3 = c_3 +\frac{1}{2}(\varphi_{c,\gamma}[\tilde{U}_{\theta}])''(0)$.  
	
	Moreover, there exists some $\delta'=\delta'(K)>0$, such that if $\| U - U^{c,\gamma } \|_{\mathbf{X}} < \delta'$, $(c,\gamma) \in K$,  and $U$ satisfies  equation (\ref{sec31:eq:NSE}) with some constants $\hat{c}_1, \hat{c}_2, \hat{c}_3$, then (\ref{sec32:eq:U3}) holds for some $|\beta|< \delta$. 
\end{thm}

For $\tilde{U}_{\phi}\in \mathbf{M}_2$, let $\psi[\tilde{U}_{\phi}](x)$ be defined by (\ref{sec31:eq:psi}). Let $K$ be a compact set contained in one of $I_{k,l}$ with $k=1$ or $5\le k\le 8$ or $(k,l)=(2,2)$ or $(3,3)$. Define a map $G = G(c,\gamma,\tilde{U})$ on $K\times \mathbf{X}$ by (\ref{sec31:eq:G}). 

\begin{prop}\label{sec32:prop}
	The map $G$ is in $C^{\infty}(K\times \mathbf{X}, \mathbf{Y})$ in the sense that $G$ has  continuous Fr\'{e}chet derivatives of every order. Moreover, the Fr\'{e}chet derivative of $G$ with respect to $\tilde{U}$ at $(c,\gamma,\tilde{U})\in K\times \mathbf{X}$ is given by the  bounded linear operator $L^{c,\gamma}_{\tilde{U}}: \mathbf{X}\rightarrow \mathbf{Y}$ defined as  in (\ref{sec31:eq:Linear}).
\end{prop}
To prove Proposition \ref{sec32:prop}, we first prove the following lemmas:\\

\begin{lem}\label{sec32:lem:A:well-def}
	For every $(c,\gamma)\in K$, $A(c,\gamma,\cdot): \mathbf{X}\to \mathbf{Y}$ defined by (\ref{sec31:eq:A}) is a well-defined bounded linear operator.
\end{lem}
\begin{proof}
	In the following, $C$ denotes a universal constant which may change from line to line. For convenience we denote $l=l_{c,\gamma}[\tilde{U}_{\theta}]$ defined by (\ref{sec31:eq:l}), and $A=A(c,\gamma,\cdot)$ for some fixed $(c,\gamma)\in K$. We make use of the property of $\bar{U}_{\theta}$ that $\bar{U}_{\theta}\in C^2(-1,1)\cap L^{\infty}(-1,1)$.

	$A$ is clearly linear. For every $\tilde{U}\in\mathbf{X}$, we prove that $A\tilde{U}$ defined by (\ref{sec31:eq:A}) is in $\mathbf{Y}$ and there exists some constant $C$ such that $\|A\tilde{U}\|_{\mathbf{Y}}\le C\|\tilde{U}\|_{\mathbf{X}}$ for all $\tilde{U}\in \mathbf{X}$.

	By computation,
	\[
	    l'(x)=(1-x^2)\tilde{U}''_{\theta}+\bar{U}_{\theta}\tilde{U}'_{\theta}+(2+\bar{U}'_{\theta})\tilde{U}_{\theta},
	\]
	\begin{equation*}
		l''(x)=(1-x^2)\tilde{U}'''_{\theta}+(\bar{U}_{\theta}-2x)\tilde{U}''_{\theta}+2(\bar{U}'_{\theta}+1)\tilde{U}'_{\theta}+\bar{U}''_{\theta}\tilde{U}_{\theta}.
	\end{equation*}
	
	By the fact that $\tilde{U}_{\theta}\in \mathbf{M}_1$, we have 
	\[
	  |l''(0)|\le |\tilde{U}'''_{\theta}(0)|
		+(|\bar{U}_{\theta}(0)|+2)|\tilde{U}''_{\theta}(0)|+2(|\bar{U}'_{\theta}(0)|+1)|\tilde{U}'_{\theta}(0)|+|\bar{U}''_{\theta}(0)| |\tilde{U}_{\theta}(0)|\le C\|\tilde{U}_{\theta}\|_{\mathbf{M}_1}.
	\]
	For $-1<x<1$,
	\[
	\begin{split}
		|(1-x^2)^{-1+2\epsilon}A_{\theta}|& \le |(1-x^2)^{-1+2\epsilon}l(x)|+\frac{1}{2}|l''(0)|(1-x^2)^{2\epsilon}\\
		& \le |(1-x^2)^{2\epsilon}\tilde{U}'_{\theta}|+(2+|\bar{U}_{\theta}|)(1-x^2)^{-1+2\epsilon}|\tilde{U}_{\theta}|+\frac{1}{2}(1-x^2)^{2\epsilon}|l''(0)|\\
		& \le C\|\tilde{U}_{\theta}\|_{\mathbf{M}_1}.
	\end{split}
	\]	
	For $-\frac{1}{2}<x<\frac{1}{2}$, 
	\[
	  \begin{split}
	    |A'_{\theta}|& =|l'(x)-l''(0)x|\\
	       & \le |\tilde{U}''_{\theta}|+|\bar{U}_{\theta}||\tilde{U}'_{\theta}|+(2+|\bar{U}'_{\theta}|)|\tilde{U}_{\theta}|+|l''(0)|\\
	       & \le C\|\tilde{U}_{\theta}\|_{\mathbf{M}_1},
	    \end{split}
	\]
	and
	\[
	\begin{split}
	   |A''_{\theta}|& =|l''(x)-l''(0)|\\
	   & \le |\tilde{U}'''_{\theta}|+(|\bar{U}_{\theta}|+2)|\tilde{U}''_{\theta}|+2(|\bar{U}'_{\theta}|+1)|\tilde{U}'_{\theta}|+|\bar{U}''_{\theta}||\tilde{U}_{\theta}|+|l''(0)|\\
	   & \le C\|\tilde{U}_{\theta}\|_{\mathbf{M}_1}.
	   \end{split} 
	\]

	We also see from the above that $\displaystyle{\lim_{x\to \pm 1}A_{\theta}(x)=0}$. By computation $A''_{\theta}(0)=0$. So we have $A_{\theta}\in \mathbf{N}_1$ and $\|A_{\theta}\|_{\mathbf{N}_1}\le C\|\tilde{U}_{\theta}\|_{\mathbf{M}_1}$.
	
	Next, since $A_{\phi}=(1-x^2)\tilde{U}''_{\phi}+\bar{U}_{\theta}\tilde{U}'_{\phi}$, by the fact that $\tilde{U}_{\phi}\in \mathbf{M}_2$  we have that
	\[
		\left|(1-x^2)^{1+\epsilon}A_{\phi}\right| \le (1-x^2)^{2+\epsilon}|\tilde{U}''_{\phi}|+(1-x^2)^{1+\epsilon}|\bar{U}_{\theta}\|\tilde{U}'_{\phi}|\le C\|\tilde{U}_{\phi}\|_{\mathbf{M}_2}.
	\]
	So $A_{\phi}\in \mathbf{N}_1$, and $\|A_{\phi}\|_{\mathbf{N}_1}\le C\|\tilde{U}_{\phi}\|_{\mathbf{M}_2}$. We have proved that $A\tilde{U}\in \mathbf{Y}$ and $\|A\tilde{U}\|_{\mathbf{Y}}\le C\|\tilde{U}\|_{\mathbf{X}}$ for every $\tilde{U}\in \mathbf{X}$. The proof is finished.
\end{proof}

\begin{lem}\label{sec32:lem:Q:well-def}
	The map $Q:\mathbf{X}\times\mathbf{X}\to \mathbf{Y}$ defined by (\ref{sec31:eq:Q}) is a well-defined bounded bilinear operator.
\end{lem}
\begin{proof}
	In the following, $C$ denotes a universal constant which may change from line to line. For convenience we denote $\psi=\psi[\tilde{U}_{\phi}, \tilde{V}_{\phi}]$ defined by (\ref{sec31:eq:psi}).  

	It is clear that $Q$ is a bilinear operator. For every $\tilde{U},\tilde{V}\in\mathbf{X}$, we will prove that $Q(\tilde{U},\tilde{V})$ is in $\mathbf{Y}$ and there exists some constant $C$ independent of $\tilde{U}$ and $\tilde{V}$ such that $\|Q(\tilde{U},\tilde{V})\|_{\mathbf{Y}}\le C\|\tilde{U}\|_{\mathbf{X}}\|\tilde{V}\|_{\mathbf{X}}$.
    
	For $\tilde{U},\tilde{V}\in \mathbf{X}$, we have, using the fact that $\tilde{U}_{\phi}, \tilde{V}_{\phi}\in \mathbf{M}_2$, that
	\begin{equation}\label{eq3_2_0}
		\left|\frac{ \tilde{U}_\phi(s)\tilde{V} _\phi'(s)}{1-s^2}\right| \le (1-s^2)^{-2-2\epsilon}\|\tilde{U}_{\phi}\|_{\mathbf{M}_2}\|\tilde{V}_{\phi}\|_{\mathbf{M}_2}, \quad  -1<s<1.
	\end{equation}
	It follows that $\psi(\tilde{U},\tilde{V})(x)$ is well-defined and
	\begin{equation}\label{eq3_2_1}
	    |\psi(-1)|\le C\|\tilde{U}_{\phi}\|_{\mathbf{M}_2}\|\tilde{V}_{\phi}\|_{\mathbf{M}_2}, \quad |\psi(1)|\le C\|\tilde{U}_{\phi}\|_{\mathbf{M}_2}\|\tilde{V}_{\phi}\|_{\mathbf{M}_2}.  
	\end{equation}
	Moreover,
	\begin{equation}\label{sec32:eq:temp1}
		\begin{split}
		& \left|\psi(x)-\frac{1}{2}\psi(-1)(1-x)-\frac{1}{2}\psi(1)(1+x)\right|\\
		& =\left|\frac{1}{2}\psi(x)(1-x)+\frac{1}{2}\psi(x)(1+x)-\frac{1}{2}\psi(-1)(1-x)-\frac{1}{2}\psi(1)(1+x)\right|\\
		& \le \frac{1}{2}(1-x)|\psi(x)-\psi(-1)|+\frac{1}{2}(1+x)|\psi(x)-\psi(1)|\\
		& =\frac{1}{2}(1-x)\left|\int_{-1}^{x} \int_{0}^{l} \int_{0}^{t} \frac{2 \tilde{U}_\phi(s)\tilde{V} _\phi'(s)}{1-s^2} ds dt dl\right|+\frac{1}{2}(1+x)\left|\int_{1}^{x} \int_{0}^{l} \int_{0}^{t} \frac{2 \tilde{U}_\phi(s)\tilde{V} _\phi'(s)}{1-s^2} ds dt dl\right|\\
		& \le C(1-x)(1+x)^{1-2\epsilon}\|\tilde{U}_\phi\|_{\mathbf{M}_2}\|\tilde{V} _{\phi}\|_{\mathbf{M}_2}+C(1+x)(1-x)^{1-2\epsilon}\|\tilde{U}_\phi\|_{\mathbf{M}_2}\|\tilde{V} _{\phi}\|_{\mathbf{M}_2}\\
		& \le C(1-x^2)^{1-2\epsilon}\|\tilde{U}_\phi\|_{\mathbf{M}_2}\|\tilde{V} _{\phi}\|_{\mathbf{M}_2}.
		\end{split}
	\end{equation}
	By (\ref{eq3_2_0}), we also have
	\begin{equation}\label{eq3_2_2}
	   |\psi'(x)|=\left|\int_{0}^{x} \int_{0}^{t} \frac{2 \tilde{U}_\phi(s)\tilde{V} _\phi'(s)}{1-s^2} ds dt\right|\le C \|\tilde{U}_\phi\|_{\mathbf{M}_2}\|\tilde{V} _{\phi}\|_{\mathbf{M}_2}, \quad -\frac{1}{2}<x<\frac{1}{2},
	\end{equation}
	and 
	\begin{equation}\label{eq3_2_3}
	   |\psi''(x)|=\left|\int_{0}^{x} \frac{2 \tilde{U}_\phi(s)\tilde{V} _\phi'(s)}{1-s^2} ds\right|\le C \|\tilde{U}_\phi\|_{\mathbf{M}_2}\|\tilde{V} _{\phi}\|_{\mathbf{M}_2}, \quad -\frac{1}{2}<x<\frac{1}{2}.
	\end{equation}
	Using the fact that $\tilde{U}_{\theta},\tilde{V}_{\theta}\in \mathbf{M}_1$, we have
	\begin{equation}\label{sec32:eq:temp2}
		\begin{split}
		|(\tilde{U}_{\theta}\tilde{V}_{\theta})''(0)| & \le |\tilde{U}''_{\theta}(0)\|\tilde{V}_{\theta}(0)|+2|\tilde{U}'_{\theta}(0)\|\tilde{V}'_{\theta}(0)|+|\tilde{U}_{\theta}(0)\|\tilde{V}''_{\theta}(0)|\\
		& \le C\|\tilde{U}_{\theta}\|_{\mathbf{M}_1}\|\tilde{V}_{\theta}\|_{\mathbf{M}_1}.
	\end{split}
	\end{equation}
	So by (\ref{sec32:eq:temp1}),  (\ref{sec32:eq:temp2}), and the fact that $\tilde{U}_{\theta}, \tilde{V}_{\theta}\in \mathbf{M}_1$, we have that for $-1<x<1$,
	\[
	\begin{split}
		& |(1-x^2)^{-1+2\epsilon}Q_{\theta}(x)|\\
		& \le \frac{1}{2}|(1-x^2)^{-1+2\epsilon}\tilde{U}_{\theta}(x)\|\tilde{V}_{\theta}(x)|+(1-x^2)^{-1+2\epsilon}\left|\psi(x)-\frac{1}{2}\psi(-1)(1-x)-\frac{1}{2}\psi(1)(1+x)\right|\\
		& +\frac{1}{4}(1-x^2)^{2\epsilon}|(\tilde{U}_{\theta}\tilde{V}_{\theta})''(0)|\\
		& \le \frac{1}{2}\|\tilde{U}_{\theta}\|_{\mathbf{M}_1}\|\tilde{V}_{\theta}\|_{\mathbf{M}_1}+C\|\tilde{U}_\phi\|_{\mathbf{M}_2}\|\tilde{V} _{\phi}\|_{\mathbf{M}_2}+C(1-x^2)^{2\epsilon}\|\tilde{U}_{\theta}\|_{\mathbf{M}_1}\|\tilde{V}_{\theta}\|_{\mathbf{M}_1}\\
		& \le C\|\tilde{U}\|_{\mathbf{X}}\|\tilde{V}\|_{\mathbf{X}}.
	\end{split}
	\]
	From this we also have $\displaystyle \lim_{x\to \pm 1}Q_{\theta}(x)=0$. 
	
	By (\ref{eq3_2_1}), (\ref{eq3_2_2}), (\ref{eq3_2_3}),  (\ref{sec32:eq:temp2}) and the fact that $\tilde{U}_{\theta}, \tilde{V}_{\theta}\in \mathbf{M}_1$, we have that for $-\frac{1}{2}<x<\frac{1}{2}$,
	\[
	  \begin{split}
	    |Q'_{\theta}(x)| & =\left|\frac{1}{2}(\tilde{U}'_{\theta}\tilde{V}_{\theta}+\tilde{U}_{\theta}\tilde{V}'_{\theta})+\psi'(x)+\frac{1}{2}(\psi(-1)-\psi(1))-\frac{1}{2}(\tilde{U}_{\theta}\tilde{V}_{\theta})''(0)x\right|\\
	    & \le C\|\tilde{U}\|_{\mathbf{X}}\|\tilde{V}\|_{\mathbf{X}},
	    \end{split}
	\]
	and 
	\[
	    \begin{split}
	    |Q''_{\theta}(x)| 
	    & =\left|\frac{1}{2}(\tilde{U}''_{\theta}\tilde{V}_{\theta}+2\tilde{U}'_{\theta}\tilde{V}'_{\theta}+\tilde{U}_{\theta}\tilde{V}''_{\theta})+\psi''(x)-\frac{1}{2}(\tilde{U}_{\theta}\tilde{V}_{\theta})''(0)\right|\\
	    & \le C\|\tilde{U}\|_{\mathbf{X}}\|\tilde{V}\|_{\mathbf{X}}.
	    \end{split}
	\]
	It is obvious that $\psi''(0)=0$, so we have $Q''_{\theta}(0)=0$. Hence $Q_{\theta}\in\mathbf{N}_1$ and $\|Q_{\theta}\|_{\mathbf{N}_1}\le C(\epsilon)\|\tilde{U}\|_{\mathbf{X}}\|\tilde{V}\|_{\mathbf{X}}$.
	
	Next, since $Q_{\phi}(x)=\tilde{U}_{\theta}(x)\tilde{V}'_{\phi}(x)$, for $-1<x<1$, 
	\[
		\left|(1-x^2)^{1+\epsilon}Q_{\phi}(x)\right|  \le |\tilde{U}_{\theta}(x)|(1-x^2)^{1+\epsilon}|\tilde{V}'_{\phi}|
		\le \|\tilde{U}_{\theta}\|_{\mathbf{M}_1}\|\tilde{V}_{\phi}\|_{\mathbf{M}_2}.
	\]
	So $Q_{\phi}\in \mathbf{N}_2$ and 
	$
		\|Q_{\phi}\|_{\mathbf{N}_2}\le \|\tilde{U}_{\theta}\|_{\mathbf{M}_1}\|\tilde{V}_{\phi}\|_{\mathbf{M}_2}.
	$
	Thus we have proved that $Q(\tilde{U}, \tilde{V})\in \mathbf{Y}$ and $\|Q(\tilde{U},\tilde{V})\|_{\mathbf{Y}}\le C\|\tilde{U}\|_{\mathbf{X}}\|\tilde{V}\|_{\mathbf{X}}$ for all $\tilde{U}, \tilde{V}\in \mathbf{X}$. Lemma \ref{sec32:lem:Q:well-def} is proved.
\end{proof}

\noindent {\bf Proof of Proposition \ref{sec32:prop}.}
By definition, $G(c,\gamma,\tilde{U})=A(c,\gamma,\tilde{U})+Q(\tilde{U},\tilde{U})$ for $(c,\gamma,\tilde{U})\in K \times \mathbf{X}$.  
Using standard theories in functional analysis, by Lemma \ref{sec32:lem:Q:well-def} it is clear that $Q$ is $C^{\infty}$ on $\mathbf{X}$.  
By Lemma \ref{sec32:lem:A:well-def}, $A(c,\gamma; \cdot): \mathbf{X}\to \mathbf{Y}$ is $C^\infty$ for each $(c,\gamma)\in K$. Let $\alpha=(\alpha_1,\alpha_2,\alpha_3)$ be a multi-index where $\alpha_i\ge 0$, $i=1,2,3$, and $j\ge 0$. For all $|\alpha|+j\ge 1$,  we have
\begin{equation}\label{sec32:eq:paradiff:A}
	\pt_{c}^\alpha \pt_{\gamma}^j A(c,\gamma,\tilde{U}) = \pt_{c}^\alpha \pt_{\gamma}^j U^{c,\gamma}_{\theta}\left(
	\begin{matrix}
		\tilde{U}_{\theta}  \\  \tilde{U}'_{\phi} 
	\end{matrix}\right) + \frac{1}{2} (\pt_{c}^\alpha \pt_{\gamma}^j U^{c,\gamma}_{\theta} \cdot \tilde{U}_\theta)'' (0) 
	\begin{pmatrix}
	1-x^2   \\   0  
	\end{pmatrix}.  
\end{equation}
By Proposition \ref{propA_1} (1), we have
\[
	|(1-x^2)^{-1+2\epsilon} \pt_{c}^\alpha \pt_{\gamma}^j A_{\theta}(c,\gamma,\tilde{U}) | \leq C(\alpha,j,K) \| \tilde{U}_{\theta} \|_{\mathbf{M}_1}, \quad -1<x<1,
\]
and for $-\frac{1}{2}<x<\frac{1}{2}.$
\[
     | \pt_{c}^\alpha \pt_{\gamma}^j A'_{\theta}(c,\gamma,\tilde{U}) | \leq C(\alpha,j,K) \| \tilde{U}_{\theta} \|_{\mathbf{M}_1}, \quad | \pt_{c}^\alpha \pt_{\gamma}^j A''_{\theta}(c,\gamma,\tilde{U}) | \leq C(\alpha,j,K) \| \tilde{U}_{\theta} \|_{\mathbf{M}_1}.
\]
The above estimates and (\ref{sec32:eq:paradiff:A}) also implies that 
$$
	\pt_{c}^\alpha \pt_{\gamma}^j A_{\theta}(c,\gamma,\tilde{U}) (-1) = \pt_{c}^\alpha \pt_{\gamma}^jA_{\theta}(c,\gamma,\tilde{U})(1) = \pt_{c}^\alpha \pt_{\gamma}^j A_{\theta}(c,\gamma,\tilde{U})'' (0) = 0. 
$$
So $\pt_{c}^\alpha \pt_{\gamma}^j A_{\theta}(c,\gamma,\tilde{U}) \in \mathbf{N}_1$, with $\|\pt_{c}^\alpha \pt_{\gamma}^j A_{\theta}(c,\gamma,\tilde{U}) \|_{\mathbf{N}_1} \leq C(\alpha,j,K) \|\tilde{U}_{\theta} \|_{\mathbf{M}_1}$ for all $(c,\gamma,\tilde{U}) \in K\times \mathbf{X}$. 

Next, by Proposition \ref{propA_1} (1) and the fact that $\tilde{U}_{\phi}\in \mathbf{M}_1$, we have
\begin{equation}\label{sec32:eq:paradiff:temp}
	(1-x^2)^{1+\epsilon} | \pt_{c}^\alpha \pt_{\gamma}^j A_{\phi}(c,\gamma,\tilde{U})(x)| 
	= |\pt_{c}^\alpha \pt_{\gamma}^j  U^{c,\gamma}_{\theta}(x)| \cdot |(1-x^2)^{1+\epsilon}U'_{\phi}|
	\leq C(\alpha,j,K) \| \tilde{U}_{\phi} \|_{\mathbf{M}_2}. 
\end{equation}
So $\pt_{c}^\alpha \pt_{\gamma}^j A_{\phi}(c,\gamma,\tilde{U})\in \mathbf{N}_2$ with $ \|\pt_{c}^\alpha \pt_{\gamma}^j A_{\phi}(c,\gamma,\tilde{U}) \|_{\mathbf{N}_2} \leq C(\alpha,j,K) \| \tilde{U}_{\phi} \|_{\mathbf{M}_2}$ for all $(c,\gamma,\tilde{U})\in K\times \mathbf{X}$. Thus $\pt_{c}^\alpha \pt_{\gamma}^j A(c,\gamma,\tilde{U})\in \mathbf{Y}$, with $\| \pt_{c}^\alpha \pt_{\gamma}^j A(c,\gamma,\tilde{U}) \|_{\mathbf{Y}} \leq C(\alpha,j,K) \| \tilde{U} \|_{\mathbf{X}}$ for all $(c,\gamma,\tilde{U})\in K\times \mathbf{X}$, $|\alpha|+j\ge 1$. 

So for each $(c,\gamma)\in K$, $\pt_{c}^\alpha \pt_{\gamma}^jA(c,\gamma; \cdot): \mathbf{X}\to \mathbf{Y}$ is a bounded linear map with uniform bounded norm on $K$. Then by standard theories in functional analysis, $A: K\times\mathbf{X}\to \mathbf{Y}$ is $C^{\infty}$. So $G$ is a $C^{\infty}$ map from $K  \times \mathbf{X}$ to $\mathbf{Y}$. By direct calculation we get its Fr\'{e}chet derivative with respect to $\mathbf{X}$ is given by  the linear bounded operator $L^{c,\gamma}_{\tilde{U}}: \mathbf{X}\rightarrow \mathbf{Y}$ defined as  (\ref{sec31:eq:Linear}). The proof is finished.     \qed 
\\



By Proposition \ref{sec32:prop}, $L_0^{c,\gamma}: \mathbf{X}\to\mathbf{Y}$, the Fr\'{e}chet derivative of $G$ with respect to $\tilde{U}$ at $\tilde{U}=0$, is given by (\ref{sec31:eq:Linear0}). 

Let $a_{c,\gamma}(x), b_{c,\gamma}(x)$ be the functions defined by (\ref{sec31:eq:ab}). For convenience we denote $a(x)=a_{c,\gamma}(x)$, $b(x)=b_{c,\gamma}(x)$, we have
\begin{lem}\label{lem3_2_ab}
  For $(c,\gamma)\in I_{k,l}$ with $(k,l)\in A_1$, there exists some constant $C>0$, depending only on $(c,\gamma)$, such that  for any $-1<x<1$, 
  \begin{equation}\label{sec32:eq:eb}
		e^{b(x)}\le C(1+x)^{\frac{U^{c,\gamma}_{\theta}(-1)}{2}}(1-x)^{-\frac{U^{c,\gamma}_{\theta}(1)}{2}}, \quad e^{-b(x)}\le C(1+x)^{-\frac{U^{c,\gamma}_{\theta}(-1)}{2}}(1-x)^{\frac{U^{c,\gamma}_{\theta}(1)}{2}},
	\end{equation}
	and 
	\begin{equation}\label{sec32:eq:ea}
		e^{a(x)}\le C(1+x)^{\frac{U^{c,\gamma}_{\theta}(-1)}{2}-1}(1-x)^{-1-\frac{U^{c,\gamma}_{\theta}(1)}{2}}, \quad e^{-a(x)}\le C(1+x)^{1-\frac{U^{c,\gamma}_{\theta}(-1)}{2}}(1-x)^{1+\frac{U^{c,\gamma}_{\theta}(1)}{2}}. 
	\end{equation}
\end{lem}
\begin{proof}
   Denote $\bar{U}_{\theta}:=U^{c,\gamma}_{\theta}$, let 
   \[
		\alpha_0=\left\{
		\begin{split}
		  & \min\left\{1,1-\frac{\bar{U}_{\theta}(-1)}{2},   1+\frac{\bar{U}_{\theta}(1)}{2} \right\}, \textrm{ if } \bar{U}_{\theta}(-1)<2\textrm{ and }\bar{U}_{\theta}(1)>-2,\\
		  & \min\{1, 1-\frac{\bar{U}_{\theta}(-1)}{2}\}, \textrm{ if } \bar{U}_{\theta}(-1)<2\textrm{ and }\bar{U}_{\theta}(1)\le -2,\\
		  & \min\{1, 1+\frac{\bar{U}_{\theta}(1)}{2}\}, \textrm{ if } \bar{U}_{\theta}(-1)\ge 2\textrm{ and }\bar{U}_{\theta}(1)> -2,\\
		  & 1, \textrm{ if }\bar{U}_{\theta}(-1)\ge 2\textrm{ and }\bar{U}_{\theta}(1)\le -2.
		  \end{split}
		  \right.
	 \]
	Since $(c,\gamma)\in I_{k,l}$ with $(k,l)\in A_1$, we have $\bar{U}_{\theta}(-1)<3$ and $\bar{U}_{\theta}(-1)\ne 2$, or $\bar{U}_{\theta}(-1)=2$ with $\eta_1=0$, and $\bar{U}_{\theta}(1)>-3$ and $\bar{U}_{\theta}(1)\ne -2$, or $\bar{U}_{\theta}(1)=-2$ with $\eta_2=0$. According to Theorem 1.3 in \cite{LLY1} and Lemma 2.3 in \cite{LLY2}, we then have
	\[
	   \bar{U}_{\theta}=\left\{
		\begin{split} 
		  & \bar{U}_{\theta}(-1)+O((1+x)^{\alpha_0}), \textrm{ if }\bar{U}_{\theta}(-1)\ne 0,\\
		  & O((1+x)\ln(1+x)), \textrm{ if }\bar{U}_{\theta}(-1)=0.
		  \end{split}
		\right.
	\]
	and 
	\[
	   \bar{U}_{\theta}=\left\{
		\begin{split} 
		  & \bar{U}_{\theta}(1)+O((1-x)^{\alpha_0}), \textrm{ if }\bar{U}_{\theta}(1)\ne 0,\\
		  & O((1-x)\ln(1-x)), \textrm{ if }\bar{U}_{\theta}(1)=0.
		  \end{split}
		\right.
	\]
	Thus, by definition of $a(x)$ and $b(x)$ in (\ref{sec31:eq:ab}),  for $-1<x<1$, we have
	\begin{equation*}
	\begin{split}
		& b(x)=\frac{\bar{U}_{\theta}(-1)}{2}\ln(1+x)-\frac{\bar{U}_{\theta}(1)}{2}\ln(1-x)+O(1), \\
		& a(x)=\left(\frac{\bar{U}_{\theta}(-1)}{2}-1\right)\ln(1+x)-\left(\frac{\bar{U}_{\theta}(1)}{2}+1\right)\ln(1-x)+O(1).
	\end{split}
	\end{equation*}
	The lemma then follows from the above estimates.
\end{proof}

For $\xi = (\xi_\theta, \xi_\phi) \in \mathbf{Y}$, let the map $W^{c,\gamma}$ be defined as 
$$
	W^{c,\gamma}(\xi) := (W^{c,\gamma}_\theta(\xi), W^{c,\gamma}_\phi(\xi)), 
$$
where 
\[
	W^{c,\gamma}_{\theta}(\xi)=\left\{ 
	\begin{array}{ll}
		W^{c,\gamma,1}_{\theta}(\xi), & \textrm{ if } (c,\gamma)\in I_{1,1}, \\
		W^{c,\gamma,2a}_{\theta}(\xi), & \textrm{ if }(c,\gamma)\in I_{1,2} \textrm{ or }I_{2,2}, \\
		W^{c,\gamma,2b}_{\theta}(\xi), & \textrm{ if }(c,\gamma)\in I_{1,3} \textrm{ or }I_{3,3}, \\
		W^{c,\gamma,3}_{\theta}(\xi), & \textrm{ if }  (c,\gamma)\in I_{k,l} \textrm{ for } 5\le k \le 8, \textrm{ and }1\le l\le 3. 
	\end{array}
	\right.
\]
$W^{c,\gamma,i}_{\theta}$, $i=1,2a,2b,3$, are defined by (\ref{sec31:eq:Wthei}), and $W^{c,\gamma}_\phi(\xi)$ is defined by (\ref{sec31:eq:Wphi}).

\begin{lem}\label{sec32:lem:W}
	For every $(c,\gamma)\in K$, $W^{c,\gamma}: \mathbf{Y}\rightarrow\mathbf{X}$ is continuous, and is a right inverse of $L^{c,\gamma}_{0}$. 
\end{lem}
\begin{proof}
	In the following, $C$ denotes a universal constant which may change from line to line.  We make use of the property that $\bar{U}_{\theta}\in C^2(-1,1)\cap L^{\infty}(-1,1)$ and the range of $\epsilon$.  For convenience let us write  $W:=W^{c,\gamma}(\xi)$ and $W^i_{\theta}:=W_{\theta}^{c,\gamma,i}(\xi)$ for $\xi\in \mathbf{Y}$. 
	By Lemma \ref{lem3_2_ab}, we have estimates (\ref{sec32:eq:eb}) and (\ref{sec32:eq:ea}).
	
	We first prove $W_{\theta}$ is well-defined. 

	\noindent
	\textbf{Claim.} There exists $C>0$, such that 
	\begin{equation}\label{sec32:eq:Wthe:bdd}
		|(1-x^2)^{-1+2\epsilon}W_{\theta}(x)|\le C \|\xi_{\theta}\|_{\mathbf{N}_1}.
	\end{equation}

	\noindent
	\textbf{Proof of the Claim.} We prove the claim for each $W^i$, $i=1,2a,2b,3$.

	\textbf{Case 1.} $(c,\gamma)\in I_{1,1}$, then $\bar{U}_{\theta}(-1)<2$ and $\bar{U}_{\theta}(1)>-2$. 

	In this case $W_{\theta}=W^1_{\theta}$. Using the fact that $\xi_{\theta}\in \mathbf{N}_1$, in the expression of $W^1_{\theta}$ in (\ref{sec31:eq:Wthei}), 
	\[
		|(1-x^2)^{-1+2\epsilon}W^1_{\theta}(x)|\le (1-x^2)^{-1+2\epsilon}\|\xi_{\theta}\|_{\mathbf{N}_1}e^{-a(x)}\int_{0}^{x}e^{a(s)}(1-s^2)^{-2\epsilon}ds, \quad -1<x<1.
	\]
	Applying (\ref{sec32:eq:ea}) in the above, using the fact that $4\epsilon>\max\{\bar{U}_{\theta}(-1),-\bar{U}_{\theta}(1)\}$, we have
	\begin{equation}\label{sec32:eq:Wthe:bdd1}
	\begin{split}
		& \quad |(1-x^2)^{-1+2\epsilon}W^1_{\theta}(x)| \\
		& \le \|\xi_{\theta}\|_{\mathbf{N}_1}(1+x)^{-\frac{\bar{U}_{\theta}(-1)}{2}+2\epsilon}(1-x)^{\frac{\bar{U}_{\theta}(1)}{2}+2\epsilon}\int_{0}^{x}(1+s)^{\frac{\bar{U}_{\theta}(-1)}{2}-1-2\epsilon}(1-s)^{-1-\frac{\bar{U}_{\theta}(1)}{2}-2\epsilon}ds\\
		& \le C \|\xi_{\theta}\|_{\mathbf{N}_1}\left(1+(1+x)^{-\frac{\bar{U}_{\theta}(-1)}{2}+2\epsilon}\right)\left(1+(1-x)^{\frac{\bar{U}_{\theta}(1)}{2}+2\epsilon}\right)\\
		& \le C \|\xi_{\theta}\|_{\mathbf{N}_1}.
	\end{split}
	\end{equation}

	\textbf{Case 2.} $(c,\gamma)\in I_{1,2} \textrm{ or }I_{2,2}$, then $2<\bar{U}_{\theta}(-1)<3 \textrm{ or } \bar{U}_{\theta}(-1)=2 \textrm{ with } \eta_1=0, \textrm{ and } \bar{U}_{\theta}(1)>-2$. 

	In this case $W_{\theta}=W^{2a}_{\theta}$. Using the fact that $\xi_{\theta}\in \mathbf{N}_1$, and (\ref{sec32:eq:ea}) we first have
	\[
		\int_{-1}^{0}e^{a(s)}\frac{|\xi_{\theta}(s)|}{1-s^2}ds\le C\|\xi_{\theta}\|_{\mathbf{N}_1}\int_{-1}^{0}(1+s)^{\frac{\bar{U}_{\theta}(-1)}{2}-1-2\epsilon}ds\le C\|\xi_{\theta}\|_{\mathbf{N}_1}. 
	\]
	So the definition of $W^{2a}_{\theta}$ makes sense.
	
	In the expression of $W^{2a}_{\theta}$ in (\ref{sec31:eq:Wthei}),
	\[
		|(1-x^2)^{-1+2\epsilon}W^{2a}_{\theta}(x)|\le (1-x^2)^{-1+2\epsilon}\|\xi_{\theta}\|_{\mathbf{N}_1}e^{-a(x)}\int_{-1}^{x}e^{a(s)}(1-s^2)^{-2\epsilon}ds, \quad -1<x<1.
	\]
	Applying (\ref{sec32:eq:ea}) in the above, using $-\frac{\bar{U}_{\theta}(1)}{4}<\epsilon<\frac{1}{2}$ and $\bar{U}_{\theta}(-1)>2$, we have
	\begin{equation}\label{sec32:eq:Wthe:bdd2}
	\begin{split}
		& \quad |(1-x^2)^{-1+2\epsilon}W^{2a}_{\theta}(x)|  \\
		& \le \|\xi_{\theta}\|_{\mathbf{N}_1}(1+x)^{-\frac{\bar{U}_{\theta}(-1)}{2}+2\epsilon}(1-x)^{\frac{\bar{U}_{\theta}(1)}{2}+2\epsilon}\int_{-1}^{x}(1+s)^{\frac{\bar{U}_{\theta}(-1)}{2}-1-2\epsilon}(1-s)^{-1-\frac{\bar{U}_{\theta}(1)}{2}-2\epsilon}ds\\
		& \le C \|\xi_{\theta}\|_{\mathbf{N}_1}\left(1+(1-x)^{\frac{\bar{U}_{\theta}(1)}{2}+2\epsilon}\right)\\
		& \le C \|\xi_{\theta}\|_{\mathbf{N}_1}. 
	\end{split}
	\end{equation}

	\textbf{Case} 3. $(c,\gamma)\in I_{1,3} \textrm{ or }I_{3,3}$. 
	Similar as in Case 2, we can prove
	\begin{equation}\label{sec32:eq:Wthe:bdd3}
		|(1-x^2)^{-1+2\epsilon}W^{2b}_{\theta}(x)|  \le C \|\xi_{\theta}\|_{\mathbf{N}_1}.
	\end{equation}
	
	\textbf{Case} 4. $(c,\gamma)\in I_{k,l} \textrm{ for } 5\le k \le 8, \textrm{ and }1\le l\le 3$, then 
	$2<\bar{U}_{\theta}(-1)<3$ or $\bar{U}_{\theta}(-1)=2 $ with $ \eta_1=0$,  and $-3<\bar{U}_{\theta}(1)<-2$ or $ \bar{U}_{\theta}(1)=-2$ with $\eta_2=0$.
	
	In this case, $W_{\theta}=W^{3}_{\theta}$. For convenience write $C_W:=C^{c,\gamma}_W(\xi)$, defined in \eqref{eq:CW}. Using the fact that $\xi_{\theta}\in \mathbf{N}_1$, and (\ref{sec32:eq:ea}) we have
	\[
		\int_{-1}^{1}e^{a(s)}\frac{|\xi_{\theta}(s)|}{1-s^2}ds\le C\|\xi_{\theta}\|_{\mathbf{N}_1}\int_{-1}^{1}(1+s)^{\frac{\bar{U}_{\theta}(-1)}{2}-1-2\epsilon}(1-s)^{-1-\frac{\bar{U}_{\theta}(1)}{2}-2\epsilon}ds\le C\|\xi_{\theta}\|_{\mathbf{N}_1},
	\]
	and 
	\[
		\int_{-1}^{1}e^{a(s)}ds \ge C\int_{-1}^{1}(1+s)^{\frac{\bar{U}_{\theta}(-1)}{2}-1}(1-s)^{-1-\frac{\bar{U}_{\theta}(1)}{2}}ds\ge C>0. 
	\]
	So $C_W$ is finite, and 
	\begin{equation}\label{eq_cw}
	  |C_W|\le C\|\xi_{\theta}\|_{\mathbf{N}_1}.
	\end{equation}
	The definition of $W^3_{\theta}$ makes sense.
	
	For $-1<x<0$, using (\ref{sec32:eq:ea}), the fact that $\xi_{\theta}\in \mathbf{N}_1$, and $0<\epsilon<\frac{1}{2}$, $\bar{U}_{\theta}(-1)>2$, we have
	\[
	\begin{split}
		|(1-x^2)^{-1+2\epsilon}W^3_{\theta}(x)|& \le (1-x^2)^{-1+2\epsilon}\|\xi_{\theta}\|_{\mathbf{N}_1}e^{-a(x)}\int_{-1}^{x}e^{a(s)}\left((1-s^2)^{-2\epsilon}-C_W\right)ds, \\
		&  \le C \|\xi_{\theta}\|_{\mathbf{N}_1}(1+x)^{-\frac{\bar{U}_{\theta}(-1)}{2}+2\epsilon}\int_{-1}^{x}(1+s)^{\frac{\bar{U}_{\theta}(-1)}{2}-1-2\epsilon}ds\\
		& \le C \|\xi_{\theta}\|_{\mathbf{N}_1}. 
	\end{split}
	\]
	
	For $0\le x<1$, by computation,
	\[
	\begin{split}
		& \quad W^3_{\theta}(x)\\
		& =e^{-a_{c,\gamma}(x)}\int_{-1}^{x}e^{a_{c,\gamma}(s)}\left(\frac{\xi_{\theta}(s)}{1-s^2}-C_W\right)ds\\
		& =e^{-a_{c,\gamma}(x)}\int_{-1}^{x}e^{a_{c,\gamma}(s)}\frac{\xi_{\theta}(s)}{1-s^2}ds-e^{-a(x)}\frac{\int_{-1}^{1}e^{a_{c,\gamma}(s)}\frac{\xi_{\theta}(s)}{1-s^2}ds}{\int_{-1}^{1}e^{a_{c,\gamma}(s)}ds}\left(\int_{-1}^{1}e^{a(s)}ds+\int_{1}^{x}e^{a(s)}ds\right)\\
		& =e^{-a_{c,\gamma}(x)}\int_{1}^{x}e^{a_{c,\gamma}(s)}\frac{\xi_{\theta}(s)}{1-s^2}ds-C_We^{-a_{c,\gamma}(x)}\int_{1}^{x}e^{a_{c,\gamma}(s)}ds.
	\end{split}
	\]
	Then using (\ref{sec32:eq:ea}), the fact that $\xi_{\theta}\in \mathbf{N}_1$, and $0<\epsilon<\frac{1}{2}$, $\bar{U}_{\theta}(1)<-2$, we have
	\[
	\begin{split}
		|(1-x^2)^{-1+2\epsilon}W^3_{\theta}(x)| 
		& \le C\|\xi_{\theta}\|_{\mathbf{N}_1}(1-x)^{\frac{\bar{U}_{\theta}(1)}{2}+2\epsilon}\int_{1}^{x}(1-s)^{-1-\frac{\bar{U}_{\theta}(1)}{2}-2\epsilon}ds\\
		& +C\|\xi_{\theta}\|_{\mathbf{N}_1}(1-x)^{\frac{\bar{U}_{\theta}(1)}{2}+2\epsilon}\int_{1}^{x}(1-s)^{-1-\frac{\bar{U}_{\theta}(1)}{2}}ds\\
		& \le C \|\xi_{\theta}\|_{\mathbf{N}_1}. 
	\end{split}
	\]
	Thus for all $-1<x<1$, 
	\begin{equation}\label{sec32:eq:Wthe:bdd4}
		(1-x^2)^{-1+2\epsilon} |W^3_{\theta}(x)| \leq C \|\xi_{\theta}\|_{\mathbf{N}_1}. 
	\end{equation}
	So (\ref{sec32:eq:Wthe:bdd}) can be obtained from (\ref{sec32:eq:Wthe:bdd1}), (\ref{sec32:eq:Wthe:bdd2}), (\ref{sec32:eq:Wthe:bdd3}) and (\ref{sec32:eq:Wthe:bdd4}). The claim is proved. From this we also have $\lim_{x\to -1^+}W_{\theta}(x)=\lim_{x\to 1^-}W_{\theta}(x)=0$.
	
	By the first line of (\ref{sec31:eq:diff:Wthe}), (\ref{sec31:eq:diff:a}) and (\ref{sec32:eq:Wthe:bdd}), we have that for $i=1,2a,2b$,
	\[
	\begin{split}
		|(1-x^2)^{2\epsilon}(W_{\theta}^i)'| 
		& \le |(2+|\bar{U}_{\theta}|)(1-x^2)^{-1+2\epsilon}W^i_{\theta}|+(1-x^2)^{-1+2\epsilon}|\xi_{\theta}(x)|  \\
		& \le  C \|\xi_{\theta}\|_{\mathbf{N}_1},\quad -1<x<1.
	\end{split}
	\]
	By the second line of (\ref{sec31:eq:diff:Wthe}), (\ref{sec31:eq:diff:a}), (\ref{sec32:eq:Wthe:bdd}), and (\ref{eq_cw}), we have
	\[
	    \begin{split}
		& \quad |(1-x^2)^{2\epsilon}(W_{\theta}^3)'| \\
		& \le |(2+|\bar{U}_{\theta}|)(1-x^2)^{-1+2\epsilon}W^3_{\theta}|+(1-x^2)^{-1+2\epsilon}|\xi_{\theta}(x)| + |C_W|(1-x^2)^{2\epsilon}  \\
		& \le  C \|\xi_{\theta}\|_{\mathbf{N}_1},\quad -1<x<1.
	\end{split}
	\] 
	Thus,
	\begin{equation}\label{eq3_2_W_1}
	   |(1-x^2)^{2\epsilon}W'_{\theta}|\le  C \|\xi_{\theta}\|_{\mathbf{N}_1},\quad -1<x<1.
	\end{equation}
	
	By (\ref{sec31:eq:diff:a}), it can be seen that $|a''(x)|,|a'''(x)|\le C$ for $-\frac{1}{2}<x<\frac{1}{2}$. Then using this fact, (\ref{sec32:eq:Wthe:bdd}) and (\ref{eq3_2_W_1}), we have, for $-\frac{1}{2}<x<\frac{1}{2}$,
	\[
	    |W''(x)|=\left|a''(x)W_{\theta}(x)+a'(x)W'_{\theta}(x)+\left(\frac{\xi_{\theta}}{1-x^2}\right)'\right|\le C \|\xi_{\theta}\|_{\mathbf{N}_1},
	\]
	and 
	\[
	     |W'''(x)|=\left|a'''(x)W_{\theta}(x)+2a''(x)W'_{\theta}(x)+a'(x)W''_{\theta}(x)+\left(\frac{\xi_{\theta}}{1-x^2}\right)''\right|\le C \|\xi_{\theta}\|_{\mathbf{N}_1}. 
	\]
	So we have shown that $W_{\theta}\in \mathbf{M}_1$, and $\|W_{\theta}\|_{\mathbf{M}_1}\le C\|\xi_{\theta}\|_{\mathbf{N}_1}$ for some constant $C$. 
	
	By the definition of $W_{\phi}(\xi)$ in (\ref{sec31:eq:Wphi}) , using (\ref{sec32:eq:eb}) and the fact that $\xi_{\phi}\in \mathbf{N}_2$, we have, for every $-1<x<1$, 
	\[
	\begin{split}
		& \quad (1-x^2)^{\epsilon}|W_{\phi}(x)|  \le  (1-x^2)^{\epsilon}\int_{0}^{x}e^{-b(t)}\int_{0}^{t}e^{b(s)}\frac{|\xi_{\phi}(s)|}{1-s^2}dsdt\\
		& \le \|\xi_{\phi}\|_{\mathbf{N}_2}  (1-x^2)^{\epsilon}\int_{0}^{x}e^{-b(t)}\int_{0}^{t}e^{b(s)}(1-s^2)^{-2-\epsilon}dsdt \\
		& \le  C\|\xi_{\phi}\|_{\mathbf{N}_2} (1-x^2)^{\epsilon} \int_{0}^{x}(1+t)^{-\frac{\bar{U}_{\theta}(-1)}{2}}(1-t)^{\frac{\bar{U}_{\theta}(1)}{2}} \\ & \quad \cdot \int_{0}^{t}(1+s)^{\frac{\bar{U}_{\theta}(-1)}{2}-2-\epsilon}(1-s)^{-\frac{\bar{U}_{\theta}(1)}{2}-2-\epsilon}dsdt\\
		& \le C (1-x^2)^{\epsilon}\|\xi_{\phi}\|_{\mathbf{N}_2}\int_{0}^{x}(1+t)^{-1-\epsilon}(1-t)^{-1-\epsilon}dt\\
		& \le C\|\xi_{\phi}\|_{\mathbf{N}_2}. 
		\end{split}
	\]
	
	Using (\ref{sec31:eq:diff1:Wphi}), (\ref{sec32:eq:eb}) and the fact that $\xi_{\phi}\in \mathbf{N}_2$, we have, for every $-1<x<1$, 
	\begin{equation}\label{sec32:eq:Wphi:bdd:temp}
	\begin{split}
		& \quad |(1-x^2)^{1+\epsilon}W'_{\phi}(x)| \\
		& \le \|\xi_{\phi}\|_{\mathbf{N}_2}(1+x)^{-\frac{\bar{U}_{\theta}(-1)}{2}+1+\epsilon}(1-x)^{\frac{\bar{U}_{\theta}(1)}{2}+1+\epsilon}\int_{0}^{x} (1+s)^{\frac{\bar{U}_{\theta}(-1)}{2}-2-\epsilon}(1-s)^{-\frac{\bar{U}_{\theta}(1)}{2}-2-\epsilon}ds\\
		& \le C\|\xi_{\phi}\|_{\mathbf{N}_2}.
	\end{split}
	\end{equation}
	Similarly, since $|b'(x)|=\frac{|\bar{U}_{\theta}|}{1-x^2}$, using (\ref{sec31:eq:diff2:Wphi}),  (\ref{sec32:eq:Wphi:bdd:temp}) and the fact that $\xi_{\phi}\in \mathbf{N}_2$, we have
	\begin{equation*}
		|(1-x^2)^{2+\epsilon}W''_{\phi}(x)|\le C(1-x^2)^{1+\epsilon}|W'_{\phi}|+(1-x^2)^{1+\epsilon}|\xi_{\phi}|\le C\|\xi_{\phi}\|_{\mathbf{N}_2}.
	\end{equation*}
	Then $W(\xi)\in\mathbf{X}$ for all $\xi\in\mathbf{Y}$, and $\|W(\xi)\|_{\mathbf{X}}\le C\|\xi\|_{\mathbf{Y}}$ for some constant $C$. So $W:\mathbf{Y}\rightarrow\mathbf{X}$ is well-defined and continuous. 
	 
	By definition of $W^i$, $i=1,2a,2b$, we have $l_{c,\gamma}[W^i_{\theta}](x)=\xi_{\theta}$. So $(l_{c,\gamma}[W^i_{\theta}])''(0)=\xi''_{\theta}(0)=0$, then $l_{c,\gamma}[W^i_{\theta}](x)+\frac{1}{2}(l_{c,\gamma}[W^i_{\theta}])''(0)(1-x^2)=\xi_{\theta}$. 
	
	By definition of $W^3$, we have $l_{c,\gamma}[W^3_{\theta}](x)=\xi_{\theta} - C_W^{c,\gamma}(1-x^2)$. So $(l_{c,\gamma}[W^3_{\theta}])''(0)=\xi''_{\theta}(0)  + 2 C_W^{c,\gamma} =2 C_W^{c,\gamma}$, then $l_{c,\gamma}[W^3_{\theta}](x)+\frac{1}{2}(l_{c,\gamma}[W^3_{\theta}])''(0)(1-x^2)=\xi_{\theta}$.  Thus $L^{c,\gamma}_0W(\xi)=\xi$, $W$ is a right inverse of $L^{c,\gamma}_0$.
\end{proof}

Let $V_{c,\gamma}^i$, $1\le i\le 4$, $V_{c,\gamma}^{2a}$, $V_{c,\gamma}^{2b}$ be vectors defined by (\ref{eq_basis}), (\ref{sec32:eq:V2a}), (\ref{sec32:eq:V2b}), we have
\begin{lem}\label{sec32:lem:ker}
	\[
		\ker L^{c,\gamma}_{0}=\left\{
		\begin{array}{ll}
		\mathrm{span}\{V_{c,\gamma}^1, V_{c,\gamma}^2, V_{c,\gamma}^3, V_{c,\gamma}^4\}, & \textrm{if } (c,\gamma)\in I_{1,1},\\
		\mathrm{span}\{ V_{c,\gamma}^{2a}, V_{c,\gamma}^3, V_{c,\gamma}^4\}, & \textrm{if } (c,\gamma)\in I_{1,2} \textrm{ or }I_{2,2},\\
		\mathrm{span}\{ V_{c,\gamma}^{2b}, V_{c,\gamma}^3, V_{c,\gamma}^4\}, & \textrm{if } (c,\gamma)\in I_{1,3} \textrm{ or }I_{3,3},\\
		\mathrm{span}\{ V_{c,\gamma}^3, V_{c,\gamma}^4\}, & \textrm{if } (c,\gamma)\in I_{k,l} \textrm{ for } 5\le k \le 8, 1\le l\le 3.
		\end{array}
		\right.
	 \]
\end{lem}
\begin{proof}
	Let $V\in\mathbf{X}$ satisfy $L^{c,\gamma}_{0}V=0$. We know that $V$ is given by (\ref{sec31:eq:ker}) for some $d_1,d_2,d_3,d_4\in\mathbb{R}$.
	For convenience we denote $a(x)=a_{c,\gamma}(x)$, $b(x)=b_{c,\gamma}(x)$ and $V^i=V^i_{c,\gamma}$, $i=1,2,2a,2b,3,4$.
	
	By Lemma \ref{lem3_2_ab} and the expressions of $V^1,V^2$ in (\ref{eq_basis}), we have that
	\begin{equation}\label{eq_ker_1}
		V^1_{\theta}(x) = e^{-a(x)}= O(1)(1+x)^{1-\frac{\bar{U}_{\theta}(-1)}{2}}(1-x)^{1+\frac{\bar{U}_{\theta}(1)}{2}}, 
	\end{equation}
	and 
	\begin{equation}\label{eq_ker_2}
		V^2_{\theta}(x) =e^{-a(x)}\int_{0}^{x}e^{a(s)}ds
		= O(1) (1-x^2)\left((1+x)^{-\frac{\bar{U}_{\theta}(-1)}{2}}+1\right)\left((1-x)^{\frac{\bar{U}_{\theta}(1)}{2}}+1\right). 
	\end{equation}
	
Then by (\ref{sec31:eq:diff:a}), we also have
	\begin{equation}\label{eq_ker_3}
		 \left|\frac{d}{dx}V^1_{\theta}(x)\right|=\left|e^{-a(x)}a'(x)\right|
		=  O(1)(1+x)^{-\frac{\bar{U}_{\theta}(-1)}{2}}(1-x)^{\frac{\bar{U}_{\theta}(1)}{2}},
	\end{equation}
	\begin{equation}\label{eq_ker_4}
		 \left|\frac{d}{dx}V^2_{\theta}(x)\right|=\left| -V^2_{\theta}(x)a'(x) + 1\right| 
		=  O(1)\left((1+x)^{-\frac{\bar{U}_{\theta}(-1)}{2}}+1\right)\left((1-x)^{\frac{\bar{U}_{\theta}(1)}{2}}+1\right).
	\end{equation}
	
	If $\bar{U}_{\theta}(-1)> 2$ or $\bar{U}_{\theta}(-1)=2$ with $\eta_1=0$, by (\ref{sec32:eq:V2a}) and Lemma \ref{lem3_2_ab}, we have that 
	\begin{equation}\label{eq_ker_5}
		 \left|V^{2a}_{\theta}(x)\right|=\left|e^{-a(x)}\int_{-1}^{x}e^{a(s)}ds\right| 
		=  O(1) (1-x^2)(1+(1-x)^{\frac{\bar{U}_{\theta}(1)}{2}}),
	\end{equation}
	and 
	\begin{equation}\label{eq_ker_6}
	   \left|\frac{d}{dx}V^{2a}_{\theta}(x)\right|=\left| -V^{2a}_{\theta}(x)a'(x) + 1\right| 
				= O(1) (1+(1-x)^{\frac{\bar{U}_{\theta}(1)}{2}}).
	\end{equation}
	
	Similarly, we have if $\bar{U}_{\theta}(1)<- 2$ or $\bar{U}_{\theta}(1)=-2$ with $\eta_2=0$, 
	\begin{equation*}
	   \left|V^{2b}_{\theta}(x)\right|=  O(1) (1-x^2)(1+(1+x)^{-\frac{\bar{U}_{\theta}(-1)}{2}}),\quad  \left|\frac{d}{dx}V^{2b}_{\theta}(x)\right|
				= O(1) (1+(1+x)^{-\frac{\bar{U}_{\theta}(-1)}{2}}).
	\end{equation*}
	Next, by computation we have for $i=1,2,2a,2b$
	\[
	 \frac{d^2}{dx^2}V^i_{\theta}=(V^i_{\theta})'a'(x)+V^i_{\theta}a''(x),\quad  \frac{d^3}{dx^3}V^i_{\theta}=(V^i_{\theta})''a'(x)+2(V^i_{\theta})'a''(x)+V^i_{\theta}a'''(x).
	\]
	Using the definition of $a(x)$ in (\ref{sec31:eq:ab}), there exists some constant $C$, depending on $c,\gamma$, such that
	\begin{equation}\label{eq_ker_8}
	   \left |\frac{d^2}{dx^2}V^i_{\theta}\right|\le C, \quad \left|\frac{d^3}{dx^3}V^i_{\theta}\right|\le C, \quad -\frac{1}{2}<x<\frac{1}{2},\; i=1,2,2a,2b.
	\end{equation}
	
	Moreover, by Lemma \ref{lem3_2_ab}, and the expressions of $V^3$ in (\ref{eq_basis}), we have
	\begin{equation}\label{eq_ker_9}
		V^3_{\phi}(x)=\int_{0}^{x}e^{-b(t)}dt=O(1)(1+x)^{1-\frac{\bar{U}_{\theta}(-1)}{2}}(1-x)^{1+\frac{U^{c,\gamma}(1)}{2}}+O(1),
	\end{equation}
	and 
	\begin{equation}\label{eq_ker_10}
	\begin{split}
		&\left|\frac{d}{dx}V^3_{\phi}(x)\right|=e^{-b(x)}= O(1)(1+x)^{-\frac{\bar{U}_{\theta}(-1)}{2}} (1-x)^{\frac{\bar{U}_{\theta}(1)}{2}},\\
			&\left|\frac{d^2}{dx^2}V^3_{\phi}(x)\right|=e^{-b(x)}|b'(x)| = O(1)(1+x)^{-1-\frac{\bar{U}_{\theta}(-1)}{2}} (1-x)^{-1+\frac{\bar{U}_{\theta}(1)}{2}}. 
			\end{split}
	\end{equation}

	When $(c,\gamma)\in I_{1,1}$, $\bar{U}(-1)<2$ and $\bar{U}(1)>-2$, using estimates (\ref{eq_ker_1})-(\ref{eq_ker_4}),  (\ref{eq_ker_8}), (\ref{eq_ker_9}),  (\ref{eq_ker_10}), and the definition of $V_{c,\gamma}^4$, it is not hard to verify that $V_{c,\gamma}^i\in \mathbf{X}$, $1\le i\le 4$.  It is clear that $\{V_{c,\gamma}^i , 1\le i\le 4\}$ are independent. So $\{V_{c,\gamma}^i , 1\le i\le 4\}$ is a basis of the kernel.
	
	Similarly, when $(c,\gamma)\in I_{1,2} \textrm{ or }I_{2,2}$, it can be checked that $\mathrm{span}\{V_{c,\gamma}^1, V_{c,\gamma}^2\}=\mathrm{span}\{V_{c,\gamma}^1, V_{c,\gamma}^{2a}\}$, where $V_{c,\gamma}^{2a}$, given by (\ref{sec32:eq:V2a}), is a linear combination of $V_{c,\gamma}^1, V_{c,\gamma}^2$.  So $L^{c,\gamma}_{0}V=0$ implies 
	$$
		V = d_1 V_{c,\gamma}^1 + d^a_2 V_{c,\gamma}^{2a} + d_3 V_{c,\gamma}^3 + d_4 V_{c,\gamma}^4,
	$$
	for some constants $d_1,d_2^a, d_3$ and $d_4$. 
	It can be checked by estimates (\ref{eq_ker_1}), (\ref{eq_ker_3}), (\ref{eq_ker_5}), (\ref{eq_ker_6}),  (\ref{eq_ker_8}), (\ref{eq_ker_9}),  (\ref{eq_ker_10}) that in this case $V_{c,\gamma}^{2a}, V_{c,\gamma}^3, V_{c,\gamma}^4 \in \mathbf{X}$, and  $  V_{c,\gamma}^1 \notin \mathbf{X}$.  So $d_1 V_{c,\gamma}^1\in \mathbf{X}$. This means $d_1 (V_{c,\gamma}^1)_{\theta}\in \mathbf{M}_1$ Thus $d_1=0$.

	When $(c,\gamma)\in I_{1,3} \textrm{ or }I_{3,3}$, similarly as the proof of the previous case, we have that 
	$$
		V =  d^b_2 V_{c,\gamma}^{2b} + d_3 V_{c,\gamma}^3+d_4 V_{c,\gamma}^4,
	$$
	for some constants $d_2^b,d_3,d_4$, and $V_{c,\gamma}^{2b}$ is given by (\ref{sec32:eq:V2b}).
	
	When $(c,\gamma)\in I_{k,l} \textrm{ for } 5\le k \le 8, \textrm{ and }1\le l\le 3$, by (\ref{eq_ker_1})-(\ref{eq_ker_2}),  (\ref{eq_ker_8}), (\ref{eq_ker_9}),  and (\ref{eq_ker_10}), we have $ V_{c,\gamma}^3, V_{c,\gamma}^4 \in \mathbf{X}$, and  $  V_{c,\gamma}^1, V_{c,\gamma}^2 \notin \mathbf{X}$. If there exists some $d_1,d_2\in \mathbb{R}$, such that $\hat{V}_{\theta}:=d_1 V_{c,\gamma}^1+d_2 V_{c,\gamma}^{2}\in \mathbf{X}$, then by $\hat{V}_{\theta}(-1)=\hat{V}_{\theta}(1)=0$, we have
	\[
	    d_1=-d_2\int_{0}^{1}e^{a_{c,\gamma}(s)}ds=-d_2\int_{0}^{-1}e^{a_{c,\gamma}(s)}ds.
	\]
	This means $d_2\int_{-1}^{1}e^{a_{c,\gamma}(s)}ds=0$, thus $d_2=0$, and $d_1=0$. 	
	%
	%
	%
	%
	The lemma is proved.
\end{proof}
\begin{cor}\label{sec32:cor:all:sol}
	For any $\xi\in\mathbf{Y}$, all solutions of $L^{c,\gamma}_{0}V=\xi$, $V\in\mathbf{X}$, are given by
	\[
		V=W^{c,\gamma}(\xi)+\left\{
		\begin{array}{ll}
			d_1V_{c,\gamma}^1+d_2V_{c,\gamma}^2+d_3V_{c,\gamma}^3+d_4 V_{c,\gamma}^4, & \textrm{if }(c,\gamma)\in I_{1,1},\\
			d_2V_{c,\gamma}^{2a}+d_3V_{c,\gamma}^3+d_4 V_{c,\gamma}^4, & \textrm{if }(c,\gamma)\in I_{1,2} \textrm{ or }I_{2,2},\\
			d_2V_{c,\gamma}^{2b}+d_3V_{c,\gamma}^3+d_4 V_{c,\gamma}^4, & \textrm{if }(c,\gamma)\in I_{1,3} \textrm{ or }I_{3,3},\\
			d_3V_{c,\gamma}^3+d_4 V_{c,\gamma}^4, & \textrm{if }(c,\gamma)\in I_{k,l}, 5\le k \le 8, 1\le l\le 3.
		\end{array}
		\right.
	\]
\end{cor}

Let $l_1,l_2,l_3,l_4$ be the functionals on $\mathbf{X}$ defined by (\ref{sec31:eq:fcnal:l}), and $\mathbf{X}_i$, $i=1,2a,2b,3$, be the subspaces of $\mathbf{X}$ defined by (\ref{sec32:eq:X}), (\ref{sec32:eq:Xab}) and (\ref{sec32:eq:X3}). As shown in Section \ref{sec_2}, the matrix $(l_i(V^{j}_{c,\gamma}))$ is invertible, for every $(c,\gamma)\in K$. So $\mathbf{X}_i$ is a closed subspace of $\mathbf{X}$, and
\begin{equation}\label{sec32:eq:decomp:X}
	\mathbf{X} =\left\{
	\begin{array}{ll}
		\mbox{span} \{ V_{c,\gamma}^{1}, V_{c,\gamma}^{2}, V_{c,\gamma}^{3},V_{c,\gamma}^4 \} \oplus\mathbf{X}_1,& \quad (c,\gamma)\in I_{1,1},\\
		\mbox{span} \{ V_{c,\gamma}^{2a}, V_{c,\gamma}^{3}, V_{c,\gamma}^4 \} \oplus\mathbf{X}_{2a},& \quad (c,\gamma)\in I_{1,2} \textrm{ or }I_{2,2},\\
		\mbox{span} \{ V_{c,\gamma}^{2b}, V_{c,\gamma}^{3}, V_{c,\gamma}^4\} \oplus\mathbf{X}_{2b},& \quad (c,\gamma)\in I_{1,3} \textrm{ or }I_{3,3},\\
		\mbox{span} \{ V_{c,\gamma}^{3}, V_{c,\gamma}^4 \} \oplus\mathbf{X}_3,& \quad (c,\gamma)\in I_{k,l}, 5\le k \le 8, 1\le l\le 3,\\
	\end{array}
	\right. 
\end{equation}
with the projection operator $P_i: \mathbf{X}\rightarrow\mathbf{X}_i$  for $i=1,2a,2b,3$ given by 
\begin{equation}\label{sec32:eq:proj}
\begin{split}
	&P_1V = V -  l_1(V)V_{c,\gamma}^{1} -\left(\frac{l_2(V) - l_1(V) l_2(V_{c,\gamma}^{1})}{l_2 (V_{c,\gamma}^{2})} \right) V_{c,\gamma}^{2} - l_3(V)V_{c,\gamma}^{3} - l_4(V)V_{c,\gamma}^{4}, \\
	& P_{2a}V = V  - \frac{l_1(V)}{l_1(V_{c,\gamma}^{2a})}V_{c,\gamma}^{2a} - l_3(V)V_{c,\gamma}^{3}- l_4(V)V_{c,\gamma}^{4}, \\
	& P_{2b}V = V  - \frac{l_1(V)}{l_1(V_{c,\gamma}^{2b})}V_{c,\gamma}^{2b} - l_3(V)V_{c,\gamma}^{3}- l_4(V)V_{c,\gamma}^{4}, \\
	&P_3V = V - l_3(V)V_{c,\gamma}^{3}- l_4(V)V_{c,\gamma}^{4},
\end{split}
\end{equation}
for all $V\in \mathbf{X}$.
\begin{lem}\label{sec32:lem:iso}
	If $(c,\gamma)\in I_{1,1}$, the operator $ L^{c,\gamma}_{0}: \mathbf{X}_1\rightarrow\mathbf{Y}$ is an isomorphism.  
	
	If $(c,\gamma)\in I_{1,2} \textrm{ or }I_{2,2}$, the operator $ L^{c,\gamma}_{0}: \mathbf{X}_{2a}\rightarrow\mathbf{Y}$ is an isomorphism. 
	
	If $(c,\gamma)\in I_{1,3} \textrm{ or }I_{3,3}$, the operator $ L^{c,\gamma}_{0}: \mathbf{X}_{2b}\rightarrow\mathbf{Y}$ is an isomorphism. 
	
	If $(c,\gamma)\in I_{k,l} \textrm{ for } 5\le k \le 8 \textrm{ and }1\le l\le 3$, the operator $ L^{c,\gamma}_{0}: \mathbf{X}_{3}\rightarrow\mathbf{Y}$ is an isomorphism. 
\end{lem}

\begin{proof}
	By Corollary \ref{sec32:cor:all:sol} and Lemma \ref{sec32:lem:ker}, $L^{c,\gamma}_{0}:\mathbf{X}\rightarrow\mathbf{Y}$ is surjective and $\ker L^{c,\gamma}_{0}$  is given by Lemma \ref{sec32:lem:ker}. The conclusion of the lemma then follows in view of the direct sum property (\ref{sec32:eq:decomp:X}).	
\end{proof}

\begin{lem}\label{sec32:lem:V:smooth}
	$V_{c,\gamma}^1, V_{c,\gamma}^2\in C^{\infty}(K,\mathbf{X})$ for compact $K\subset I_{1,1}$.
	
	$V_{c,\gamma}^{2a} \in C^{\infty}(K,\mathbf{X})$  for compact $K\subset I_{1,2}$ or $I_{2,2}$. 

	$V_{c,\gamma}^{2b} \in C^{\infty}(K,\mathbf{X})$for compact $K\subset I_{1,3}$ or $I_{3,3}$. 
	
	$V_{c,\gamma}^3, V_{c,\gamma}^4 \in C^{\infty}(K,\mathbf{X})$ for compact $K\subset I_{k,l}$ with $k=1$ or $5\le k\le 8$ or $(k,l)=(2,2), (3,3)$. 
\end{lem}
\begin{proof}
    It is clear that $V^4_{c,\gamma}\in  C^{\infty}(K,\mathbf{X})$ for all compact set $K$ described as in the lemma.
    
    Let $\alpha=(\alpha_1,\alpha_2,\alpha_3)$ be a multi-index where $\alpha_i\ge 0$, $i=1,2,3$, and $j\ge 0$. For convenience we denote $a(x)=a_{c,\gamma}(x)$, $b(x)=b_{c,\gamma}(x)$ and $V^i=V^i_{c,\gamma}$, $i=1,2,2a,2b,3$.
    
		 
		 Using Proposition \ref{propA_1} part (1), we have that for all $|\alpha|+j\ge 1$ and $(c,\gamma)\in K$,
		 \begin{equation}\label{eq3_6_3}
		    \partial_c^{\alpha}\partial_{\gamma}^j a(x)=\partial_c^{\alpha}\partial_{\gamma}^j b(x)=\int_{0}^{x}\frac{1}{1-s^2} \pt_c^\alpha \pt_\gamma^j U^{c,\gamma}(s)ds = O(1)|\ln(1-x^2)|.
		 \end{equation}
		 
	(1) If $K\subset I_{1,1}$, we have $U_{\theta}^{c,\gamma}(-1)<2$ and $U_{\theta}^{c,\gamma}(1)>-2$. 

	Let $2\bar{\epsilon}:=\max\{0, \frac{1}{2}U^{c,\gamma}_{\theta}(-1), -\frac{1}{2}U^{c,\gamma}_{\theta}(1) \mid (c,\gamma)\in K\}$, then $\bar{\epsilon}<\epsilon$.

	Using the expressions of $V^1,V^2$ in (\ref{eq_basis}), Lemma \ref{lem3_2_ab}, estimates (\ref{eq_ker_1}), (\ref{eq_ker_2}), (\ref{eq3_6_3}) and Proposition \ref{propA_1},  we have that for all $|\alpha|+j\ge 1$ and $(c,\gamma)\in K$, 
	\[
	\begin{split}
		& \left| \pt_c^\alpha \pt_\gamma^j V^1_{\theta}(x) \right| = e^{-a(x)}O\left( \left|\ln(1-x^2)\right|^{|\alpha|+j} \right)=O(1)(1-x^2)^{1-2\bar{\epsilon}}\left|\ln(1-x^2)\right|^{|\alpha|+j},  \\
		& \left| \pt_c^\alpha \pt_\gamma^j V^2_{\theta}(x) \right| = \left|e^{-a(x)}\int_{0}^{x}e^{a(s)}ds\right| O\left( \left|\ln(1-x^2)\right|^{|\alpha|+j} \right)=O(1)(1-x^2)^{1-2\bar{\epsilon}}\left|\ln(1-x^2)\right|^{|\alpha|+j},
		\end{split}
	\]
	and 
	\[
	\begin{split}
	  &  \left| \frac{d}{dx}\pt_c^\alpha \pt_\gamma^j V^1_{\theta}(x) \right|=e^{-a(x)}|a'(x)|O\left( \left|\ln(1-x^2)\right|^{|\alpha|+j} \right)=O(1)(1-x^2)^{-2\bar{\epsilon}}\left|\ln(1-x^2)\right|^{|\alpha|+j},\\
	  & \left| \frac{d}{dx}\pt_c^\alpha \pt_\gamma^j V^2_{\theta}(x) \right|=|-V^2_{\theta}(x)a'(x)+1|O\left( \left|\ln(1-x^2)\right|^{|\alpha|+j} \right)=O(1)(1-x^2)^{-2\bar{\epsilon}}\left|\ln(1-x^2)\right|^{|\alpha|+j}.
	    \end{split}
	\]
	From the above we can see that for all $|\alpha|+j\ge 1$, there exists some constant $C=C(\alpha,j,K)$, such that $(c,\gamma)\in K$, 
	\[
		\left|(1-x^2)^{-1+2\epsilon}\pt_c^\alpha \pt_\gamma^j  V^i_{\theta}(x)\right|\le C,\quad \left|(1-x^2)^{2\epsilon}\frac{d}{dx}\pt_c^\alpha \pt_\gamma^j V^i_{\theta}(x)\right|\le C, \quad i=1,2.
	\]
	From the above we also have that for $|\alpha|+j \ge 1$, 
	\[
		\pt_c^\alpha \pt_\gamma^j V^i_{\theta}(1) = \pt_c^\alpha \pt_\gamma^j V^i_{\theta}(-1) = 0, \quad i=1,2.
	\]
	Next, 
	using the definition of $a(x)$ in (\ref{sec31:eq:ab}), there exists some constant $C=C(K)$, such that  
	\[
	    \left| \frac{d^2}{dx^2}\pt_c^\alpha \pt_\gamma^jV^i_{\theta}\right|\le C, \quad \left| \frac{d^3}{dx^3}\pt_c^\alpha \pt_\gamma^jV^i_{\theta}\right|\le C, \quad -\frac{1}{2}<x<\frac{1}{2},\quad i=1,2.
	\]
	The above implies that  for all $|\alpha|+j \ge 1$, $\pt_c^\alpha \pt_\gamma^j  V^i_{\theta}\in \mathbf{M}_1$, $i=1,2$, so $V^1,V^2\in C^{\infty}(K, \mathbf{X})$.	
	
	(2) If $K\subset I_{1,2}$ or $I_{2,2}$, we have $2\le U_{\theta}^{c,\gamma}(-1)<3$ with $U^{c,\gamma}(1)>-2$. 
	
	Let $2\bar{\epsilon} := \max\{U^{c,\gamma}_{\theta}(-1)-2, -\frac{1}{2}U^{c,\gamma}_{\theta}(1) \mid (c,\gamma)\in K\}$, then $\bar{\epsilon} <\epsilon$. 
	
	
	 In this case $\gamma = \gamma^+(c_1,c_2,c_3)$. Using the expressions of $V^{2a}$ in (\ref{sec32:eq:V2a}), Lemma \ref{lem3_2_ab},  the estimates (\ref{eq_ker_5}), (\ref{eq_ker_6}), (\ref{eq3_6_3}) and Proposition \ref{propA_1}, we have that for all $|\alpha|\ge 1$, 
	\[
		\left| \pt_c^\alpha V^{2a}_{\theta}(x) \right| 
		 = O(1)(1+x)(1-x)^{1-2\bar{\epsilon}}\left|\ln(1-x^2)\right|^{|\alpha|}, 
	\]
	and 
	\[
	    \left| \frac{d}{dx}\pt_c^\alpha V^{2a}_{\theta}(x) \right|= O(1)(1-x)^{-2\bar{\epsilon}}\left|\ln(1-x^2)\right|^{|\alpha|}.
	\]
	From the above we can see that  for any $|\alpha|\ge 1$, there exists some constant $C=C(\alpha,K)$, such that for all $(c,\gamma)\in K$
	\[
		\left|(1-x^2)^{-1+2\epsilon}\pt^\alpha_{c} V^{2a}_{\theta}(x)\right|\le C,\quad \left|(1-x^2)^{2\epsilon}\frac{d}{dx}\pt^\alpha_c V^{2a}_{\theta}(x)\right|\le C. 
	\]
	We also have that for $|\alpha| \ge 1$, 
	\[
		\pt^\alpha_{c} V^{2a}_{\theta}(1) = \pt^\alpha_{c} V^{2a}_{\theta}(-1) = 0. 
	\]
	Similarly as part (1), we have
	\[
	    \left| \frac{d^2}{dx^2}\pt_c^\alpha \pt_\gamma^jV^{2a}_{\theta}\right|\le C, \quad \left| \frac{d^3}{dx^3}\pt_c^\alpha \pt_\gamma^jV^{2b}_{\theta}\right|\le C, \quad -\frac{1}{2}<x<\frac{1}{2}.
	\]
	The above implies that  for all $|\alpha|\ge 1$, $\pt^\alpha_{c} V^{2a}_{\theta}\in \mathbf{M}_1$, so $V^{2a}\in C^{\infty}(K,\mathbf{X})$. 
	
	(3) If $K\subset I_{1,3}$ or $I_{3,3}$, then by similar argument as part (2), we have that $V^{2b}\in C^{\infty}(K,\mathbf{X})$.
	
	(4) Let $K$ be a subset of $I_{k,l}$ with $k=1$ or $5\le k\le 8$ or $(k,l)=(2,2), (3,3)$. 
	Using the expressions of $V^3$ in (\ref{eq_basis}), Lemma \ref{lem3_2_ab},  the estimates (\ref{eq_ker_9}), (\ref{eq_ker_10}), (\ref{eq3_6_3}) and Proposition \ref{propA_1}, we have that for all $|\alpha|+j\ge 1$,
	\[
	   \left|\partial_{c}^{\alpha}\partial_{\gamma}^{j}V^3_{\phi}\right|=O(1)\left((1+x)^{1-\frac{U^{c,\gamma}_{\theta}(-1)}{2}}(1-x)^{1+\frac{U^{c,\gamma}(1)}{2}}+1\right)|\ln(1-x^2)|^{|\alpha|+j},
	\]
	\[
	    \left|\frac{d}{dx}\partial_{c}^{\alpha}\partial_{\gamma}^{j}V^3_{\phi}(x)\right|= O(1)(1+x)^{-\frac{U^{c,\gamma}_{\theta}(-1)}{2}} (1-x)^{\frac{U^{c,\gamma}_{\theta}(1)}{2}}|\ln(1-x^2)|^{|\alpha|+j},
	\]
	\[
	    \left|\frac{d^2}{dx^2}\partial_{c}^{\alpha}\partial_{\gamma}^{j}V^3_{\phi}(x)\right|=e^{-b(x)}|b'(x)| = O(1)(1+x)^{-1-\frac{U^{c,\gamma}_{\theta}(-1)}{2}} (1-x)^{-1+\frac{U^{c,\gamma}_{\theta}(1)}{2}}|\ln(1-x^2)|^{|\alpha|+j}.
	\]
	Since $\epsilon>\max\{0,\frac{U^{c,\gamma}_{\theta}(-1)}{2}-1,-1-\frac{U^{c,\gamma}(1)}{2}\}$, there exists some $C=C(\alpha, j, K)$ such that for all  $(c,\gamma)\in K$,
	\[
	   | (1-x^2)^{\epsilon}\partial_{c}^{\alpha}\partial_{\gamma}^{j}V^3_{\phi}|\le C, \quad \left|(1-x^2)^{1+\epsilon}\frac{d}{dx}\partial_{c}^{\alpha}\partial_{\gamma}^{j}V^3_{\phi}\right|\le C, \quad \quad \left|(1-x^2)^{2+\epsilon}\frac{d^2}{dx^2}\partial_{c}^{\alpha}\partial_{\gamma}^{j}V^3_{\phi}\right|\le C.
	\]
	The above implies that for any $|\alpha|+j\ge 1$, $\partial_{c}^{\alpha}\partial_{\gamma}^{j}V^3_{\phi}\in \mathbf{M}_2$, so $V^3\in C^{\infty}(K,\mathbf{X})$. 
\end{proof}

\begin{lem}\label{sec32:lem:combine:bdd}
	(i) If $ K\subset \subset I_{1,1}$,  then there exists $C=C(K)>0$ such that for all $(c,\gamma)\in K$, $\beta:=(\beta_1,\beta_2,\beta_3,\beta_4)\in \mathbb{R}^4$, and $V\in \mathbf{X}_1$, 
	\[
		\| V \|_{\mathbf{X}} + | \beta | 
		\leq C \| \sum_{i=1}^{4}\beta_i V^i_{c,\gamma} + V \|_{\mathbf{X}}. 
	\] 

	(ii) If $K\subset \subset I_{1,2}$ or $I_{2,2}$, then there exists $C=C(K)>0$ such that for all $(c,\gamma)\in K$,  $(\beta_2,\beta_3,\beta_4)\in \mathbb{R}^3$, and $V\in \mathbf{X}_{2a}$, 
	\[
		\| V \|_{\mathbf{X}} + | (\beta_2, \beta_3,\beta_4) | 
		\leq C \| \beta_2 V^{2a}_{c,\gamma} + \beta_3 V^3_{c,\gamma} +\beta_4V^4_{c,\gamma}+ V \|_{\mathbf{X}}. 
	\] 

	(iii) If $K\subset \subset I_{1,3}$ or $I_{3,3}$, then there exists $C=C(K)>0$ such that for all $(c,\gamma)\in K$,  $(\beta_2,\beta_3,\beta_4)\in \mathbb{R}^3$, and $V\in \mathbf{X}_{2b}$, 
	\[
		\| V \|_{\mathbf{X}} + | (\beta_2, \beta_3,\beta_4) | 
		\leq C \| \beta_2 V^{2b}_{c,\gamma} + \beta_3 V^3_{c,\gamma} +\beta_4V^4_{c,\gamma}+ V \|_{\mathbf{X}}. 
	\] 

	(iv)If $ K\subset \subset I_{k,l}$ with $5\le k\le 8$, $1\le l\le 3$,  then there exists $C=C(K)>0$ such that for all $(c,\gamma)\in K$,  $(\beta_3,\beta_4)\in \mathbb{R}^2$, and $V\in \mathbf{X}_3$, 
	\[
		\| V \|_{\mathbf{X}} + | (\beta_3,\beta_4) | 
		\leq C \| \beta_3 V^3_{c,\gamma} +\beta_4V^4_{c,\gamma}+ V \|_{\mathbf{X}}. 
	\] 
\end{lem}
\begin{proof}
	We only prove (i). Similar arguments yield (ii), (iii) and (iv). 
	
	We use contradiction argument. Assume there exists a sequence $(c^i, \gamma^i)\in K$,  $\beta^i:=(\beta_1^i, \beta_2^i, \beta_3^i,\beta_4^i)\in \mathbb{R}^4$, and $V^i \in \mathbf{X}_1$, such that
	\begin{equation}\label{sec32:eq:combine:bdd}
		\|V^i\|_{\mathbf{X}} + | \beta^i | 
		\ge i \|\sum_{j=1}^{4}\beta_j^iV^j_{c^i,\gamma^i} + V^i \|_{\mathbf{X}}.
	\end{equation}
	Without loss of generality we assume that
	\[
		\| V^i \|_{\mathbf{X}} + | \beta^i  | = 1. 
	\] 
	Since $|\beta^i  | \le 1$, there exists a subsequence, still denote as $\beta^i $, and some $\beta\in \mathbb{R}^4$, such that $\beta^i  \to \beta$ as $i \to \infty$. 
	
	Since $K$ is compact, there exists a subsequence of $(c^i,\gamma^i)$, still denoted as $(c^i,\gamma^i)$, and some $(c,\gamma)\in K$ such that $(c^i, \gamma^i) \to (c,\gamma)\in K$ as $i \to \infty$.  Then by Lemma \ref{sec32:lem:V:smooth} we have
	\[
		V^j_{c^i, \gamma^i}\to V^j_{c,\gamma}, \quad 1\le j\le 4. 
	\]
	By (\ref{sec32:eq:combine:bdd}),
	\[
		\sum_{j=1}^{4}\beta^i_j V^j_{c^i, \gamma^i} + V^i \to 0.
	\]
	This implies
	\[
		V^i \to V:= - \sum_{j=1}^{4}\beta^i_j V^j_{c^i, \gamma^i}. 
	\]
	On the other hand, $V^i\in \mathbf{X}_1$. 
	Since $\mathbf{X}_1$ is a closed subspace of $\mathbf{X}$,  we have $V\in \mathbf{X}_1$.
	Thus $V\in \mathbf{X}_1\cap \mbox{span} \{ V_{c,\gamma}^{1}, V_{c,\gamma}^{2}, V_{c,\gamma}^{3}, V^4_{c,\gamma} \}$. So $V=0$. Since $V^1_{c,\gamma}, V^2_{c,\gamma}, V^3_{c,\gamma}, V^4_{c,\gamma}$ are independent for any $(c,\gamma)\in K$. We have $\beta_1=\beta_2=\beta_3=\beta_4=0$. However, $\|V^i\|_{\mathbf{X}}+|(\beta^i_1,\beta^i_2,\beta^i_3,\beta^i_4)|=1$ leads to $\|V\|_{\mathbf{X}}+|(\beta_1,\beta_2,\beta_3,\beta_4)|=1$, contradiction. The lemma is proved. 
\end{proof}

\noindent {\bf Proof of Theorem \ref{sec32:thm:1}.}
Define a map $F: K \times \mathbb{R}^4 \times \mathbf{X}_1 \to \mathbf{Y}$ by
\[
	F(c,\gamma,\beta, V) = G(c,\gamma, \sum_{i=1}^{4}\beta_i V^i_{c,\gamma} + V).
\]
By Proposition \ref{sec32:prop}, $G$ is a $C^{\infty}$ map from $K\times \mathbf{X}$ to $\mathbf{Y}$. Let $\tilde{U} = \tilde{U}(c,\gamma,\beta, V) = \sum_{i=1}^{4}\beta_i V^i_{c,\gamma}+ V$. Using Lemma \ref{sec32:lem:V:smooth}, we have $\tilde{U}\in C^{\infty}(K \times \mathbb{R}^4 \times \mathbf{X}_1, \mathbf{X})$. So it concludes that $F \in C^{\infty}(K \times \mathbb{R}^4 \times \mathbf{X}_1, \mathbf{Y})$. 

Next, by definition $F(c,\gamma,0, 0) = 0$ for all $(c,\gamma)\in K$. Fix some $(\bar{c}, \bar{\gamma}) \in K$, using Lemma \ref{sec32:lem:iso}, we have $F_{V}(\bar{c}, \bar{\gamma}, 0,0)=L_0^{\bar{c}, \bar{\gamma}}:\mathbf{X}_1\to \mathbf{Y}$ is an isomorphism.

Applying Theorem \ref{thm:IFT}, there exist some $\delta>0$ depending only on $K$ and a unique $V\in C^{\infty}(B_{\delta}(\bar{c}, \bar{\gamma})\times B_{\delta}(0), \mathbf{X}_1 )$, such that
\[
	F(c,\gamma,\beta, V(c,\gamma,\beta) ) = 0, \quad \forall (c,\gamma)\in B_{\delta}(\bar{c}, \bar{\gamma}), \beta \in B_{\delta}(0),
\]
and 
\[
	V(\bar{c}, \bar{\gamma},0)=0.
\]
The uniqueness part of Theorem \ref{thm:IFT} holds in the sense that there exists some $0<\bar{\delta}<\delta$,  such that $B_{\bar{\delta}}(\bar{c},\bar{\gamma},0,0)\cap F^{-1}(0) \subset  \{(c,\gamma,\beta,V(c,\gamma,\beta) ) |(c,\gamma) \in B_{\delta}(\bar{c}, \bar{\gamma}), \beta\in B_{\delta}(0)\}$.

\noindent
\textbf{Claim.} there exists some $0 < \delta_1 < \frac{\bar{\delta}}{2}$, such that $V(c,\gamma,0)=0$ for all $(c,\gamma)\in B_{\delta_1}(\bar{c}, \bar{\gamma})$.

\noindent
{\bf Proof of the Claim.}
Since $V(\bar{c}, \bar{\gamma},0) = 0$ and $V(c,\gamma,0)$ is continuous in $(c,\gamma)$, there exists some $0<\delta_1 < \frac{\bar{\delta}}{2}$, such that for all $(c,\gamma)\in B_{\delta_1}(\bar{c}, \bar{\gamma})$, $(c,\gamma,0, V(c,\gamma,0))\in B_{\bar{\delta}(\bar{c}, \bar{\gamma},0,0)}$. We know that for all $(c,\gamma)\in B_{\delta_1}(\bar{c},\bar{\gamma})$,
\[
	F(c,\gamma,0, 0) = 0, 
\]
and 
\[
	F(c,\gamma,0, V(c,\gamma,0))=0.
\]
By the above mentioned uniqueness result, $V(c,\gamma,0)=0$, for every $(c,\gamma)\in B_{\delta_1}(\bar{c}, \bar{\gamma})$.

Now we have $V\in C^{\infty}(B_{\delta_1}(\bar{c}, \bar{\gamma})\times B_{\delta_1}(0), \mathbf{X}_1(\bar{c}, \bar{\gamma}) )$, and 
\[
	F(c,\gamma,\beta,V(c,\gamma,\beta))=0, \quad \forall (c,\gamma)\in B_{\delta_1}(\bar{c}, \bar{\gamma}), \beta\in B_{\delta_1}(0),
\]
i.e. for any $(c,\gamma)\in B_{\delta_1}(\bar{c}, \bar{\gamma})$, $\beta\in B_{\delta_1}(0)$
\[
	G(c, \gamma, \sum_{i=1}^{4}\beta_i V^i_{c,\gamma} + V(c,\gamma,\beta) )=0. 
\]
Take derivative of the above with respect to $\beta_i$ at $(c,\gamma,0)$, $1\le i\le 4$, we have
\[
	G_{\tilde{U}}(c,\gamma,0)(V^i_{c,\gamma}+\partial_{\beta_i}V(c,\gamma,0))=0.
\]
Since $G_{\tilde{U}}(c,\gamma,0)V^i_{c,\gamma}=0$ by Lemma \ref{sec32:lem:ker}, we have 
\[
	G_{\tilde{U}}(c,\gamma,0)\partial_{\beta_i}V(c,\gamma,0)=0.
\]
But $\partial_{\beta_i}V(c,\gamma,0)\in C^{\infty}(\mathbf{X}_1)$, so 
\[
	\partial_{\beta_i}V(c,\gamma,0)=0, \quad 1\le i\le 4.
\]
Since $K$ is compact, we can take $\delta_1$ to be a universal constant for each  $(c,\gamma)\in K$. So we have proved the existence of $V$ in Theorem \ref{sec32:thm:1}.

Next, let $(c,\gamma)\in B_{\delta_1}(\bar{c}, \bar{\gamma})$. Let $\delta'$ be a small constant to be determined.  For any $U$ satisfies equation (\ref{sec31:eq:NSE}) with $U-U^{c,\gamma}\in \mathbf{X}$, and $\|U-U^{c,\gamma}\|_{ \mathbf{X}}\le \delta'$ there exist some $\beta\in \mathbb{R}^4$ and $V^*\in \mathbf{X}_1$ such that
\[
	U-U^{c,\gamma} = \sum_{i=1}^{4}\beta_i V^i_{c,\gamma}+ V^*.
\]
Then by Lemma \ref{sec32:lem:combine:bdd}, there exists some constant $C>0$ such that
\[
	\frac{1}{C}(|\beta| + \|V^*\|_{\mathbf{X}})
	\le \|\sum_{i=1}^{4}\beta_i V^i_{c,\gamma} + V^*\|_{\mathbf{X}}\le \delta'.
\]
This gives $\|V^*\|_{\mathbf{X}}\le C\delta'$.

Choose $\delta'$ small enough such that $C\delta'<\delta_1$. We have the uniqueness of $V^*$. 
So  $V^*=V(c,\gamma,\beta)$ in (\ref{sec32:eq:U1}).
The theorem is proved.
\qed

Theorem \ref{sec32:thm:2}, \ref{sec32:thm:2'} and Theorem \ref{sec32:thm:3} can be proved by replacing $\mathbf{X}_1$ by $\mathbf{X}_{2a}$, $\mathbf{X}_{2b}$, $\mathbf{X}_3$, replacing $\sum_{i=1}^{4}\beta_i V^i_{c,\gamma}$ by $\beta_2 V^{2a}_{c,\gamma}+ \beta_3 V^3_{c,\gamma}+\beta_4V^4_{c,\gamma}$, $\beta_2 V^{2b}_{c,\gamma}+ \beta_3 V^3_{c,\gamma}+\beta_4V^4_{c,\gamma}$ and $\beta_3 V^3_{c,\gamma}+\beta_4V^4_{c,\gamma}$ respectively.

\texorpdfstring{}

\section{Existence of axisymmetric, with swirl solutions around \texorpdfstring{$\bm{U^{c,\gamma}}$}{}, when \texorpdfstring{$\bm{(c,\gamma)\in I_{k,l}}$}{} with \texorpdfstring{$\bm{(k,l)\in A_2}$}{} or \texorpdfstring{$\bm{A_3}$}{}} \label{sec33:sec}


Denote $\bar{U}_{\theta}=U^{c,\gamma}_{\theta}$. If $(c, \gamma)\in I_{k,l}$ with $(k,l)\in A_2$, then $\bar{U}_{\theta}(-1)=2$ with $\eta_1=4$ and $-3<\bar{U}_{\theta}(1)\ne -2$ or $\bar{U}_{\theta}(1)=-2$ with $\eta_2=0$. If $(c, \gamma)\in I_{k,l}$ with $(k,l)\in A_3$, then $\bar{U}_{\theta}(1)=-2$ with $\eta_2=-4$ and $3>\bar{U}_{\theta}(-1)\ne 2$ or $\bar{U}_{\theta}(-1)=2$ with $\eta_1=0$.

We only need to concentrate on the cases when $(k,l)\in A_2$, since the results of other cases can be obtained from these cases by the transformation $\tilde{x}=-x$ and $\tilde{U}(\tilde{x})=-U(-x)$. 

Choose $0<\epsilon<1/2$, satisfying $\epsilon>-\bar{U}_{\theta}(1)/4$ if $\bar{U}_{\theta}(1)>-2$, and 
    $\epsilon>-\bar{U}_{\theta}(1)/2-1$ if $\bar{U}_{\theta}(1)\le -2$. 
Define
\[
\begin{split}
	\mathbf{M}_1 = & \mathbf{M}_1(\epsilon)\\
		:= & \Big\{ \tilde{U}_{\theta}\in C^3(-\frac{1}{2},\frac{1}{2})\cap C^1(-1,1)\cap C[-1,1] \mid \tilde{U}_\theta(-1) = \tilde{U}_\theta(1) = 0, \\
		& \| \ln\frac{1+x}{3}(1-x)^{-1+2\epsilon} \tilde{U}_{\theta} \|_{L^\infty(-1,1)}< \infty, \\
		& \| (1+x) \left( \ln\frac{1+x}{3} \right)^2 (1-x)^{2\epsilon} \tilde{U}'_{\theta} \|_{L^\infty(-1,1)}<\infty, \\
		& || \tilde{U}_{\theta}'' ||_{L^{\infty}(-\frac{1}{2}, \frac{1}{2})} < \infty, ||\tilde{U}_{\theta}'''||_{L^{\infty}(-\frac{1}{2}, \frac{1}{2})}<\infty \Big\}, \\
	\mathbf{M}_2=& \mathbf{M}_2(\epsilon)\\
		:= & \{ \tilde{U}_{\phi} \in C^2(-1,1) \mid 
		\| (1-x^2)^{\epsilon}\tilde{U}_{\phi} \|_{L^\infty(-1,1)}<\infty,  \|(1-x^2)^{1+\epsilon}\tilde{U}'_{\phi} \|_{L^\infty(-1,1)}<\infty, \\
        	&\| (1 - x^2)^{2+\epsilon} \tilde{U}''_{\phi} \|_{L^\infty(-1,1)}<\infty \}, 
\end{split}
\]
with the following norms accordingly
\[
	\begin{split}
	& \| \tilde{U}_{\theta} \|_{\mathbf{M}_1} := \| \ln\frac{1+x}{3}(1-x)^{-1+2\epsilon} \tilde{U}_{\theta} \|_{L^\infty(-1,1)} 
	+ \| (1+x)\left(\ln\frac{1+x}{3}\right)^2(1-x)^{2\epsilon} \tilde{U}'_{\theta} \|_{L^\infty(-1,1)}\\
	& \hspace{1.5cm}  + ||\tilde{U}_{\theta}''||_{L^{\infty}(-\frac{1}{2}, \frac{1}{2})}+ ||\tilde{U}_{\theta}'''||_{L^{\infty}(-\frac{1}{2}, \frac{1}{2})}, \\
	& \| \tilde{U}_{\phi} \|_{\mathbf{M}_2} := \|(1-x^2)^{\epsilon}\tilde{U}_{\phi}\|_{L^\infty(-1,1)} + \|(1-x^2)^{1+\epsilon}\tilde{U}'_{\phi} \|_{L^\infty(-1,1)} + \| (1-x^2)^{2+\epsilon}\tilde{U}''_{\phi} \|_{L^\infty(-1,1)}. \\
	\end{split}
\]
Next define the following function spaces: 
\[
	\begin{split}
	& \mathbf{N}_1 = \mathbf{N}_1(\epsilon) := \{ \xi_{\theta}\in C^2(-\frac{1}{2},\frac{1}{2})\cap C[-1,1] \mid \xi_\theta(-1) = \xi_\theta(1) =\xi''_{\theta}(0) = 0, \\
	& \hspace{2.8 cm} \| \left(\ln\frac{1+x}{3}\right)^2(1-x)^{-1+2\epsilon} \xi_{\theta} \|_{L^\infty(-1,1)}<\infty, ||\xi'_{\theta}||_{L^{\infty}(-\frac{1}{2},\frac{1}{2})}<\infty,  ||\xi''_{\theta}||_{L^{\infty}(-\frac{1}{2},\frac{1}{2})}<\infty\}, \\
	& \mathbf{N}_2 = \mathbf{N}_2(\epsilon) := \{\xi_{\phi}\in C(-1,1) \mid \| (1-x^2)^{1+\epsilon}\xi_{\phi} \|_{L^\infty(-1,1)}<\infty\}, \\
	\end{split}
\]
with the following norms accordingly: 
\[
	\begin{split}
	& \| \xi_{\theta} \|_{\mathbf{N}_1} := \| \left(\ln\frac{1+x}{3}\right)^2(1-x)^{-1+2\epsilon} \xi_{\theta} \|_{L^\infty(-1,1)}+||\xi'_{\theta}||_{L^{\infty}(-\frac{1}{2},\frac{1}{2})}+||\xi''_{\theta}||_{L^{\infty}(-\frac{1}{2},\frac{1}{2})}, \\
	& \| \xi_{\phi} \|_{\mathbf{N}_2} := \|(1-x^2)^{1+\epsilon}\xi_{\phi} \|_{L^\infty(-1,1)}. 
	\end{split}
\]
Then let $\mathbf{X} := \{ \tilde{U} = (\tilde{U}_\theta, \tilde{U}_\phi ) \mid \tilde{U}_\theta \in \mathbf{M}_1, \tilde{U}_{\phi} \in \mathbf{M}_2\}$ with norm $ \| \tilde{U} \|_{\mathbf{X}} = \| \tilde{U}_\theta \|_{\mathbf{M}_1} + \| \tilde{U}_\phi \|_{\mathbf{M}_2}$, $\mathbf{Y} := \{ \xi = ( \xi_\theta, \xi_\phi ) \mid \xi_\theta \in \mathbf{N}_1, \xi_\phi \in \mathbf{N}_2 \}$, with norm $ \| \xi \|_{\mathbf{Y}} = \| \xi_\theta \|_{\mathbf{N}_1} + \| \xi_\phi \|_{\mathbf{N}_2}$. It can be proved that $\mathbf{M}_1$, $\mathbf{M}_2$, $\mathbf{N}_1$, $\mathbf{N}_2$, $\mathbf{X}$ and $\mathbf{Y}$ are Banach spaces.


Let $l_i:\mathbf{X}\to \mathbb{R}$, $1\le i\le 4$, be the bounded linear functionals defined by (\ref{sec31:eq:fcnal:l}) for each $V\in \mathbf{X}$. Define $\mathbf{X}_1:=\ker l_1\cap \ker l_{2} \cap \ker l_3 \cap \ker l_4$. It can be seen that $\mathbf{X}_1$ is independent of $(c,\gamma)$.

\begin{thm}\label{sec33:thm:1}
	For every compact subset $K\subset I_{2,1}$, for every $(c,\gamma)\in K$, there exists $\delta = \delta(K) > 0$, and $V\in C^{\infty}( K\times B_{\delta}(0), \mathbf{X}_1)$ satisfying $V(c,\gamma,0)=0$ and $\displaystyle \frac{\partial V}{\partial \beta_i}\big |_{\beta=0}=0$, $1\le i\le 4$, $\beta=(\beta_1,\beta_2,\beta_3,\beta_4)$, such that 
	\begin{equation}\label{sec33:eq:U1}
		U = U^{c,\gamma} + \sum_{i=1}^{4} \beta_i V_{c,\gamma}^i + V(c,\gamma,\beta)
	\end{equation}
	satisfies equation (\ref{sec31:eq:NSE}) with $\hat{c}_1= c_1 - \frac{1}{2}\psi[\tilde{U}_{\phi}](-1)$, $\hat{c}_2 = c_2 - \frac{1}{2}\psi[\tilde{U}_{\phi}](1)$, $\hat{c}_3 = c_3 +\frac{1}{2}(\varphi_{c,\gamma}[\tilde{U}_{\theta}])''(0)$.

	Moreover, there exists some $\delta'=\delta'(K)>0$, such that if $\| U - U^{c,\gamma } \|_{\mathbf{X}} < \delta'$, $(c,\gamma) \in K$,  and $U$ satisfies  equation (\ref{sec31:eq:NSE}) with some constants $\hat{c}_1, \hat{c}_2, \hat{c}_3$, then (\ref{sec33:eq:U1}) holds for some $|\beta|< \delta$. 
\end{thm}

Let $l_{2b}$ be the bounded linear functional defined by (\ref{eq_l2ab}). Define $\mathbf{X}_{2b}:=\ker l_{2b} \cap \ker l_3 \cap \ker l_4$. Then $\mathbf{X}_{2b}$ is independent of $(c,\gamma)$. 

\begin{thm}\label{sec33:thm:2}
	For every compact subset $K$ of $I_{2,3}$ or $I_{4,3}$, there exist $\delta = \delta(K) > 0$, and $V\in C^{\infty}( K\times B_{\delta}(0), \mathbf{X}_{2b})$ satisfying $V(c,\gamma,0)=0$ and $\displaystyle \frac{\partial V}{\partial \beta_i}\big |_{\beta=0}=0$, $i=2,3,4$, $\beta=(\beta_2,\beta_3,\beta_4)$, such that 
	\begin{equation}\label{sec33:eq:U2}
		U = U^{c,\gamma} + \beta_2 V_{c,\gamma}^{2b} + \beta_3 V_{c,\gamma}^3 + \beta_4 V_{c,\gamma}^4+V(c,\gamma,\beta)
	\end{equation}
	satisfies equation (\ref{sec31:eq:NSE}) with $\hat{c}_1= c_1 - \frac{1}{2}\psi[\tilde{U}_{\phi}](-1)$, $\hat{c}_2 = c_2 - \frac{1}{2}\psi[\tilde{U}_{\phi}](1)$, $\hat{c}_3 = c_3 +\frac{1}{2}(\varphi_{c,\gamma}[\tilde{U}_{\theta}])''(0)$.

	Moreover, there exists some $\delta'=\delta'(K)>0$, such that if $\| U - U^{c,\gamma } \|_{\mathbf{X}} < \delta'$, $(c,\gamma) \in K$,  and $U$ satisfies  equation (\ref{sec31:eq:NSE}) with some constants $\hat{c}_1, \hat{c}_2, \hat{c}_3$, then (\ref{sec33:eq:U2}) holds for some $|\beta|< \delta$. 
\end{thm}

For $\tilde{U}_{\phi}\in \mathbf{M}_2$, let $\psi[\tilde{U}_{\phi}](x)$ be defined by (\ref{sec31:eq:psi}). Let $K$ be a compact subset contained in either $I_{2,1}$, $I_{2,3}$ or $I_{4,3}$. Define a map $G = G(c,\gamma,\tilde{U})$ on $K\times \mathbf{X}$ by (\ref{sec31:eq:G}). 

\begin{prop}\label{sec33:prop}
	The map $G$ is in $C^{\infty}(K\times \mathbf{X}, \mathbf{Y})$ in the sense that $G$ has  continuous Fr\'{e}chet derivatives of every order. Moreover, the Fr\'{e}chet derivative of $G$ with respect to $\tilde{U}$ at $(c,\gamma,\tilde{U})\in K\times \mathbf{X}$ is given by the linear bounded operator $L^{c,\gamma}_{\tilde{U}}: \mathbf{X}\rightarrow \mathbf{Y}$ defined as in (\ref{sec31:eq:Linear}).
\end{prop}
To prove Proposition \ref{sec33:prop}, we first prove the following lemmas:\\

\begin{lem}\label{sec33:lem:A:well-def}
	For every $(c,\gamma)\in K$, $A(c,\gamma,\cdot): \mathbf{X}\to \mathbf{Y}$ defined by (\ref{sec31:eq:A}) is a well-defined bounded linear operator.
\end{lem}
\begin{proof}
	In the following, $C$ denotes a universal constant which may change from line to line. For convenience we denote $l=l_{c,\gamma}[\tilde{U}_{\theta}]$ defined by (\ref{sec31:eq:l}), and $A=A(c,\gamma,\cdot)$ for some fixed $(c,\gamma)\in K$. We make use of the property of $\bar{U}_{\theta}$ that $\bar{U}_{\theta}\in C^2(-1,1)\cap L^{\infty}(-1,1)$, and $\bar{U}_{\theta}=2+4\left(\ln \frac{1+x}{3}\right)^{-1}+O(1)\left(\ln \frac{1+x}{3}\right)^{-2}$.

	$A$ is clearly linear. For every $\tilde{U}\in\mathbf{X}$, we prove that $A\tilde{U}$ defined by (\ref{sec31:eq:A}) is in $\mathbf{Y}$ and there exists some constant $C$ such that $\|A\tilde{U}\|_{\mathbf{Y}}\le C\|\tilde{U}\|_{\mathbf{X}}$ for all $\tilde{U}\in \mathbf{X}$.

	By computation,
	\[
	    l'(x)=(1-x^2)\tilde{U}''_{\theta}+\bar{U}_{\theta}\tilde{U}'_{\theta}+(2+\bar{U}'_{\theta})\tilde{U}_{\theta},
	\]
	\[
		l''(x)=(1-x^2)\tilde{U}'''_{\theta}+(\bar{U}_{\theta}-2x)\tilde{U}''_{\theta}+2(\bar{U}'_{\theta}+1)\tilde{U}'_{\theta}+\bar{U}''_{\theta}\tilde{U}_{\theta}.
	\]
	
	By the fact that $\tilde{U}_{\theta}\in \mathbf{M}_1$, we have $|l''(0)|\le C||\tilde{U}_{\theta}||_{\mathbf{M}_1}$. So for $-1<x<1$, we have
	\[
	\begin{split}
		& \quad |\left( \ln \frac{1+x}{3} \right)^2 (1-x)^{-1+2\epsilon}A_{\theta}|  \\
		& \le | \left( \ln \frac{1+x}{3} \right)^2 (1-x)^{-1+2\epsilon}l(x)|+\frac{1}{2}|l''(0)|\left( \ln \frac{1+x}{3} \right)^2 (1+x) (1-x)^{2\epsilon}\\
		& \le |\left( \ln \frac{1+x}{3} \right)^2 (1+x) (1-x)^{2\epsilon}\tilde{U}'_{\theta}|+ \left| (2x+\bar{U}_{\theta})\left( \ln \frac{1+x}{3} \right) \right| \cdot\left|\left( \ln \frac{1+x}{3} \right)\right|(1-x)^{-1+2\epsilon}|\tilde{U}_{\theta}|\\
		& \quad +C||\tilde{U}_{\theta}|| \left( \ln \frac{1+x}{3} \right)^2 (1+x) (1-x)^{2\epsilon}\\
		& \le C\|\tilde{U}_{\theta}\|_{\mathbf{M}_1},
	\end{split}
	\]
	where we have used the property that $\bar{U}_{\theta}(x)=2+4\left(\ln \frac{1+x}{3}\right)^{-1}+O(1)\left(\ln \frac{1+x}{3}\right)^{-2}$, so there exists constant $C>0$ such that 
	\begin{equation*}
		(2x+\bar{U}_\theta) \left( \ln \frac{1+x}{3} \right) \leq C, \quad -1<x<1.
	\end{equation*}
	We also see from the above that $\displaystyle{\lim_{x\to \pm 1}A_{\theta}(x)=0}$. 
	
For $-\frac{1}{2}<x<\frac{1}{2}$, 
	\[
	  \begin{split}
	    |A'_{\theta}|& =|l'(x)-l''(0)x|\\
	       & \le |\tilde{U}''_{\theta}|+|\bar{U}_{\theta}||\tilde{U}'_{\theta}|+(2+|\bar{U}'_{\theta}|)|\tilde{U}_{\theta}|+|l''(0)|\\
	       & \le C\|\tilde{U}_{\theta}\|_{\mathbf{M}_1},
	    \end{split}
	\]
	and
	\[
	\begin{split}
	   |A''_{\theta}|& =|l''(x)-l''(0)|\\
	   & \le |\tilde{U}'''_{\theta}|+(|\bar{U}_{\theta}|+2)|\tilde{U}''_{\theta}|+2(|\bar{U}'_{\theta}|+1)|\tilde{U}'_{\theta}|+|\bar{U}''_{\theta}||\tilde{U}_{\theta}|+|l''(0)|\\
	   & \le C\|\tilde{U}_{\theta}\|_{\mathbf{M}_1}.
	   \end{split} 
	\]
	
	By computation $A''_{\theta}(0)=0$. So we have $A_{\theta}\in \mathbf{N}_1$ and $\|A_{\theta}\|_{\mathbf{N}_1}\le C\|\tilde{U}_{\theta}\|_{\mathbf{M}_1}$.
	
	Next, since $A_{\phi}=(1-x^2)\tilde{U}''_{\phi}+\bar{U}_{\theta}\tilde{U}'_{\phi}$, by the fact that $\tilde{U}_{\phi}\in \mathbf{M}_2$  we have that
	\[
		\left|(1-x^2)^{1+\epsilon}A_{\phi}\right| \le (1-x^2)^{2+\epsilon}|\tilde{U}''_{\phi}|+(1-x^2)^{1+\epsilon}|\bar{U}_{\theta}\|\tilde{U}'_{\phi}|\le C\|\tilde{U}_{\phi}\|_{\mathbf{M}_2}.
	\]
	So $A_{\phi}\in \mathbf{N}_1$, and $\|A_{\phi}\|_{\mathbf{N}_1}\le C\|\tilde{U}_{\phi}\|_{\mathbf{M}_2}$. We have proved that $A\tilde{U}\in \mathbf{Y}$ and $\|A\tilde{U}\|_{\mathbf{Y}}\le C\|\tilde{U}\|_{\mathbf{X}}$ for every $\tilde{U}\in \mathbf{X}$. The proof is finished.
\end{proof}

\begin{lem}\label{sec33:lem:Q:well-def}
	The map $Q:\mathbf{X}\times\mathbf{X}\to \mathbf{Y}$ defined by (\ref{sec31:eq:Q}) is a well-defined bounded bilinear operator.
\end{lem}
\begin{proof}
	In the following, $C$ denotes a universal constant which may change from line to line. For convenience we denote $\psi=\psi[\tilde{U}_{\phi}, \tilde{V}_{\phi}]$ defined by (\ref{sec31:eq:psi}).  

	It is clear that $Q$ is a bilinear operator. For every $\tilde{U},\tilde{V}\in\mathbf{X}$, we will prove that $Q(\tilde{U},\tilde{V})$ is in $\mathbf{Y}$ and there exists some constant $C$ independent of $\tilde{U}$ and $\tilde{V}$ such that $\|Q(\tilde{U},\tilde{V})\|_{\mathbf{Y}}\le C\|\tilde{U}\|_{\mathbf{X}}\|\tilde{V}\|_{\mathbf{X}}$.
    
	For $\tilde{U},\tilde{V}\in \mathbf{X}$, by the same arguments as in Lemma \ref{sec32:lem:Q:well-def}, there exists some constant $C>0$ such that 
	
	\[
	    |(\tilde{U}_{\theta}\tilde{V}_{\theta})''(0)|\le C||\tilde{U}_{\theta}||_{\mathbf{M}_1}||\tilde{V}_{\theta}||_{\mathbf{M}_1}, 
	\]
	and 
	\[
		\left|\psi(x)-\frac{1}{2}\psi(-1)(1-x)-\frac{1}{2}\psi(1)(1+x)\right|\le C(1-x^2)^{1-2\epsilon}||\tilde{U}_{\phi}||_{\mathbf{M}_2}||\tilde{V}_{\phi}||_{\mathbf{M}_2},\quad -1<x<1.
	\]

	So we have that for $-1<x<1$, 
	\[
	\begin{split}
		& |\left(\ln \frac{1+x}{3}\right)^2 (1-x)^{-1+2\epsilon}Q_{\theta}(x)|\\
		& \le \frac{1}{2} \left| \left(\ln \frac{1+x}{3}\right) (1-x)^{-1+2\epsilon} \tilde{U}_{\theta}(x)\right| \cdot \left|\left(\ln \frac{1+x}{3}\right) \tilde{V}_{\theta}(x)\right| \\
		& \quad + \left| \ln \frac{1+x}{3} \right| (1-x)^{-1+2\epsilon}\left|\psi(x)-\frac{1}{2}\psi(-1)(1-x)-\frac{1}{2}\psi(1)(1+x)\right|\\
		& \quad +\frac{1}{4} \left| \ln \frac{1+x}{3} \right| (1+x) (1-x)^{2\epsilon}|(\tilde{U}_{\theta}\tilde{V}_{\theta})''(0)|\\
		& \le \frac{1}{2}\|\tilde{U}_{\theta}\|_{\mathbf{M}_1}\|\tilde{V}_{\theta}\|_{\mathbf{M}_1}+C\|\tilde{U}_\phi\|_{\mathbf{M}_2}\|\tilde{V} _{\phi}\|_{\mathbf{M}_2}+C(1-x^2)^{2\epsilon}\|\tilde{U}_{\theta}\|_{\mathbf{M}_1}\|\tilde{V}_{\theta}\|_{\mathbf{M}_1}\\
		& \le C\|\tilde{U}\|_{\mathbf{X}}\|\tilde{V}\|_{\mathbf{X}}.
	\end{split}
	\]
	From this we also have $\displaystyle \lim_{x\to \pm 1}Q_{\theta}(x)=0$. 
	
	Similar as in Lemma \ref{sec32:lem:Q:well-def}, we have that for $-\frac{1}{2}<x<\frac{1}{2}$,
	\[
	    |Q'_{\theta}(x)| 
 \le C\|\tilde{U}\|_{\mathbf{X}}\|\tilde{V}\|_{\mathbf{X}},\quad |Q''_{\theta}(x)|\le C\|\tilde{U}\|_{\mathbf{X}}\|\tilde{V}\|_{\mathbf{X}}.
	\]
	By computation $Q_{\theta}''(0)=0$, so $Q_{\theta}\in\mathbf{N}_1$, and $\|Q_{\theta}\|_{\mathbf{N}_1}\le C(\epsilon)\|\tilde{U}\|_{\mathbf{X}}\|\tilde{V}\|_{\mathbf{X}}$.
	
	Next, since $Q_{\phi}(x)=\tilde{U}_{\theta}(x)\tilde{V}'_{\phi}(x)$, for $-1<x<1$, 
	\[
		\left|(1-x^2)^{1+\epsilon}Q_{\phi}(x)\right|  \le |\tilde{U}_{\theta}(x)|(1-x^2)^{1+\epsilon}|\tilde{V}'_{\phi}|
		\le \|\tilde{U}_{\theta}\|_{\mathbf{M}_1}\|\tilde{V}_{\phi}\|_{\mathbf{M}_2}.
	\]
	So $Q_{\phi}\in \mathbf{N}_2$ and 
	$
		\|Q_{\phi}\|_{\mathbf{N}_2}\le \|\tilde{U}_{\theta}\|_{\mathbf{M}_1}\|\tilde{V}_{\phi}\|_{\mathbf{M}_2}.
	$
	Thus we have proved that $Q(\tilde{U}, \tilde{V})\in \mathbf{Y}$ and $\|Q(\tilde{U},\tilde{V})\|_{\mathbf{Y}}\le C\|\tilde{U}\|_{\mathbf{X}}\|\tilde{V}\|_{\mathbf{X}}$ for all $\tilde{U}, \tilde{V}\in \mathbf{X}$. Lemma \ref{sec33:lem:Q:well-def} is proved.
\end{proof}


\vspace{1em}

\noindent 
{\bf Proof of Proposition \ref{sec33:prop}.}
By definition, $G(c,\gamma,\tilde{U})=A(c,\gamma,\tilde{U})+Q(\tilde{U},\tilde{U})$ for $(c,\gamma,\tilde{U})\in K \times \mathbf{X}$.  
Using standard theories in functional analysis, by Lemma \ref{sec33:lem:Q:well-def} it is clear that $Q$ is $C^{\infty}$ on $\mathbf{X}$.  
By Lemma \ref{sec33:lem:A:well-def}, $A(c,\gamma; \cdot): \mathbf{X}\to \mathbf{Y}$ is $C^{\infty}$ for each $(c,\gamma)\in K$.

Let $\alpha=(\alpha_1,\alpha_2,\alpha_3)$ be a multi-index where $\alpha_i\ge 0$, $i=1,2,3$, and $j\ge 0$. For all $|\alpha|+j\ge 1$,  we have
\begin{equation}\label{eq_prop4_3_1}
	\pt_{c}^\alpha \pt_{\gamma}^j A(c,\gamma,\tilde{U}) = \pt_{c}^\alpha \pt_{\gamma}^j U^{c,\gamma}_{\theta}\left(
	\begin{matrix}
		\tilde{U}_{\theta}  \\  \tilde{U}'_{\phi} 
	\end{matrix}\right) + \frac{1}{2} (\pt_{c}^\alpha \pt_{\gamma}^j U^{c,\gamma}_{\theta} \cdot \tilde{U}_\theta)'' (0) 
	\begin{pmatrix}
	1-x^2   \\   0  
	\end{pmatrix}.  
\end{equation}

By Proposition \ref{propA_1} (2), we have
\[
	|\left(\ln\frac{1+x}{3}\right)^2 (1-x)^{-1+2\epsilon} \pt_{c}^\alpha \pt_{\gamma}^j A_{\theta}(c,\gamma,\tilde{U}) | \leq C(\alpha,j,K) \| \tilde{U}_{\theta} \|_{\mathbf{M}_1}, \quad -1<x<1,
\]
and for $-\frac{1}{2}<x<\frac{1}{2}$,
\[
     | \pt_{c}^\alpha \pt_{\gamma}^j A'_{\theta}(c,\gamma,\tilde{U}) | \leq C(\alpha,j,K) \| \tilde{U}_{\theta} \|_{\mathbf{M}_1}, \quad | \pt_{c}^\alpha \pt_{\gamma}^j A''_{\theta}(c,\gamma,\tilde{U}) | \leq C(\alpha,j,K) \| \tilde{U}_{\theta} \|_{\mathbf{M}_1}.
\]
The above estimates and (\ref{eq_prop4_3_1}) also imply that 
$$
	\pt_{c}^\alpha \pt_{\gamma}^j A_{\theta}(c,\gamma,\tilde{U}) (-1) = \pt_{c}^\alpha \pt_{\gamma}^j A_{\theta}(c,\gamma,\tilde{U})(1) = \pt_{c}^\alpha \pt_{\gamma}^j A_{\theta}(c,\gamma,\tilde{U})'' (0) = 0. 
$$
So $\pt_{c}^\alpha \pt_{\gamma}^j A_{\theta}(c,\gamma,\tilde{U}) \in \mathbf{N}_1$, with $\|\pt_{c}^\alpha \pt_{\gamma}^j A_{\theta}(c,\gamma,\tilde{U}) \|_{\mathbf{N}_1} \leq C(\alpha,j,K) \|\tilde{U}_{\theta} \|_{\mathbf{M}_1}$ for all $(c,\gamma,\tilde{U}) \in K\times \mathbf{X}$. 

Next, by Proposition \ref{propA_1} (2) and the fact that $\tilde{U}_{\phi}\in \mathbf{M}_1$, we have 
\begin{equation}\label{eq_prop3_3_2}
	(1-x^2)^{1+\epsilon} | \pt_{c}^\alpha \pt_{\gamma}^j A_{\phi}(c,\gamma,\tilde{U})(x)| 
	= |\pt_{c}^\alpha \pt_{\gamma}^j  U^{c,\gamma}_{\theta}(x)| \cdot |(1-x^2)^{1+\epsilon}U'_{\phi}|
	\leq C(\alpha,j,K) \| \tilde{U}_{\phi} \|_{\mathbf{M}_2}. 
\end{equation}
So $\pt_{c}^\alpha \pt_{\gamma}^j A_{\phi}(c,\gamma,\tilde{U})\in \mathbf{N}_2$ with $ \|\pt_{c}^\alpha \pt_{\gamma}^j A_{\phi}(c,\gamma,\tilde{U}) \|_{\mathbf{N}_2} \leq C(\alpha,j,K) \| \tilde{U}_{\phi} \|_{\mathbf{M}_2}$ for all $(c,\gamma,\tilde{U})\in K\times \mathbf{X}$. Thus $\pt_{c}^\alpha \pt_{\gamma}^j A(c,\gamma,\tilde{U})\in \mathbf{Y}$, with $\| \pt_{c}^\alpha \pt_{\gamma}^j A(c,\gamma,\tilde{U}) \|_{\mathbf{Y}} \leq C(\alpha,j,K) \| \tilde{U} \|_{\mathbf{X}}$ for all $(c,\gamma,\tilde{U})\in K\times \mathbf{X}$, $|\alpha|+j\ge 1$. 

So for each $(c,\gamma)\in K$, $\pt_{c}^\alpha \pt_{\gamma}^j A(c,\gamma; \cdot): \mathbf{X}\to \mathbf{Y}$ is a bounded linear map with uniform bounded norm on $K$. Then by standard theories in functional analysis, $A: K\times\mathbf{X}\to \mathbf{Y}$ is $C^{\infty}$. So $G$ is a $C^{\infty}$ map from $K  \times \mathbf{X}$ to $\mathbf{Y}$. By direct calculation we get its Fr\'{e}chet derivative with respect to $\mathbf{X}$ is given by  the linear bounded operator $L^{c,\gamma}_{\tilde{U}}: \mathbf{X}\rightarrow \mathbf{Y}$ defined as  (\ref{sec31:eq:Linear}). The proof is finished.     \qed 
\\


Let $a_{c,\gamma}(x), b_{c,\gamma}(x)$ be functions defined by (\ref{sec31:eq:ab}). For convenience we denote $a(x)=a_{c,\gamma}(x)$, $b(x)=b_{c,\gamma}(x)$, and $\bar{U}_{\theta}=U^{c,\gamma}_{\theta}$. 
\begin{lem}\label{lem3_3_ab}
  For $(c,\gamma)\in I_{k,l}$ with $(k,l)\in A_2$, there exists some constant $C>0$, depending only on $(c,\gamma)$, such that  for any $-1<x<1$, 
 \begin{equation}\label{sec33:eq:eb}
		 e^{b(x)}\le C \left(\ln\frac{1+x}{3}\right)^2 (1+x) (1-x)^{-\frac{\bar{U}_{\theta}(1)}{2}},  \quad
		 e^{-b(x)}\le C \left(\ln\frac{1+x}{3}\right)^{-2} (1+x)^{-1}(1-x)^{\frac{\bar{U}_{\theta}(1)}{2}},
	\end{equation}
	and 
	\begin{equation}\label{sec33:eq:ea}
		e^{a(x)}\le C \left| \ln \frac{1+x}{3} \right|^2 (1-x)^{-1-\frac{\bar{U}_{\theta}(1)}{2}}, \quad 
		e^{-a(x)}\le C\left(\ln\frac{1+x}{3}\right)^{-2} (1-x)^{1+\frac{\bar{U}_{\theta}(1)}{2}}. 
	\end{equation}
\end{lem}
\begin{proof}
   Let 
	\[
		\alpha_0=\min\left\{1, \left(1+\frac{\bar{U}_{\theta}(1)}{2}\right)\}\chi_{\{\bar{U}_{\theta}(1)>-2\}}+\chi_{\{\bar{U}_{\theta}(1)\le - 2\}} \right\}.
	\]
	Under the assumption of $\bar{U}_{\theta}$ in this case, by Theorem 1.3 in \cite{LLY1} or Lemma 2.14 in \cite{LLY2} we have,
		\[
		\bar{U}_{\theta}=2+\frac{4}{\ln\frac{1+x}{3}}+O(1)\left(\ln\frac{1+x}{3}\right)^{-2}=\bar{U}_{\theta}(1)+O((1-x)^{\alpha_0}), \quad -1<x<1.
	\]
	Thus, by definition of $a(x)$ and $b(x)$ in (\ref{sec31:eq:ab}),  for $-1<x<1$, we have
	\begin{equation*}
	\begin{split}
		& b(x)=\ln\left(\frac{1+x}{3}\right) + 2\ln\left(-\ln\left(\frac{1+x}{3}\right)\right) - \frac{\bar{U}_{\theta}(1)}{2}\ln(1-x)+O(1),\\
		& a(x)=2\ln\left(-\ln\left(\frac{1+x}{3}\right)\right)-\left(\frac{\bar{U}_{\theta}(1)}{2}+1\right)\ln(1-x)+O(1).
	\end{split}
	\end{equation*}
	The lemma then follows from the above etimates.
\end{proof}


For $\xi = (\xi_\theta, \xi_\phi) \in \mathbf{Y}$, let the map $W^{c,\gamma}$ be defined as 
$$
	W^{c,\gamma}(\xi) := (W^{c,\gamma}_\theta(\xi), W^{c,\gamma}_\phi(\xi)), 
$$
where 
\begin{equation*}
	W^{c,\gamma}_{\theta}(\xi)=\left\{ 
	\begin{array}{ll}
		W^{c,\gamma,1}_{\theta}(\xi) & \textrm{ if } (c,\gamma)\in I_{2,1},\\
		W^{c,\gamma,2b}_{\theta}(\xi) & \textrm{ if }(c,\gamma)\in I_{2,3} \textrm{ or } I_{4,3},\\
	\end{array}
	\right.
\end{equation*}
$W^{c,\gamma,i}_{\theta}$, $i=1,2b$, are defined by (\ref{sec31:eq:Wthei}), and $W^{c,\gamma}_\phi(\xi)$ is defined by (\ref{sec31:eq:Wphi}).

\begin{lem}\label{sec33:lem:W}
	For every $(c,\gamma)\in K$, $W^{c,\gamma}: \mathbf{Y}\rightarrow\mathbf{X}$ is continuous, and is a right inverse of $L^{c,\gamma}_{0}$. 
\end{lem}
\begin{proof}
	In the following, $C$ denotes a universal constant which may change from line to line.  We make use of the property that $\bar{U}_{\theta}\in C^2(-1,1)\cap L^{\infty}(-1,1)$ and the range of $\epsilon$.  For convenience let us write  $W:=W^{c,\gamma}(\xi)$ and $W^i_{\theta}:=W_{\theta}^{c,\gamma,i}(\xi)$ for $\xi\in \mathbf{Y}$. 


	We first prove $W_{\theta}: \mathbf{Y}\to\mathbf{X}$ is well-defined. 

	\noindent
	{\bf Claim.} There exists $C>0$, such that 
	\begin{equation}\label{sec33:eq:Wthe:bdd}
		\left| \left( \ln \frac{1+x}{3} \right) (1-x)^{-1+2\epsilon}W_{\theta}(x)\right| \le C \|\xi_{\theta}\|_{\mathbf{N}_1}. 
	\end{equation}
	
	\noindent 
	{\bf Proof of the Claim.} We prove the claim for each $W^i$, $i=1,2b$. 

	\textbf{Case 1.} $(c,\gamma)\in I_{2,1}$, then $\bar{U}_{\theta}(-1)=2$ with $\eta_1=4$ and $\bar{U}_{\theta}(1)>-2$.
	
	 In this case $W_{\theta}=W^1_{\theta}$. 
	Using the fact that $\xi_{\theta}\in \mathbf{N}_1$, in the expression of $W_{\theta}=W^1_{\theta}$ in (\ref{sec31:eq:Wthei}), for any $-1<x<1$ 
	\[
	\begin{split}
		& \quad \left| \left( \ln \frac{1+x}{3} \right) (1-x)^{-1+2\epsilon}W^1_{\theta}(x)\right| \\
		& \le \left| \ln \frac{1+x}{3} \right| (1-x)^{-1+2\epsilon}\|\xi_{\theta}\|_{\mathbf{N}_1}e^{-a(x)} 
		\int_{0}^{x}\frac{e^{a(s)}}{1+s} \left( \ln \frac{1+s}{3} \right)^{-2} (1-s)^{-2\epsilon}ds.
	\end{split}
	\]
	Applying (\ref{sec33:eq:ea}) in the above, using the fact that $4\epsilon> -\bar{U}_{\theta}(1)$, we have
	\begin{equation}\label{sec33:eq:Wthe:bdd1}
	\begin{split}
		& \quad  \left| \ln \left( \frac{1+x}{3} \right) (1-x)^{-1+2\epsilon} W^1_{\theta}(x)\right| \\
		& \le \|\xi_{\theta}\|_{\mathbf{N}_1} \left| \ln \frac{1+x}{3} \right|^{-1} (1-x)^{\frac{\bar{U}_{\theta}(1)}{2}+2\epsilon}
		\int_{0}^{x} \frac{1}{1+s} (1-s)^{-1-\frac{\bar{U}_{\theta}(1)}{2}-2\epsilon}ds\\
		& \le C \|\xi_{\theta}\|_{\mathbf{N}_1} \left( 1 + \left| \ln \frac{1+x}{3} \right|^{-1} \right)\left(1+(1-x)^{\frac{\bar{U}_{\theta}(1)}{2}+2\epsilon}\right)\\
		& \le C \|\xi_{\theta}\|_{\mathbf{N}_1}. 
	\end{split}
	\end{equation}

	\textbf{Case 2.} $(c,\gamma)\in I_{2,3}$ or $I_{4,3}$, then $\bar{U}_{\theta}(-1)=2$ with $\eta_1=4$ and $-3<\bar{U}_{\theta}(1)<-2$ or $\bar{U}_{\theta}(1)=-2$ with $\eta_2=0$.
	
	In this case $W_{\theta}=W^{2b}_{\theta}$. 
	Using the fact that $\xi_{\theta}\in \mathbf{N}_1$, and (\ref{sec33:eq:ea}) we first have
	\[
		\int_1^0e^{a(s)}\frac{|\xi_{\theta}(s)|}{1-s^2}ds
		\le C\|\xi_{\theta}\|_{\mathbf{N}_1}
		\int_1^0 (1-s)^{-\frac{\bar{U}_{\theta}(1)}{2}-1-2\epsilon}ds\le C\|\xi_{\theta}\|_{\mathbf{N}_1}. 
	\]
	So the definition of $W^{2b}_{\theta}$ makes sense.
	
	In the expression of $W^{2b}_{\theta}$ in (\ref{sec31:eq:Wthei}), we have for any $-1<x<1$ that 
	\[
	\begin{split}
		& \quad \left| \left( \ln \frac{1+x}{3} \right) (1-x)^{-1+2\epsilon}W^{2b}_{\theta}(x) \right|  \\
		& \le \left| \ln \frac{1+x}{3} \right| (1-x)^{-1+2\epsilon}\|\xi_{\theta}\|_{\mathbf{N}_1}e^{-a(x)}\int_{1}^{x}e^{a(s)} \left( \ln \frac{1+s}{3} \right)^{-2} \frac{1}{1+s} (1-s)^{-2\epsilon}ds. 
	\end{split}
	\]
	Applying (\ref{sec33:eq:ea}) in the above, using the fact that $\bar{U}_{\theta}(1)\le -2$, and $\epsilon<1/2$, we have that 
	\begin{equation}\label{sec33:eq:Wthe:bdd2}
	\begin{split}
		& \quad \left| \left( \ln \frac{1+x}{3} \right) (1-x)^{-1+2\epsilon} W^{2b}_{\theta}(x) \right|  \\
		& \le \|\xi_{\theta}\|_{\mathbf{N}_1} \left| \ln \frac{1+x}{3} \right|^{-1} (1-x)^{\frac{\bar{U}_{\theta}(1)}{2}+2\epsilon}
		\int_{1}^{x} \frac{1}{1+s} (1-s)^{-1-\frac{\bar{U}_{\theta}(1)}{2}-2\epsilon}ds  \\
		& \le C \|\xi_{\theta}\|_{\mathbf{N}_1}\left(1 + \left| \ln \frac{1+x}{3} \right|^{-1} \right) \\
		& \le C \|\xi_{\theta}\|_{\mathbf{N}_1}.
	\end{split}
	\end{equation}
	%
	So (\ref{sec33:eq:Wthe:bdd}) can be obtained from (\ref{sec33:eq:Wthe:bdd1}) and (\ref{sec33:eq:Wthe:bdd2}). The claim is proved.
	
	From the claim we also have that $\lim_{x\to \pm 1}W_{\theta}(x)=0$.
	
	By (\ref{sec31:eq:diff:Wthe}), (\ref{sec31:eq:diff:a}), (\ref{sec33:eq:Wthe:bdd}), and the property that $\bar{U}_{\theta}=2+4\left(\ln \frac{1+x}{3}\right)^{-1}+O(1)\left(\ln \frac{1+x}{3}\right)^{-2}$, we have that for $-1<x<1$,
	\begin{equation}\label{eq3_3_W_1}
	\begin{split}
		& \quad \left| (1+x) \left( \ln \frac{1+x}{3} \right)^2 (1-x)^{2\epsilon}W'_{\theta} \right| \\
		& \le \left| (2x+\bar{U}_{\theta}) \left( \ln \frac{1+x}{3} \right)\right|\cdot\left|  \left( \ln \frac{1+x}{3} \right)(1-x)^{-1+2\epsilon}W_{\theta} \right|
		+ \left( \ln \frac{1+x}{3} \right)^2 (1-x)^{-1+2\epsilon}|\xi_{\theta}(x)|  \\
		& \le  C \|\xi_{\theta}\|_{\mathbf{N}_1}.
	\end{split}
	\end{equation}
	By (\ref{sec31:eq:diff:a}), it can be seen that $|a''(x)|,|a'''(x)|\le C$ for $-\frac{1}{2}<x<\frac{1}{2}$. Then using this fact and (\ref{sec33:eq:Wthe:bdd}) and (\ref{eq3_3_W_1}), we have, for $-\frac{1}{2}<x<\frac{1}{2}$,
	\[
	    |W_{\theta}''(x)|=\left|a''(x)W_{\theta}(x)+a'(x)W'_{\theta}(x)+\left(\frac{\xi_{\theta}}{1-x^2}\right)'\right|\le C \|\xi_{\theta}\|_{\mathbf{N}_1},
	\]
	and 
	\[
	     |W_{\theta}'''(x)|=\left|a'''(x)W_{\theta}(x)+2a''(x)W'_{\theta}(x)+a'(x)W''_{\theta}(x)+\left(\frac{\xi_{\theta}}{1-x^2}\right)''\right|\le C \|\xi_{\theta}\|_{\mathbf{N}_1}.
	\]

	So we have shown that $W_{\theta}\in \mathbf{M}_1$, and $\|W_{\theta}\|_{\mathbf{M}_1}\le C\|\xi_{\theta}\|_{\mathbf{N}_1}$ for some constant $C$. 
	
	By the definition of $W_{\phi}(\xi)$ in (\ref{sec31:eq:Wphi}) , using (\ref{sec33:eq:eb}) and the fact that $\xi_{\phi}\in \mathbf{N}_2$, we have, for every $-1<x<1$, 
	\[
	\begin{split}
		& \quad (1-x^2)^{\epsilon}|W_{\phi}(x)|  \le  (1-x^2)^{\epsilon}\int_{0}^{x}e^{-b(t)}\int_{0}^{t}e^{b(s)}\frac{|\xi_{\phi}(s)|}{1-s^2}dsdt\\
		& \le \|\xi_{\phi}\|_{\mathbf{N}_2}  (1-x^2)^{\epsilon}\int_{0}^{x}e^{-b(t)}\int_{0}^{t}e^{b(s)}(1-s^2)^{-2-\epsilon}dsdt \\
		& \le  C\|\xi_{\phi}\|_{\mathbf{N}_2} (1-x^2)^{\epsilon} \int_{0}^{x} \left(\ln\frac{1+t}{3}\right)^{-2} (1+t)^{-1} (1-t)^{\frac{\bar{U}_{\theta}(1)}{2}}\\
		& \quad \cdot \int_{0}^{t} \left(\ln\frac{1+s}{3}\right)^2 (1+s)^{-1-\epsilon} (1-s)^{-\frac{\bar{U}_{\theta}(1)}{2}-2-\epsilon}dsdt\\
		& \le C \|\xi_{\phi}\|_{\mathbf{N}_2} (1-x^2)^{\epsilon}\int_{0}^{x}(1-t^2)^{-1-\epsilon}dt\\
		& \le C\|\xi_{\phi}\|_{\mathbf{N}_2}. 
		\end{split}
	\]
	
	Using (\ref{sec31:eq:diff1:Wphi}), (\ref{sec33:eq:eb}) and the fact that $\xi_{\phi}\in \mathbf{N}_2$, we have, for every $-1<x<1$, 
	\begin{equation}\label{sec33:eq:Wphi:bdd:temp}
	\begin{split}
		|(1-x^2)^{1+\epsilon}W'_{\phi}(x)| & \le \|\xi_{\phi}\|_{\mathbf{N}_2} \left( \ln \frac{1+x}{3} \right)^{-2} (1+x)^{\epsilon}(1-x)^{\frac{\bar{U}_{\theta}(1)}{2}+1+\epsilon}  \\
		& \quad \cdot \int_{0}^{x} \left( \ln \frac{1+s}{3} \right)^2 (1+s)^{-1-\epsilon}(1-s)^{-\frac{\bar{U}_{\theta}(1)}{2}-2-\epsilon}ds\\
		& \le C\|\xi_{\phi}\|_{\mathbf{N}_2}. 
	\end{split}
	\end{equation}
	Similarly, since $|b'(x)|=\frac{|\bar{U}_{\theta}|}{1-x^2}$, using (\ref{sec31:eq:diff2:Wphi}),  (\ref{sec33:eq:Wphi:bdd:temp}) and the fact that $\xi_{\phi}\in \mathbf{N}_2$, we have
	\[
		|(1-x^2)^{2+\epsilon}W''_{\phi}(x)|\le C(1-x^2)^{1+\epsilon}|W'_{\phi}|+(1-x^2)^{1+\epsilon}|\xi_{\phi}|\le C\|\xi_{\phi}\|_{\mathbf{N}_2}.
	\]
	Then $W(\xi)\in\mathbf{X}$ for all $\xi\in\mathbf{Y}$, and $\|W(\xi)\|_{\mathbf{X}}\le C\|\xi\|_{\mathbf{Y}}$ for some constant $C$. So $W:\mathbf{Y}\rightarrow\mathbf{X}$ is well-defined and continuous. 
	 
	By definition of $W^i$, $i=1,2b$, we have $l_{c,\gamma}[W^i_{\theta}](x)=\xi_{\theta}$. So $(l_{c,\gamma}[W^i_{\theta}])''(0)=\xi''_{\theta}(0)=0$, then $l_{c,\gamma}[W^i_{\theta}](x)+\frac{1}{2}(l_{c,\gamma}[W^i_{\theta}])''(0)(1-x^2)=\xi_{\theta}$.  Thus $L_0^{c,\gamma}W(\xi)=\xi$, $W$ is a right inverse of $L_0^{c,\gamma}$. 
\end{proof}



Let $V_{c,\gamma}^i$, $1\le i\le 4$ and  $V_{c,\gamma}^{2b}$ be vectors defined by (\ref{eq_basis}) and (\ref{sec32:eq:V2b}), we have
\begin{lem}\label{sec33:lem:ker}
	\[
		\ker L^{c,\gamma}_{0}=\left\{
		\begin{array}{ll}
		\mathrm{span}\{V_{c,\gamma}^1, V_{c,\gamma}^2, V_{c,\gamma}^3, V_{c,\gamma}^4 \}, & \textrm{if } (c,\gamma)\in I_{2,1},\\
		\mathrm{span}\{ V_{c,\gamma}^{2b}, V_{c,\gamma}^3, V_{c,\gamma}^4 \}, & \textrm{if } (c,\gamma)\in I_{2,3} \textrm{ or }I_{4,3}.\\
		\end{array}
		\right.
	 \]
\end{lem}
\begin{proof}
	Let $V\in\mathbf{X}$ satisfy $L^{c,\gamma}_{0}V=0$. We know that $V$ is given by (\ref{sec31:eq:ker}) for some $d_1,d_2,d_3,d_4\in\mathbb{R}$.
	For convenience we denote $a(x)=a_{c,\gamma}(x)$, $b(x)=b_{c,\gamma}(x)$ and $V^i=V^i_{c,\gamma}$, $i=1,2,2b,3,4$.
	
	By Lemma \ref{lem3_3_ab} and the expressions of $V^1,V^2$ in (\ref{eq_basis}), we have that	
		
	\begin{equation}\label{eq3_3_ker_1}
		V^1_{\theta}(x) = e^{-a(x)}= O(1) \left( \ln \frac{1+x}{3} \right)^{-2} (1-x)^{1+\frac{\bar{U}_{\theta}(1)}{2}}, 
	\end{equation}
	and 
	\begin{equation}\label{eq3_3_ker_2}
		 V^2_{\theta}(x) =e^{-a(x)}\int_{0}^{x}e^{a(s)}ds		
		=  O(1) \left( \ln \frac{1+x}{3} \right)^{-2} (1-x)\left((1-x)^{\frac{\bar{U}_{\theta}(1)}{2}}+1\right). 
	\end{equation}
	
By (\ref{sec31:eq:diff:a}), we also have
	\begin{equation}\label{eq3_3_ker_3}
		 \left|\frac{d}{dx}V^1_{\theta}(x)\right|=\left|e^{-a(x)}a'(x)\right|
		=  O(1) \left( \ln \frac{1+x}{3} \right)^{-2} (1+x)^{-1} (1-x)^{\frac{\bar{U}_{\theta}(1)}{2}},
	\end{equation}
	\begin{equation}\label{eq3_3_ker_4}
		 \left|\frac{d}{dx}V^2_{\theta}(x)\right|=\left| -V^2_{\theta}(x)a'(x) + 1\right| 
		=  O(1)\left( \ln \frac{1+x}{3} \right)^{-2} (1+x)^{-1}\left((1-x)^{\frac{\bar{U}_{\theta}(1)}{2}}+1\right).
	\end{equation}

  If $\bar{U}_{\theta}(1)<- 2$ or $\bar{U}_{\theta}(1)=-2$ with $\eta_2=0$, by \eqref{sec32:eq:V2b} we have 
	\begin{equation}\label{eq3_3_ker_7}
	   \left|V^{2b}_{\theta}(x)\right|=  O(1) \left( \ln \frac{1+x}{3} \right)^{-2}(1-x)^{1+\frac{\bar{U}_{\theta}(1)}{2}},
	\end{equation}
	\begin{equation*}
	   \left|\frac{d}{dx}V^{2b}_{\theta}(x)\right|
				= O(1) \left( \ln \frac{1+x}{3} \right)^{-2} (1+x)^{-1}(1-x)^{\frac{\bar{U}_{\theta}(1)}{2}}.
	\end{equation*}
	Next, by computation we have for $i=1,2,2b$
	\[
	 \frac{d^2}{dx^2}V^i_{\theta}=(V^i_{\theta})'a'(x)+V^i_{\theta}a''(x),\quad  \frac{d^3}{dx^3}V^i_{\theta}=(V^i_{\theta})''a'(x)+2(V^i_{\theta})'a''(x)+V^i_{\theta}a'''(x).
	\]
	Using the definition of $a(x)$ in (\ref{sec31:eq:ab}), there exists some constant $C$, depending on $c,\gamma$, such that
	\begin{equation}\label{eq3_3_ker_8}
	   \left |\frac{d^2}{dx^2}V^i_{\theta}\right|\le C, \quad \left|\frac{d^3}{dx^3}V^i_{\theta}\right|\le C, \quad -\frac{1}{2}<x<\frac{1}{2}, \; i=1,2,2b.
	\end{equation}
	
	Moreover, by Lemma \ref{lem3_3_ab}, and the expressions of $V^3$ in (\ref{eq_basis}), we have
	\begin{equation}\label{eq3_3_ker_9}
		V^3_{\phi}(x)=\int_{0}^{x}e^{-b(t)}dt=O(1)(1-x)^{1+\frac{\bar{U}_{\theta}(1)}{2}}+O(1),
	\end{equation}
	and 
	\begin{equation}\label{eq3_3_ker_10}
	\begin{split}
		&\left|\frac{d}{dx}V^3_{\phi}(x)\right|=e^{-b(x)}= O(1) \left( \ln \frac{1+x}{3} \right)^{-2} (1+x)^{-1}(1-x)^{\frac{\bar{U}_{\theta}(1)}{2}},\\
			&\left|\frac{d^2}{dx^2}V^3_{\phi}(x)\right|=e^{-b(x)}|b'(x)| = O(1)\left( \ln \frac{1+x}{3} \right)^{-2} (1+x)^{-2} (1-x)^{-1+\frac{\bar{U}_{\theta}(1)}{2}}. 
			\end{split}
	\end{equation}

	When $(c,\gamma)\in I_{2,1}$, $\bar{U}(-1)=2$ with $\eta_1=4$, and $\bar{U}(1)>-2$, using estimates (\ref{eq3_3_ker_1})-(\ref{eq3_3_ker_4}),  (\ref{eq3_3_ker_8})-(\ref{eq3_3_ker_10}), and the definition of $V_{c,\gamma}^4$, it is not hard to verify that $V_{c,\gamma}^i\in \mathbf{X}$, $1\le i\le 4$.  It is clear that $\{V_{c,\gamma}^i , 1\le i\le 4\}$ are independent. So $\{V_{c,\gamma}^i, 1\le i\le 4\}$ is a basis of the kernel.

	Similarly, when $(c,\gamma)\in I_{2,3} \textrm{ or }I_{4,3}$, it can be checked that $\mathrm{span}\{V_{c,\gamma}^1, V_{c,\gamma}^2\}=\mathrm{span}\{V_{c,\gamma}^1, V_{c,\gamma}^{2b}\}$, where $V_{c,\gamma}^{2b}$, given by (\ref{sec32:eq:V2b}), is a linear combination of $V_{c,\gamma}^1, V_{c,\gamma}^2$.  So $L^{c,\gamma}_{0}V=0$ implies 
	$$
		V = d_1 V_{c,\gamma}^1 + d_2 V_{c,\gamma}^{2b} + d_3 V_{c,\gamma}^3 + d_4 V_{c,\gamma}^4.
	$$
	It can be checked by estimates (\ref{eq3_3_ker_1}), (\ref{eq3_3_ker_3}), and (\ref{eq3_3_ker_7})-(\ref{eq3_3_ker_10}) that in this case $V_{c,\gamma}^{2b}, V_{c,\gamma}^3, V_{c,\gamma}^4 \in \mathbf{X}$, and  $  V_{c,\gamma}^1 \notin \mathbf{X}$.  So $d_1 V_{c,\gamma}^1\in \mathbf{X}$. This means $d_1 (V_{c,\gamma}^1)_{\theta}\in \mathbf{M}_1$ Thus $d_1=0$.
	The lemma is proved.
\end{proof}

\begin{cor}\label{sec33:cor:all:sol}
	For any $\xi\in\mathbf{Y}$, all solutions of $L^{c,\gamma}_{0}V=\xi$, $V\in\mathbf{X}$, are given by
	\[
		V=W^{c,\gamma}(\xi)+\left\{
		\begin{array}{ll}
			d_1V_{c,\gamma}^1+d_2V_{c,\gamma}^2+d_3V_{c,\gamma}^3+d_4V_{c,\gamma}^4, & \textrm{if }(c,\gamma)\in I_{2,1},\\
			d_2V_{c,\gamma}^{2b}+d_3V_{c,\gamma}^3+d_4V_{c,\gamma}^4, & \textrm{if }(c,\gamma)\in I_{2,3} \textrm{ or }I_{4,3}.
		\end{array}
		\right.
	\]
\end{cor}

Let $l_i$, $1\le i\le 4$ and $l_{2b}$ be the functionals on $\mathbf{X}$ defined by (\ref{sec31:eq:fcnal:l}) and (\ref{eq_l2ab}), and $\mathbf{X}_1=\cap_{i=1}^{4}\ker l_i$, $\mathbf{X}_{2b}=\ker l_{2b}\cap\ker l_3\cap \ker  l_4$.  It can be checked that $\mathbf{X}_1$  and $\mathbf{X}_{2b}$ are closed subspaces of $\mathbf{X}$, and
\begin{equation}\label{sec33:eq:decomp:X}
	\mathbf{X} =\left\{
	\begin{array}{ll}
		\mbox{span} \{ V_{c,\gamma}^{1}, V_{c,\gamma}^{2}, V_{c,\gamma}^{3}, V_{c,\gamma}^{4} \} \oplus\mathbf{X}_1,& \quad (c,\gamma)\in I_{2,1},\\
		\mbox{span} \{ V_{c,\gamma}^{2b}, V_{c,\gamma}^{3}, V_{c,\gamma}^{4} \} \oplus\mathbf{X}_{2b}, & \quad (c,\gamma)\in I_{2,3} \textrm{ or }I_{4,3},
	\end{array}
	\right. 
\end{equation}
with the projection operator $P_1: \mathbf{X}\rightarrow\mathbf{X}_1$, $P_{2b}: \mathbf{X}\rightarrow\mathbf{X}_{2b}$ given by (\ref{sec32:eq:proj}). 

\begin{lem}\label{sec33:lem:iso}
	If $(c,\gamma)\in I_{2,1}$, the operator $ L^{c,\gamma}_{0}: \mathbf{X}_1\rightarrow\mathbf{Y}$ is an isomorphism.  
		
	If $(c,\gamma)\in I_{2,3} \textrm{ or }I_{4,3}$, the operator $ L^{c,\gamma}_{0}: \mathbf{X}_{2b}\rightarrow\mathbf{Y}$ is an isomorphism. \end{lem}

\begin{proof}
	By Corollary \ref{sec33:cor:all:sol} and Lemma \ref{sec33:lem:ker}, $L^{c,\gamma}_{0}:\mathbf{X}\rightarrow\mathbf{Y}$ is surjective and $\ker L^{c,\gamma}_{0}$  is given by Lemma \ref{sec33:lem:ker}. The conclusion of the lemma then follows in view of the direct sum property (\ref{sec33:eq:decomp:X}).	
\end{proof}

\begin{lem}\label{sec33:lem:V:smooth}
	$V_{c,\gamma}^1, V_{c,\gamma}^2\in C^{\infty}(K,\mathbf{X})$ for compact $K\subset I_{2,1}$. 

	$V_{c,\gamma}^{2b} \in C^{\infty}(K,\mathbf{X})$ for compact $K\subset I_{2,3}$ or $I_{4,3}$. 
	
	$V_{c,\gamma}^3, V_{c,\gamma}^4 \in C^{\infty}(K,\mathbf{X})$ for compact $K\subset I_{k,l}$ with $(k,l)\in A_2$. 
\end{lem}
\begin{proof}
It is clear that $V^4_{c,\gamma}\in  C^{\infty}(K,\mathbf{X})$ for all compact set $K$ described as in the lemma.
    
    Let $\alpha=(\alpha_1,\alpha_2,\alpha_3)$ be a multi-index where $\alpha_i\ge 0$, $i=1,2,3$, and $j\ge 0$. 
    For convenience we denote $a(x)=a_{c,\gamma}(x)$, $b(x)=b_{c,\gamma}(x)$ and $V^i=V^i_{c,\gamma}$, $i=1,2,2b,3$.

		 Using Proposition \ref{propA_1} part (2), we have that for all $|\alpha|+j\ge 1$ and $(c,\gamma)\in K$,
		 \begin{equation}\label{sec33:eq:Vsmooth:temp}
		\pt_c^j \pt_\gamma^i a(x) = \pt_c^j \pt_\gamma^i b(x) = \int_{0}^{x}\frac{1}{1-s^2} \pt_c^j \pt_\gamma^i U^{c,\gamma}(s)ds = O(1)|\ln(1-x)|. 
	\end{equation}
	
	(1) If $K\subset I_{2,1}$, we have $U_{\theta}^{c,\gamma}(-1)=2$ with $\eta_1=4$ and $U_{\theta}^{c,\gamma}(1)>-2$. 

	Let $2\bar{\epsilon}:=\max\{0, -\frac{1}{2}U^{c,\gamma}_{\theta}(1) \mid (c,\gamma)\in K\}$, then $ \bar{\epsilon}<\epsilon$. 
	Using the expressions of $V^1,V^2$ in (\ref{eq_basis}), Lemma \ref{lem3_3_ab}, estimates (\ref{eq3_3_ker_1}), (\ref{eq3_3_ker_2}), (\ref{sec33:eq:Vsmooth:temp}) and Proposition \ref{propA_1} (2), we have that for all $|\alpha|+j\ge 1$ and $(c,\gamma)\in K$, 
	\[
	\begin{split}
		\left| \pt_c^\alpha \pt_\gamma^j V^1_{\theta}(x) \right| 
		& = e^{-a(x)}O\left( \left|\ln(1-x)\right|^{|\alpha|+j} \right) \\
		& =O(1) \left( \ln \frac{1+x}{3} \right)^{-2} (1-x)^{1-2\bar{\epsilon}}\left|\ln(1-x)\right|^{|\alpha|+j}, \\
		\left| \pt_c^\alpha \pt_\gamma^j V^2_{\theta}(x) \right| & = e^{-a(x)}\left|\int_{0}^{x}e^{a(s)}ds\right|O\left( \left|\ln(1-x)\right|^{|\alpha|+j}\right)\\
		&=O(1) \left( \ln \frac{1+x}{3} \right)^{-2} (1-x)^{1-2\bar{\epsilon}}\left|\ln(1-x)\right|^{|\alpha|+j},
		\end{split}
	\]
	and 
	\[
	\begin{split}
		\left| \frac{d}{dx}\pt_c^\alpha \pt_\gamma^j V^1_{\theta}(x) \right| 
		& = e^{-a(x)}|a'(x)|O\left( \left|\ln(1-x)\right|^{|\alpha|+j} \right) \\
		& =O(1) \left( \ln \frac{1+x}{3} \right)^{-2} (1+x)^{-1}(1-x)^{-2\bar{\epsilon}}\left|\ln(1-x)\right|^{|\alpha|+j},\\
		\left| \frac{d}{dx}\pt_c^\alpha \pt_\gamma^j V^2_{\theta}(x) \right| 
		& = |-V^2_{\theta}a'(x)+1|O\left( \left|\ln(1-x)\right|^{|\alpha|+j} \right) \\
		& =O(1) \left( \ln \frac{1+x}{3} \right)^{-2} (1+x)^{-1}(1-x)^{-2\bar{\epsilon}}\left|\ln(1-x)\right|^{|\alpha|+j}.
	\end{split}
	\]
	From the above we can see that for all $|\alpha|+j\ge 1$, there exists some constant $C=C(\alpha,j,K)$, such that for $i=1,2$
	\begin{eqnarray*}
		\left| \left( \ln \frac{1+x}{3} \right) (1-x)^{-1+2\epsilon}\pt_c^\alpha \pt_\gamma^j  V^i_{\theta}(x)\right|\le C,  \\
		\left| \left( \ln \frac{1+x}{3} \right)^2 (1+x) (1-x)^{2\epsilon}\frac{d}{dx}\pt_c^\alpha \pt_\gamma^j V^i_{\theta}(x)\right|\le C.
	\end{eqnarray*}
	We also have that for $|\alpha|+j\ge 1$, $\pt_c^\alpha \pt_\gamma^j  V^i_{\theta}(1) = \pt_c^\alpha \pt_\gamma^j  V^i_{\theta}(-1) = 0$, $i=1,2$. 
	
	Next, using the definition of $a(x)$ in (\ref{sec31:eq:ab}), there exists some constant $C=C(K)$, such that  
	\[
	    \left| \frac{d^2}{dx^2}\pt_c^\alpha \pt_\gamma^jV^i_{\theta}\right|\le C, \quad \left| \frac{d^3}{dx^3}\pt_c^\alpha \pt_\gamma^jV^i_{\theta}\right|\le C, \quad -\frac{1}{2}<x<\frac{1}{2},\quad i=1,2.
	\]
	The above implies that for all $|\alpha|+j \ge 1$, $\pt_c^\alpha \pt_\gamma^j  V^i_{\theta}\in \mathbf{M}_1$, $i=1,2$, so $V^1,V^2\in C^{\infty}(K, \mathbf{X})$.

	
	
	(2) If $K\subset I_{2,3}$ or $I_{4,3}$, we have $U_{\theta}^{c,\gamma}(-1)=2$ with $\eta_1=4$, and $U^{c,\gamma}(1)\in(-3,-2)$ or $U^{c,\gamma}(1)=-2$ with $\eta_2=0$. 
	
%
	
	In this case $\gamma = \gamma^-(c_1,c_2,c_3)$. Using the expressions of $V^{2b}$ in (\ref{sec32:eq:V2b}), Lemma \ref{lem3_3_ab},  the estimates (\ref{eq3_3_ker_7}), (\ref{sec33:eq:Vsmooth:temp}) and Proposition \ref{propA_1}, we have that for all $|\alpha|\ge 1$, 
	\[
		\left| \pt_c^\alpha V^{2b}_{\theta}(x) \right|
		 = O(1)\left( \ln \frac{1+x}{3} \right)^{-2} (1-x)^{1-2\bar{\epsilon}}\left|\ln(1-x)\right|^{|\alpha|}, 
	\]
	and 
	\[
	   \left| \frac{d}{d x}\pt_c^\alpha V^{2b}_{\theta}(x) \right|
		 = O(1)\left( \ln \frac{1+x}{3} \right)^{-2} (1+x)^{-1}(1-x)^{-2\bar{\epsilon}}\left|\ln(1-x)\right|^{|\alpha|}.  
	\]
	From the above we see that for any $|\alpha|\ge 1$, there exists some constant $C=C(\alpha,K)$, such that for all $(c,\gamma)\in K$,
	\[
		\left| \left( \ln \frac{1+x}{3} \right) (1-x)^{-1+2\epsilon}\pt^\alpha_{c} V^{2b}_{\theta}(x)\right|\le C,\quad 
		\left| \left( \ln \frac{1+x}{3} \right)^2 (1+x) (1-x)^{2\epsilon}\frac{d}{dx}\pt^\alpha_c V^{2b}_{\theta}(x)\right|\le C. 
	\]
	We also have that for $|\alpha| \ge 1$, 
	$
		\pt^\alpha_{c} V^{2b}_{\theta}(1) = \pt^\alpha_{c} V^{2b}_{\theta}(-1) = 0. 
	$
	
	Similarly as part (1), we have
	\[
	    \left| \frac{d^2}{dx^2}\pt_c^\alpha \pt_\gamma^jV^{2b}_{\theta}\right|\le C, \quad \left| \frac{d^3}{dx^3}\pt_c^\alpha \pt_\gamma^jV^{2b}_{\theta}\right|\le C, \quad \textrm{ for all }-\frac{1}{2}<x<\frac{1}{2}.
	\]
	The above implies that  for all $|\alpha|\ge 1$, $\pt^\alpha_{c} V^{2b}_{\theta}\in \mathbf{M}_1$, so $V^{2b}\in C^{\infty}(K,\mathbf{X})$. 
	
	(3) Let $K$ be a subset of $I_{k,l}$ with $(k,l)\in A_2$. 
	
		Using the expressions of $V^3$ in (\ref{eq_basis}), Lemma \ref{lem3_3_ab}, estimates (\ref{eq3_3_ker_9}), (\ref{eq3_3_ker_10}), (\ref{sec33:eq:Vsmooth:temp}) and Proposition \ref{propA_1}, we have that for all $|\alpha|+j\ge 1$,
		
	\[
	   \left|\partial_{c}^{\alpha}\partial_{\gamma}^{j}V^3_{\phi}\right|=O(1)\left((1-x)^{1+\frac{\bar{U}_{\theta}(1)}{2}}+1\right)|\ln(1-x)|^{|\alpha|+j},
	\]
	\[
	    \left|\frac{d}{dx}\partial_{c}^{\alpha}\partial_{\gamma}^{j}V^3_{\phi}(x)\right|= O(1) \left( \ln \frac{1+x}{3} \right)^{-2} (1+x)^{-1}(1-x)^{\frac{\bar{U}_{\theta}(1)}{2}}|\ln(1-x)|^{|\alpha|+j},
	\]
	\[
	    \left|\frac{d^2}{dx^2}\partial_{c}^{\alpha}\partial_{\gamma}^{j}V^3_{\phi}(x)\right|=O(1)\left( \ln \frac{1+x}{3} \right)^{-2} (1+x)^{-2} (1-x)^{-1+\frac{\bar{U}_{\theta}(1)}{2}}|\ln(1-x)|^{|\alpha|+j}.
	\]
	Since $\epsilon>\max\{0,-1-\frac{\bar{U}_{\theta}(1)}{2}\}$, there exists some $C=C(\alpha, j, K)$ such that for all  $(c,\gamma)\in K$,
	\[
	   | (1-x^2)^{\epsilon}\partial_{c}^{\alpha}\partial_{\gamma}^{j}V^3_{\phi}|\le C, \quad \left|(1-x^2)^{1+\epsilon}\frac{d}{dx}\partial_{c}^{\alpha}\partial_{\gamma}^{j}V^3_{\phi}\right|\le C, \quad \quad \left|(1-x^2)^{2+\epsilon}\frac{d^2}{dx^2}\partial_{c}^{\alpha}\partial_{\gamma}^{j}V^3_{\phi}\right|\le C.
	\]
	The above implies that for any $|\alpha|+j\ge 1$, $\partial_{c}^{\alpha}\partial_{\gamma}^{j}V^3_{\phi}\in \mathbf{M}_2$, so $V^3\in C^{\infty}(K,\mathbf{X})$. 
\end{proof}
\\

Similar arguments as in Lemma \ref{sec32:lem:combine:bdd} imply the following lemma. 
\begin{lem}\label{sec33:lem:combine:bdd}
	(i) If $ K\subset \subset I_{2,1}$,  then there exists $C=C(K)>0$ such that for all $(c,\gamma)\in K$, $\beta:=(\beta_1,\beta_2,\beta_3,\beta_4)\in \mathbb{R}^4$, and $V\in \mathbf{X}_1$, 
	\[
		\| V \|_{\mathbf{X}} + | \beta | 
		\leq C \| \sum_{i=1}^{4}\beta_i V^i_{c,\gamma} + V \|_{\mathbf{X}}. 
	\] 

	(ii) If $K\subset \subset I_{2,3}$ or $I_{4,3}$, then there exists $C=C(K)>0$ such that for all $(c,\gamma)\in K$,  $(\beta_2,\beta_3,\beta_4)\in \mathbb{R}^3$, and $V\in \mathbf{X}_{2b}$, 
	\[
		\| V \|_{\mathbf{X}} + | (\beta_2, \beta_3,\beta_4) | 
		\leq C \| \beta_2 V^{2b}_{c,\gamma} + \beta_3 V^3_{c,\gamma} +\beta_4V^4_{c,\gamma}+ V \|_{\mathbf{X}}. 
	\] 
\end{lem}

\noindent 
{\bf Proof of Theorem \ref{sec33:thm:1}.}
Define a map $F: K \times \mathbb{R}^4 \times \mathbf{X}_1 \to \mathbf{Y}$ by
\[
	F(c,\gamma,\beta, V) = G(c,\gamma,\sum_{i=1}^{4}\beta_i V^i_{c,\gamma} + V).
\]
By Proposition \ref{sec33:prop}, $G$ is a $C^{\infty}$ map from $K\times \mathbf{X}$ to $\mathbf{Y}$. Let $\tilde{U} = \tilde{U}(c,\gamma,\beta, V) = \sum_{i=1}^{4}\beta_i V^i_{c,\gamma}+ V$. Using Lemma \ref{sec33:lem:V:smooth}, we have $\tilde{U}\in C^{\infty}(K \times \mathbb{R}^4 \times \mathbf{X}_1, \mathbf{X})$. So it concludes that $F \in C^{\infty}(K \times \mathbb{R}^4 \times \mathbf{X}_1, \mathbf{Y})$. 

Next, by definition $F(c,\gamma,0, 0) = 0$ for all $(c,\gamma)\in K$. Fix some $(\bar{c}, \bar{\gamma}) \in K$, using Lemma \ref{sec33:lem:iso}, we have $F_{V}(\bar{c}, \bar{\gamma}, 0,0)=L_0^{\bar{c}, \bar{\gamma}}:\mathbf{X}_1\to \mathbf{Y}$ is an isomorphism.

Applying Theorem \ref{thm:IFT}, there exist some $\delta>0$ depending only on $K$ and a unique $V\in C^{\infty}(B_{\delta}(\bar{c}, \bar{\gamma})\times B_{\delta}(0), \mathbf{X}_1 )$, such that
\[
	F(c,\gamma,\beta, V(c,\gamma,\beta) ) = 0, \quad \forall (c,\gamma)\in B_{\delta}(\bar{c}, \bar{\gamma}), \beta \in B_{\delta}(0),
\]
and 
\[
	V(\bar{c}, \bar{\gamma},0)=0.
\]
The uniqueness part of Theorem \ref{thm:IFT} holds in the sense that there exists some $0<\bar{\delta}<\delta$,  such that $B_{\bar{\delta}}(\bar{c},\bar{\gamma},0,0)\cap F^{-1}(0) \subset  \{(c,\gamma,\beta,V(c,\gamma,\beta) ) |(c,\gamma) \in B_{\delta}(\bar{c}, \bar{\gamma}), \beta\in B_{\delta}(0)\}$.

\noindent
{\bf Claim.} There exists some $0 < \delta_1 < \frac{\bar{\delta}}{2}$, such that $V(c,\gamma,0)=0$ for every $(c,\gamma)\in B_{\delta_1}(\bar{c}, \bar{\gamma})$.

\noindent 
{\bf Proof of the Claim.}
Since $V(\bar{c}, \bar{\gamma},0) = 0$ and $V(c,\gamma,0)$ is continuous in $(c,\gamma)$, there exists some $0<\delta_1 < \frac{\bar{\delta}}{2}$, such that for all $(c,\gamma)\in B_{\delta_1}(\bar{c}, \bar{\gamma})$, $(c,\gamma,0, V(c,\gamma,0))\in B_{\bar{\delta}(\bar{c}, \bar{\gamma},0,0)}$. We know that for all $(c,\gamma)\in B_{\delta_1}(\bar{c},\bar{\gamma})$,
\[
	F(c,\gamma,0, 0) = 0, 
\]
and 
\[
	F(c,\gamma,0, V(c,\gamma,0))=0.
\]
By the above mentioned uniqueness result, $V(c,\gamma,0)=0$, for every $(c,\gamma)\in B_{\delta_1}(\bar{c}, \bar{\gamma})$.

Now we have $V\in C^{\infty}(B_{\delta_1}(\bar{c}, \bar{\gamma})\times B_{\delta_1}(0), \mathbf{X}_1(\bar{c}, \bar{\gamma}) )$, and 
\[
	F(c,\gamma,\beta,V(c,\gamma,\beta))=0, \quad \forall (c,\gamma)\in B_{\delta_1}(\bar{c}, \bar{\gamma}), \beta\in B_{\delta_1}(0).
\]
i.e. for any $(c,\gamma)\in B_{\delta_1}(\bar{c}, \bar{\gamma})$, $\beta\in B_{\delta_1}(0)$
\[
	G(c,\gamma,\sum_{i=1}^{4}\beta_i V^i_{c,\gamma}+ V(c,\gamma,\beta) )=0. 
\]
Take derivative of the above with respect to $\beta_i$ at $(c,\gamma,0)$, $1\le i\le 4$, we have
\[
	G_{\tilde{U}}(c,\gamma,0)(V^i_{c,\gamma}+\partial_{\beta_i}V(c,\gamma,0))=0.
\]
Since $G_{\tilde{U}}(c,\gamma,0)V^i_{c,\gamma}=0$ by Lemma \ref{sec33:lem:ker}, we have 
\[
	G_{\tilde{U}}(c,\gamma,0)\partial_{\beta_i}V(c,\gamma,0)=0.
\]
But $\partial_{\beta_i}V(c,\gamma,0)\in C^{\infty}(\mathbf{X}_1)$, so 
\[
	\partial_{\beta_i}V(c,\gamma,0)=0, \quad 1\le i\le 4.
\]
Since $K$ is compact, we can take $\delta_1$ to be a universal constant for each  $(c,\gamma)\in K$. So we have proved the existence of $V$ in Theorem \ref{sec33:thm:1}. 

Next, let $(c,\gamma)\in B_{\delta_1}(\bar{c}, \bar{\gamma})$. Let $\delta'$ be a small constant to be determined.  For any $U$ satisfies equation (\ref{sec31:eq:NSE}) with $U-U^{c,\gamma}\in \mathbf{X}$, and $\|U-U^{c,\gamma}\|_{ \mathbf{X}}\le \delta'$ there exist some $\beta\in \mathbb{R}^4$ and $V^*\in \mathbf{X}_1$ such that
\[
	U-U^{c,\gamma} =\sum_{i=1}^{4}\beta_i V^i_{c,\gamma}+ V^*.
\]
Then by Lemma \ref{sec33:lem:combine:bdd}, there exists some constant $C>0$ such that
\[
	\frac{1}{C}(|\beta| + \|V^*\|_{\mathbf{X}})
	\le \|\sum_{i=1}^{4}\beta_i V^i_{c,\gamma}+ V^*\|_{\mathbf{X}}\le \delta'.
\]
This gives $\|V^*\|_{\mathbf{X}}\le C\delta'$.

Choose $\delta'$ small enough such that $C\delta'<\delta_1$. We have the uniqueness of $V^*$. 
So  $V^*=V(c,\gamma,\beta)$ in (\ref{sec33:eq:U1}).
The theorem is proved. 
\qed

Theorem \ref{sec33:thm:2} can be proved by replacing $\mathbf{X}_1$ by $\mathbf{X}_{2b}$, $\sum_{i=1}^{4}\beta_i V^i_{c,\gamma}$ by $(\beta_2 V^{2b}_{c,\gamma}+ \beta_3 V^3_{c,\gamma}+ \beta_4 V^4_{c,\gamma})$ respectively. 
\\
\\

The case $\bar{U}_{\theta}(1)=-2$ with $\eta_2=-4$,  ($\bar{U}_{\theta}(-1)<3$, $\bar{U}_{\theta}(-1)\ne 2$ or $\bar{U}_{\theta}(-1)=2$ with $\eta_1=0$) can conclude the following theorems and the theorems can be proved similarly. 
\addtocounter{thm}{-2}
\renewcommand{\thethm}{\thesection.\arabic{thm}'}%
\begin{thm}\label{sec33:thm:1'}
	For every compact subset $K\subset I_{3,1}$ , there exist $\delta = \delta(K) > 0$, and $V\in C^{\infty}( K\times B_{\delta}(0), \mathbf{X}_1)$ satisfying $V(c,\gamma,0)=0$ and $\displaystyle \frac{\partial V}{\partial \beta_i}\big |_{\beta=0}=0$, $1\le i\le 4$, $\beta=(\beta_1,\beta_2,\beta_3,\beta_4)$, such that 
	\begin{equation}\label{sec33:eq:U1'}
		U = U^{c,\gamma} + \sum_{i=1}^{4} \beta_i V_{c,\gamma}^i + V(c,\gamma,\beta)
	\end{equation}
	satisfies equation (\ref{sec31:eq:NSE}) with $\hat{c}_1= c_1 - \frac{1}{2}\psi[\tilde{U}_{\phi}](-1)$, $\hat{c}_2 = c_2 - \frac{1}{2}\psi[\tilde{U}_{\phi}](1)$, $\hat{c}_3 = c_3 +\frac{1}{2}(\varphi_{c,\gamma}[\tilde{U}_{\theta}])''(0)$.

	Moreover, there exists some $\delta'=\delta'(K)>0$, such that if $\| U - U^{c,\gamma } \|_{\mathbf{X}} < \delta'$, $(c,\gamma) \in K$,  and $U$ satisfies  equation (\ref{sec31:eq:NSE}) with some constants $\hat{c}_1, \hat{c}_2, \hat{c}_3$, then (\ref{sec33:eq:U1'}) holds for some $|\beta|< \delta$. 
\end{thm}

Recall that $V_{c,\gamma}^{2a}$ is defined by (\ref{sec32:eq:V2a}), $l_{2a}$ be the bounded linear functional defined by (\ref{eq_l2ab}). Define $\mathbf{X}_{2a}:=\ker l_{2a} \cap \ker l_3 \cap \ker l_4$. Then $\mathbf{X}_{2a}$ is independent of $(c,\gamma)$. 

\begin{thm}\label{sec33:thm:2'}
	For every compact subset $K$ of $I_{3,2}$ or $I_{4,2}$, for every $(c,\gamma)\in K$, there exist $\delta = \delta(K) > 0$, and $V\in C^{\infty}( K\times B_{\delta}(0), \mathbf{X}_{2a})$ satisfying $V(c,\gamma,0)=0$ and $\displaystyle \frac{\partial V}{\partial \beta_i}\big |_{\beta=0}=0$, $i=2,3,4$, $\beta=(\beta_2,\beta_3,\beta_4)$, such that 
	\begin{equation}\label{sec33:eq:U2'}
		U = U^{c,\gamma} + \beta_2 V_{c,\gamma}^{2a} + \beta_3 V_{c,\gamma}^3 + \beta_4 V_{c,\gamma}^4+ V(c,\gamma,\beta)
	\end{equation}
	satisfies equation \eqref{sec31:eq:NSE} with $\hat{c}_1= c_1 - \frac{1}{2}\psi[\tilde{U}_{\phi}](-1)$, $\hat{c}_2 = c_2 - \frac{1}{2}\psi[\tilde{U}_{\phi}](1)$, $\hat{c}_3 = c_3 +\frac{1}{2}(\varphi_{c,\gamma}[\tilde{U}_{\theta}])''(0)$.

	Moreover, there exists some $\delta'=\delta'(K)>0$, such that if $\| U - U^{c,\gamma } \|_{\mathbf{X}} < \delta'$, $(c,\gamma) \in K$,  and $U$ satisfies  equation \eqref{sec31:eq:NSE} with some constants $\hat{c}_1, \hat{c}_2, \hat{c}_3$, then \eqref{sec33:eq:U2'} holds for some $|\beta|< \delta$. 
\end{thm}
\renewcommand{\thethm}{\thesection.\arabic{thm}}%

\section{Existence of axisymmetric, with swirl solutions around \texorpdfstring{$\bm{U^{c,\gamma}}$}{}, when \texorpdfstring{$\bm{(c,\gamma)\in I_{4,3}}$}{}}\label{sec_5}

If $(c,\gamma)\in I_{4,1}$, then $\bar{U}_{\theta}(-1)=2$ with $\eta_1=4$ and $\bar{U}_{\theta}(1)=-2$ with $\eta_2=-4$.

Let $0<\epsilon<\frac{1}{2}$, define
\[ 
\begin{split}
	&
	\begin{split}
	\mathbf{M}_1 = & \mathbf{M}_1(\epsilon)
		:=  \Big\{  \tilde{U}_\theta \in C^3(-\frac{1}{2},\frac{1}{2})\cap C^1(-1,1)\cap C[-1,1]  \mid  \tilde{U}_\theta(1)=\tilde{U}_\theta(-1)=0,\\
		& \| \ln \frac{1+x}{3} \ln \frac{1-x}{3} \tilde{U}_\theta \|_{L^\infty(-1,1)} < \infty, \\
		& \| (1-x^2) \left( \ln\frac{1+x}{3} \right)^2 \left( \ln\frac{1-x}{3} \right)^2 \tilde{U}'_\theta \|_{L^\infty(-1,1)} < \infty, \\
		& \| \tilde{U}_{\theta}'' \|_{L^{\infty}(-\frac{1}{2},\frac{1}{2}) } < \infty,  
		\| \tilde{U}_{\theta}''' \|_{L^{\infty}(-\frac{1}{2},\frac{1}{2})} < \infty \Big\},
	\end{split}\\
	&
	\begin{split}
	\mathbf{M}_2=& \mathbf{M}_2(\epsilon)
		:= \Big\{  \tilde{U}_\phi \in  C^2(-1, 1) \mid  (1-x^2)^{\epsilon}\|\tilde{U}_\phi \|_{L^\infty(-1,1)} < \infty, \\ 
		& \|(1-x^2)^{1+\varepsilon} \tilde{U}_\phi'\|_{L^\infty(-1,1)} < \infty, 
		\|(1-x^2)^{2+\varepsilon} \tilde{U}_\phi'' \|_{L^\infty(-1,1)} <\infty \Big\}, 
	\end{split}
\end{split}
\]
with the following norms accordingly
\[
\begin{split}
	& \|\tilde{U}_\theta\|_{\mathbf{M}_1} := \|\ln\frac{1+x}{3}\ln\frac{1-x}{3}\tilde{U}_\theta\|_{L^\infty(-1,1)} \\
	& \qquad + \|(1-x^2)\left(\ln\frac{1+x}{3}\right)^2\left(\ln\frac{1-x}{3}\right)^2\tilde{U}_\theta'\|_{L^\infty(-1,1)} \\
	& \qquad + ||\tilde{U}_{\theta}''||_{L^{\infty}(-\frac{1}{2},\frac{1}{2})}+  ||\tilde{U}_{\theta}'''||_{L^{\infty}(-\frac{1}{2},\frac{1}{2})}, \\
	& \|\tilde{U}_\phi\|_{\mathbf{M}_2}:=  \|(1-x^2)^{\epsilon} \tilde{U}_\phi\|_{L^\infty(-1,1)} + \|(1-x^2)^{1+\varepsilon} \tilde{U}_\phi'\|_{L^\infty(-1,1)}  + \|(1-x^2)^{2+\varepsilon} \tilde{U}_\phi'' \|_{L^\infty(-1,1)} .
\end{split}
\]

Next, define the following function spaces:
\[
\begin{split}
	& 
	\begin{split}
		\mathbf{N}_1= &\mathbf{N}_{1}(\epsilon):= \left\{  \xi_\theta \in C^2( -\frac{1}{2},\frac{1}{2})\cap C[-1,1] \mid  \xi_\theta(1)=\xi_{\theta}(-1)=\xi''_{\theta}(0)=0,\right. \\ 
		& \left.  \|\left(\ln\frac{1+x}{3}\right)^2\left(\ln\frac{1-x}{3}\right)^2\xi_\theta\|_{L^\infty(-1,1)} < \infty, ||\xi_{\theta}'||_{L^{\infty}(-\frac{1}{2},\frac{1}{2})}<\infty,  ||\xi_{\theta}''||_{L^{\infty}(-\frac{1}{2},\frac{1}{2})} \right\},
	\end{split} \\
	& 
	\mathbf{N}_2=\mathbf{N}_2(\epsilon):= \left\{  \xi_\phi \in C(-1, 1) \mid  \|(1-x^2)^{1+\varepsilon} \xi_\phi\|_{L^\infty(-1,1)} < \infty  \right\}, 
\end{split}
\]
with the following norms accordingly 
\[
\begin{split}
	&\|\xi_{\theta}\|_{\mathbf{N}_1}=\|\left(\ln\frac{1+x}{3}\right)^2\left(\ln\frac{1-x}{3}\right)^2\xi_{\theta}\|_{L^\infty(-1,1)}+||\xi_{\theta}'||_{L^{\infty}(-\frac{1}{2},\frac{1}{2})}+||\xi_{\theta}''||_{L^{\infty}(-\frac{1}{2},\frac{1}{2})},\\
	&\|\xi_{\phi}\|_{\mathbf{N}_2}=\|(1-x^2)^{1+\epsilon}\xi_{\phi}\|_{L^\infty(-1,1)}. 
	\end{split}
\]
Then let $\mathbf{X} := \{ \tilde{U} = (\tilde{U}_\theta, \tilde{U}_\phi ) \mid \tilde{U}_\theta \in \mathbf{M}_1, \tilde{U}_{\phi} \in \mathbf{M}_2\}$ with norm $ \| \tilde{U} \|_{\mathbf{X}} = \| \tilde{U}_\theta \|_{\mathbf{M}_1} + \| \tilde{U}_\phi \|_{\mathbf{M}_2}$, $\mathbf{Y} := \{ \xi = ( \xi_\theta, \xi_\phi ) \mid \xi_\theta \in \mathbf{N}_1, \xi_\phi \in \mathbf{N}_2 \}$, with norm $ \| \xi \|_{\mathbf{Y}} = \| \xi_\theta \|_{\mathbf{N}_1} + \| \xi_\phi \|_{\mathbf{N}_2}$. It can be proved that $\mathbf{M}_1$, $\mathbf{M}_2$, $\mathbf{N}_1$, $\mathbf{N}_2$, $\mathbf{X}$ and $\mathbf{Y}$ are Banach spaces.

Let $l_i:\mathbf{X}\to \mathbb{R}$, $1\le i\le 4$, be the bounded linear functionals defined by (\ref{sec31:eq:fcnal:l}) for each $V\in \mathbf{X}$. Let $\mathbf{X}_1:=\cap_{i=1}^4 \ker l_i$. It can be seen that $\mathbf{X}_1$ is independent of $(c,\gamma)$. 

\begin{thm}\label{sec34:thm:1}
	For every compact subset $K\subset I_{4,1}$, for every $(c,\gamma)\in K$, there exist $\delta = \delta(K) > 0$, and $V\in C^{\infty}( K\times B_{\delta}(0), \mathbf{X}_1)$ satisfying $V(c,\gamma,0)=0$ and $\displaystyle \frac{\partial V}{\partial \beta_i}\big |_{\beta=0}=0$, $1\le i\le 4$, $\beta=(\beta_1,\beta_2,\beta_3,\beta_4)$, such that 
	\begin{equation}\label{sec34:eq:U1}
		U = U^{c,\gamma} + \sum_{i=1}^{4} \beta_i V_{c,\gamma}^i + V(c,\gamma,\beta)
	\end{equation}
	satisfies equation (\ref{sec31:eq:NSE}) with $\hat{c}_1= c_1 - \frac{1}{2}\psi[\tilde{U}_{\phi}](-1)$, $\hat{c}_2 = c_2 - \frac{1}{2}\psi[\tilde{U}_{\phi}](1)$, $\hat{c}_3 = c_3 +\frac{1}{2}(\varphi_{c,\gamma}[\tilde{U}_{\theta}])''(0)$.

	Moreover, there exists some $\delta'=\delta'(K)>0$, such that if $\| U - U^{c,\gamma } \|_{\mathbf{X}} < \delta'$, $(c,\gamma) \in K$,  and $U$ satisfies  equation (\ref{sec31:eq:NSE}) with some constants $\hat{c}_1, \hat{c}_2, \hat{c}_3$, then (\ref{sec34:eq:U1}) holds for some $|\beta|< \delta$. 
\end{thm}

For $\tilde{U}_{\phi}\in \mathbf{M}_2$, let $\psi[\tilde{U}_{\phi}](x)$ be defined by (\ref{sec31:eq:psi}). Let $K$ be a compact subset in $I_{4,1}$. Define a map $G = G(c,\gamma,\tilde{U})$ on $K\times \mathbf{X}$ by (\ref{sec31:eq:G}). 

\begin{prop}\label{sec34:prop}
	The map $G$ is in $C^{\infty}(K\times \mathbf{X}, \mathbf{Y})$ in the sense that $G$ has  continuous Fr\'{e}chet derivatives of every order. Moreover, the Fr\'{e}chet derivative of $G$ with respect to $\tilde{U}$ at $(c,\gamma,\tilde{U})\in K\times \mathbf{X}$ is given by the linear bounded operator $L^{c,\gamma}_{\tilde{U}}: \mathbf{X}\rightarrow \mathbf{Y}$ defined as  in (\ref{sec31:eq:Linear}).
\end{prop}
To prove Proposition \ref{sec34:prop}, we first prove the following lemmas:

\begin{lem}\label{sec34:lem:A:well-def}
	For every $(c,\gamma)\in K$, $A(c,\gamma,\cdot): \mathbf{X}\to \mathbf{Y}$ defined by (\ref{sec31:eq:A}) is a well-defined bounded linear operator.
\end{lem}

\begin{proof}
	In the following, $C$ denotes a universal constant which may change from line to line. We denote $l=l_{c,\gamma}[\tilde{U}_{\theta}]$ defined by (\ref{sec31:eq:l}), and $A=A(c,\gamma,\cdot)$ for some fixed $(c,\gamma)\in K$. We make use of the property of $\bar{U}_{\theta}$ that $\bar{U}_{\theta}\in C^2(-1,1)\cap L^{\infty}(-1,1)$, and $\bar{U}_{\theta}=2+4\left(\ln \frac{1+x}{3}\right)^{-1}+O(1)\left(\ln \frac{1+x}{3}\right)^{-2}=-2-4\left(\ln \frac{1-x}{3}\right)^{-1}+O(1)\left(\ln \frac{1-x}{3}\right)^{-2}$.

	$A$ is clearly linear. For every $\tilde{U}\in\mathbf{X}$, we prove that $A\tilde{U}$ defined by (\ref{sec31:eq:A}) is in $\mathbf{Y}$ and there exists some constant $C$ such that $\|A\tilde{U}\|_{\mathbf{Y}}\le C\|\tilde{U}\|_{\mathbf{X}}$ for all $\tilde{U}\in \mathbf{X}$.

	By computation, 
	\[
	    l'(x)=(1-x^2)\tilde{U}''_{\theta}+\bar{U}_{\theta}\tilde{U}'_{\theta}+(2+\bar{U}'_{\theta})\tilde{U}_{\theta},
	\]
	\[
		l''(x)=(1-x^2)\tilde{U}'''_{\theta}+(\bar{U}_{\theta}-2x)\tilde{U}''_{\theta}+2(\bar{U}'_{\theta}+1)\tilde{U}'_{\theta}+\bar{U}''_{\theta}\tilde{U}_{\theta}.
	\]
	
	By the fact that $\tilde{U}_{\theta}\in \mathbf{M}_1$, we have $|l''(0)|\le C||\tilde{U}_{\theta}||_{\mathbf{M}_1}$. So for $-1<x<1$, we have
	\[
	\begin{split}
		& \quad |\left( \ln \frac{1+x}{3} \right)^2 \left( \ln \frac{1-x}{3} \right)^2 A_{\theta}| \\
		& \le |\left( \ln \frac{1+x}{3} \right)^2 \left( \ln \frac{1-x}{3} \right)^2 (1-x^2) \tilde{U}'_{\theta}|+ \left| (2x+\bar{U}_{\theta})\left( \ln \frac{1+x}{3} \right)^2 \left( \ln \frac{1-x}{3} \right)^2 \right| |\tilde{U}_{\theta}|\\
		& \quad +\frac{1}{2} \left( \ln \frac{1+x}{3} \right)^2 \left( \ln \frac{1-x}{3} \right)^2 (1-x^2)|l''(0)|\\ 
		& \le C\|\tilde{U}_{\theta}\|_{\mathbf{M}_1},
	\end{split}
	\]
	where we have used the property that that $\bar{U}_{\theta}=2+4\left(\ln \frac{1+x}{3}\right)^{-1}+O(1)\left(\ln \frac{1+x}{3}\right)^{-2}$ $=-2-4\left(\ln \frac{1-x}{3}\right)^{-1}+O(1)\left(\ln \frac{1-x}{3}\right)^{-2}$, so there exists some constant $C>0$, such that
	\begin{equation*}
		(2x+\bar{U}_\theta) \left( \ln \frac{1+x}{3} \right)\left( \ln \frac{1-x}{3} \right) \leq C, \quad -1<x<1. 
	\end{equation*}
	We also see from the above that $\displaystyle{\lim_{x\to \pm 1}A_{\theta}(x)=0}$. 
	
	For $-\frac{1}{2}<x<\frac{1}{2}$, 
	\[
	  \begin{split}
	    |A'_{\theta}|& =|l'(x)-l''(0)x|\\
	       & \le |\tilde{U}''_{\theta}|+|\bar{U}_{\theta}||\tilde{U}'_{\theta}|+(2+|\bar{U}'_{\theta}|)|\tilde{U}_{\theta}|+|l''(0)|\\
	       & \le C\|\tilde{U}_{\theta}\|_{\mathbf{M}_1},
	    \end{split}
	\]
	and
	\[
	\begin{split}
	   |A''_{\theta}|& =|l''(x)-l''(0)|\\
	   & \le |\tilde{U}'''_{\theta}|+(|\bar{U}_{\theta}|+2)|\tilde{U}''_{\theta}|+2(|\bar{U}'_{\theta}|+1)|\tilde{U}'_{\theta}|+|\bar{U}''_{\theta}||\tilde{U}_{\theta}|+|l''(0)|\\
	   & \le C\|\tilde{U}_{\theta}\|_{\mathbf{M}_1}.
	   \end{split} 
	\]
	
	By computation $A''_{\theta}(0)=0$. So we have $A_{\theta}\in \mathbf{N}_1$ and $\|A_{\theta}\|_{\mathbf{N}_1}\le C\|\tilde{U}_{\theta}\|_{\mathbf{M}_1}$.
	
	Next, since $A_{\phi}=(1-x^2)\tilde{U}''_{\phi}+\bar{U}_{\theta}\tilde{U}'_{\phi}$, by the fact that $\tilde{U}_{\phi}\in \mathbf{M}_2$  we have that
	\[
		\left|(1-x^2)^{1+\epsilon}A_{\phi}\right| \le (1-x^2)^{2+\epsilon}|\tilde{U}''_{\phi}|+(1-x^2)^{1+\epsilon}|\bar{U}_{\theta}\|\tilde{U}'_{\phi}|\le C\|\tilde{U}_{\phi}\|_{\mathbf{M}_2}.
	\]
	So $A_{\phi}\in \mathbf{N}_1$, and $\|A_{\phi}\|_{\mathbf{N}_1}\le C\|\tilde{U}_{\phi}\|_{\mathbf{M}_2}$. We have proved that $A\tilde{U}\in \mathbf{Y}$ and $\|A\tilde{U}\|_{\mathbf{Y}}\le C\|\tilde{U}\|_{\mathbf{X}}$ for every $\tilde{U}\in \mathbf{X}$. The proof is finished.
	%
	%
\end{proof}

\begin{lem}\label{sec34:lem:Q:well-def}
	The map $Q:\mathbf{X}\times\mathbf{X}\to \mathbf{Y}$ defined by (\ref{sec31:eq:Q}) is a well-defined bounded bilinear operator.
\end{lem}
\begin{proof}
	In the following, $C$ denotes a universal constant which may change from line to line. 
	For convenience we denote $\psi=\psi[\tilde{U}_{\phi}, \tilde{V}_{\phi}]$ defined by (\ref{sec31:eq:psi}).  

	It is clear that $Q$ is a bilinear operator. For every $\tilde{U},\tilde{V}\in\mathbf{X}$, we will prove that $Q(\tilde{U},\tilde{V})$ is in $\mathbf{Y}$ and there exists some constant $C$ independent of $\tilde{U}$ and $\tilde{V}$ such that $\|Q(\tilde{U},\tilde{V})\|_{\mathbf{Y}}\le C\|\tilde{U}\|_{\mathbf{X}}\|\tilde{V}\|_{\mathbf{X}}$.
    
For $\tilde{U},\tilde{V}\in \mathbf{X}$, we have, using the fact that $\tilde{U}_{\phi}, \tilde{V}_{\phi}\in \mathbf{M}_2$, that
	\begin{equation}\label{eq3_18_0}
		\left|\frac{ \tilde{U}_\phi(s)\tilde{V} _\phi'(s)}{1-s^2}\right| \le (1-s^2)^{-2-2\epsilon}\|\tilde{U}_{\phi}\|_{\mathbf{M}_2}\|\tilde{V}_{\phi}\|_{\mathbf{M}_2}, \quad \forall -1<s<1.
	\end{equation}
	It follows that $\psi(\tilde{U},\tilde{V})(x)$ is well-defined and
	\begin{equation*}
	    |\psi(-1)|\le C\|\tilde{U}_{\phi}\|_{\mathbf{M}_2}\|\tilde{V}_{\phi}\|_{\mathbf{M}_2}, \quad |\psi(1)|\le C\|\tilde{U}_{\phi}\|_{\mathbf{M}_2}\|\tilde{V}_{\phi}\|_{\mathbf{M}_2}.  
	\end{equation*}
	Moreover,
	\begin{equation}\label{eq3_18_3}
		\begin{split}
		& \left|\psi(x)-\frac{1}{2}\psi(-1)(1-x)-\frac{1}{2}\psi(1)(1+x)\right|\\
		& =\left|\frac{1}{2}\psi(x)(1-x)+\frac{1}{2}\psi(x)(1+x)-\frac{1}{2}\psi(-1)(1-x)-\frac{1}{2}\psi(1)(1+x)\right|\\
		& \le \frac{1}{2}(1-x)|\psi(x)-\psi(-1)|+\frac{1}{2}(1+x)|\psi(x)-\psi(1)|\\
		& =\frac{1}{2}(1-x)\left|\int_{-1}^{x} \int_{0}^{l} \int_{0}^{t} \frac{2 \tilde{U}_\phi(s)\tilde{V} _\phi'(s)}{1-s^2} ds dt dl\right|+\frac{1}{2}(1+x)\left|\int_{1}^{x} \int_{0}^{l} \int_{0}^{t} \frac{2 \tilde{U}_\phi(s)\tilde{V} _\phi'(s)}{1-s^2} ds dt dl\right|\\
		& \le C(1-x)(1+x)^{1-2\epsilon}\|\tilde{U}_\phi\|_{\mathbf{M}_2}\|\tilde{V} _{\phi}\|_{\mathbf{M}_2}+C(1+x)(1-x)^{1-2\epsilon}\|\tilde{U}_\phi\|_{\mathbf{M}_2}\|\tilde{V} _{\phi}\|_{\mathbf{M}_2}\\
		& \le C(1-x^2)^{1-2\epsilon}\|\tilde{U}_\phi\|_{\mathbf{M}_2}\|\tilde{V} _{\phi}\|_{\mathbf{M}_2}.
		\end{split}
	\end{equation}
	By (\ref{eq3_18_0}), we also have
	\begin{equation*}
	   |\psi'(x)|=\left|\int_{0}^{x} \int_{0}^{t} \frac{2 \tilde{U}_\phi(s)\tilde{V} _\phi'(s)}{1-s^2} ds dt\right|\le C \|\tilde{U}_\phi\|_{\mathbf{M}_2}\|\tilde{V} _{\phi}\|_{\mathbf{M}_2}, \quad -\frac{1}{2}<x<\frac{1}{2},
	\end{equation*}
	and 
	\begin{equation*}
	   |\psi''(x)|=\left|\int_{0}^{x} \frac{2 \tilde{U}_\phi(s)\tilde{V} _\phi'(s)}{1-s^2} ds\right|\le C \|\tilde{U}_\phi\|_{\mathbf{M}_2}\|\tilde{V} _{\phi}\|_{\mathbf{M}_2}, \quad -\frac{1}{2}<x<\frac{1}{2}. 
	\end{equation*}
	Using the fact that $\tilde{U}_{\theta},\tilde{V}_{\theta}\in \mathbf{M}_1$, we have
	\begin{equation}\label{eq3_18_6}
		\begin{split}
		|(\tilde{U}_{\theta}\tilde{V}_{\theta})''(0)| & \le |\tilde{U}''_{\theta}(0)\|\tilde{V}_{\theta}(0)|+2|\tilde{U}'_{\theta}(0)\|\tilde{V}'_{\theta}(0)|+|\tilde{U}_{\theta}(0)\|\tilde{V}''_{\theta}(0)|\\
		& \le C\|\tilde{U}_{\theta}\|_{\mathbf{M}_1}\|\tilde{V}_{\theta}\|_{\mathbf{M}_1}. 
	\end{split}
	\end{equation}
	So by (\ref{eq3_18_3}),  (\ref{eq3_18_6}), and the fact that $\tilde{U}_{\theta}, \tilde{V}_{\theta}\in \mathbf{M}_1$, we have that for $-1<x<1$,
	\[
	\begin{split}
		& |\left(\ln \frac{1+x}{3}\right)^2 \left(\ln \frac{1-x}{3}\right)^2 Q_{\theta}(x)|\\
		& \le \frac{1}{2} \left| \left(\ln \frac{1+x}{3}\right) \left(\ln \frac{1-x}{3}\right) \tilde{U}_{\theta}(x)\right| \cdot \left|\left(\ln \frac{1+x}{3}\right) \left(\ln \frac{1-x}{3}\right) \tilde{V}_{\theta}(x)\right| \\
		& \quad + \left|\left( \ln \frac{1+x}{3}\right)^2 \cdot\left(\ln \frac{1-x}{3} \right)^2\right| \left|\psi(x)-\frac{1}{2}\psi(-1)(1-x)-\frac{1}{2}\psi(1)(1+x)\right|\\
		& \quad +\frac{1}{4} \left|\left( \ln \frac{1+x}{3}\right)^2 \cdot\left(\ln \frac{1-x}{3} \right)^2 \right| (1-x^2) |(\tilde{U}_{\theta}\tilde{V}_{\theta})''(0)|\\
		& \le C\|\tilde{U}\|_{\mathbf{X}}\|\tilde{V}\|_{\mathbf{X}}. 
	\end{split}
	\]
	From this we also have $\displaystyle \lim_{x\to 1}Q_{\theta}(x)= \lim_{x\to -1}Q_{\theta}(x)=0$. 
	Similar as in Lemma \ref{sec32:lem:Q:well-def}, we have that for $-\frac{1}{2}<x<\frac{1}{2}$,
	\[
	    |Q'_{\theta}(x)| 
 \le C\|\tilde{U}\|_{\mathbf{X}}\|\tilde{V}\|_{\mathbf{X}},\quad |Q''_{\theta}(x)|\le C\|\tilde{U}\|_{\mathbf{X}}\|\tilde{V}\|_{\mathbf{X}}.
	\]
	
	So there is $Q_{\theta}\in\mathbf{N}_1$, and $\|Q_{\theta}\|_{\mathbf{N}_1}\le C(\epsilon)\|\tilde{U}\|_{\mathbf{X}}\|\tilde{V}\|_{\mathbf{X}}$.
	
	Next, since $Q_{\phi}(x)=\tilde{U}_{\theta}(x)\tilde{V}'_{\phi}(x)$, for $-1<x<1$, 
	\[
		\left|(1-x^2)^{1+\epsilon}Q_{\phi}(x)\right|  \le |\tilde{U}_{\theta}(x)|(1-x^2)^{1+\epsilon}|\tilde{V}'_{\phi}|
		\le 2\|\tilde{U}_{\theta}\|_{\mathbf{M}_1}\|\tilde{V}_{\phi}\|_{\mathbf{M}_2}.
	\]
	So $Q_{\phi}\in \mathbf{N}_2$, and 
	$
		\|Q_{\phi}\|_{\mathbf{N}_2}\le \|\tilde{U}_{\theta}\|_{\mathbf{M}_1}\|\tilde{V}_{\phi}\|_{\mathbf{M}_2}.
	$
	Thus we have proved that $Q(\tilde{U}, \tilde{V})\in \mathbf{Y}$ and $\|Q(\tilde{U},\tilde{V})\|_{\mathbf{Y}}\le C\|\tilde{U}\|_{\mathbf{X}}\|\tilde{V}\|_{\mathbf{X}}$ for all $\tilde{U}, \tilde{V}\in \mathbf{X}$.  Lemma \ref{sec34:lem:Q:well-def} is proved.
\end{proof}
\\

\noindent 
{\bf Proof of Proposition \ref{sec34:prop}.}
By definition, $G(c,\gamma,\tilde{U})=A(c,\gamma,\tilde{U})+Q(\tilde{U},\tilde{U})$ for $(c,\gamma,\tilde{U})\in K \times \mathbf{X}$.  
Using standard theories in functional analysis, by Lemma \ref{sec34:lem:Q:well-def} it is clear that $Q$ is $C^{\infty}$ on $\mathbf{X}$.  
By Lemma \ref{sec34:lem:A:well-def}, $A(c,\gamma; \cdot): \mathbf{X}\to \mathbf{Y}$ is $C^{\infty}$ for each $(c,\gamma)\in K$.

Let $\alpha=(\alpha_1,\alpha_2,\alpha_3)$ be a multi-index where $\alpha_i\ge 0$, $i=1,2,3$, and $j\ge 0$. For all $|\alpha|+j\ge 1$,  we have
\begin{equation}\label{eq_prop3_3_1}
	\pt_{c}^\alpha \pt_{\gamma}^j A(c,\gamma,\tilde{U}) = \pt_{c}^\alpha \pt_{\gamma}^j U^{c,\gamma}_{\theta}\left(
	\begin{matrix}
		\tilde{U}_{\theta}  \\  \tilde{U}'_{\phi} 
	\end{matrix}\right) + \frac{1}{2} (\pt_{c}^\alpha \pt_{\gamma}^j U^{c,\gamma}_{\theta} \cdot \tilde{U}_\theta)'' (0) 
	\begin{pmatrix}
	1-x^2   \\   0  
	\end{pmatrix}.  
\end{equation}

By Proposition \ref{propA_1} (4), we have
%
\[
	\left|\ln\frac{1+x}{3}\right|^2 \left|\ln\frac{1-x}{3}\right|^2 \left| \pt_{c}^{\alpha} \pt_{\gamma}^j A_{\theta}(c,\gamma,\tilde{U}) \right| \leq C(\alpha,j,K) \| \tilde{U}_{\theta} \|_{\mathbf{M}_1}, \quad -1<x<1,
\]
and for $-\frac{1}{2}<x<\frac{1}{2}.$
\[
     | \pt_{c}^\alpha \pt_{\gamma}^j A'_{\theta}(c,\gamma,\tilde{U}) | \leq C(\alpha,j,K) \| \tilde{U}_{\theta} \|_{\mathbf{M}_1}, \quad | \pt_{c}^\alpha \pt_{\gamma}^j A''_{\theta}(c,\gamma,\tilde{U}) | \leq C(\alpha,j,K) \| \tilde{U}_{\theta} \|_{\mathbf{M}_1}.
\]
The above estimates and (\ref{eq_prop3_3_1}) also imply that 
$$
	\pt_{c}^{\alpha} \pt_{\gamma}^j A_{\theta}(c,\gamma,\tilde{U}) (-1) 
	= \pt_{c}^{\alpha} \pt_{\gamma}^j A_{\theta}(c,\gamma,\tilde{U})(1) 
	= \pt_{c}^{\alpha} \pt_{\gamma}^j  A_{\theta}(c,\gamma,\tilde{U})'' (0) = 0. 
$$
So $\pt_{c}^{\alpha} \pt_{\gamma}^j A_{\theta}(c,\gamma,\tilde{U}) \in \mathbf{N}_1$, with $\|\pt_{c}^{\alpha} \pt_{\gamma}^jA_{\theta}(c,\gamma,\tilde{U}) \|_{\mathbf{N}_1} \leq C(\alpha,j,K) \|\tilde{U}_{\theta} \|_{\mathbf{M}_1}$ for all $(c,\gamma,\tilde{U}) \in K\times \mathbf{X}$. 

Next, by Proposition \ref{propA_1} and the fact that $\tilde{U}_{\phi}\in \mathbf{M}_1$, we have 
\begin{equation}\label{eq_prop3_4_2}
	(1-x^2)^{1+\epsilon} | \pt_{c}^\alpha \pt_{\gamma}^j A_{\phi}(c,\gamma,\tilde{U})(x)| 
	= |\pt_{c}^\alpha \pt_{\gamma}^j  U^{c,\gamma}_{\theta}(x)| \cdot |(1-x^2)^{1+\epsilon}\tilde{U}'_{\phi}|
	\leq C(\alpha,j,K) \| \tilde{U}_{\phi} \|_{\mathbf{M}_2}. 
\end{equation}
So $\pt_{c}^\alpha \pt_{\gamma}^j A_{\phi}(c,\gamma,\tilde{U})\in \mathbf{N}_2$ with $ \|\pt_{c}^\alpha \pt_{\gamma}^j A_{\phi}(c,\gamma,\tilde{U}) \|_{\mathbf{N}_2} \leq C(\alpha,j,K) \| \tilde{U}_{\phi} \|_{\mathbf{M}_2}$ for all $(c,\gamma,\tilde{U})\in K\times \mathbf{X}$. Thus $\pt_{c}^\alpha \pt_{\gamma}^j A(c,\gamma,\tilde{U}) \in \mathbf{Y}$, with $\|\pt_{c}^\alpha \pt_{\gamma}^j A(c,\gamma,\tilde{U}) \|_{\mathbf{Y}} \leq C(\alpha,j,K) \| \tilde{U} \|_{\mathbf{X}}$ for all $(c,\gamma,\tilde{U})\in K\times \mathbf{X}$, $|\alpha|+j \ge 1$. 

So for each $(c,\gamma)\in K$, $\pt_{c}^\alpha \pt_{\gamma}^j A(c,\gamma; \cdot): \mathbf{X}\to \mathbf{Y}$ is a bounded linear map with uniform bounded norm on $K$. Then by standard theories in functional analysis, $A: K\times\mathbf{X}\to \mathbf{Y}$ is $C^{\infty}$. So $G$ is a $C^{\infty}$ map from $K  \times \mathbf{X}$ to $\mathbf{Y}$. By direct calculation we get its Fr\'{e}chet derivative with respect to $\mathbf{X}$ is given by  the linear bounded operator $L^{c,\gamma}_{\tilde{U}}: \mathbf{X}\rightarrow \mathbf{Y}$ defined as  (\ref{sec31:eq:Linear}). The proof is finished.     \qed 
%
%
\\

Let $a_{c,\gamma}(x), b_{c,\gamma}(x)$ be the functions defined by (\ref{sec31:eq:ab}). For convenience we denote $a(x)=a_{c,\gamma}(x)$, $b(x)=b_{c,\gamma}(x)$, and $\bar{U}_{\theta}=U^{c,\gamma}_{\theta}$. 
\begin{lem}\label{lem3_4_ab}
  For $(c,\gamma)\in I_{4,1}$, there exists some constant $C>0$, depending only on $(c,\gamma)$, such that  for any $-1<x<1$, 
 \begin{equation}\label{sec34:eq:eb}
	\begin{split}
		& e^{b(x)}\le C \left(\ln\frac{1+x}{3}\right)^2  \left(\ln\frac{1-x}{3}\right)^2 (1-x^2),  \\
		& e^{-b(x)}\le C \left(\ln\frac{1+x}{3}\right)^{-2}\left(\ln\frac{1-x}{3}\right)^{-2} (1-x^2)^{-1},
	\end{split}
	\end{equation}
	and 
	\begin{equation}\label{sec34:eq:ea}
		e^{a(x)}\le C \left| \ln \frac{1+x}{3} \right|^2 \left| \ln \frac{1-x}{3} \right|^2, \quad 
		e^{-a(x)}\le C \left(\ln\frac{1+x}{3}\right)^{-2} \left(\ln\frac{1-x}{3}\right)^{-2}. 
	\end{equation}
\end{lem}
\begin{proof}
   Under the assumption of $\bar{U}_{\theta}$ in this case, we have for some small $\epsilon>0$
	\[
		\bar{U}_{\theta}=2+\frac{4}{\ln\frac{1+x}{3}}+O(1)\left(\ln \frac{1+x}{3}\right)^{-2}, \quad -1<x<0,
	\]
	\[
		\bar{U}_{\theta}=-2-\frac{4}{\ln\frac{1-x}{3}}+O(1)\left(\ln \frac{1-x}{3}\right)^{-2}, \quad 0<x<1.
	\]
	Thus, by definition of $a(x)$ and $b(x)$ in (\ref{sec31:eq:ab}),  for $-1<x<1$, we have
	\[
	\begin{split}
		& b(x)=\ln(1+x) + 2\ln(-\ln(1+x)) + \ln(1-x) + 2\ln(-\ln(1-x)) +O(1),\\
		& a(x)=2\ln(-\ln(1+x))+2\ln(-\ln(1-x))+O(1).
	\end{split}
	\]
	The lemma follows from the above estimates. 
\end{proof}

For $\xi = (\xi_\theta, \xi_\phi) \in \mathbf{Y}$, let the map $W^{c,\gamma}$ be defined as 
$$
	W^{c,\gamma}(\xi) := (W^{c,\gamma,1}_\theta(\xi), W^{c,\gamma}_\phi(\xi)), 
$$
where 
$W^{c,\gamma,1}_{\theta}$ and $W^{c,\gamma}_\phi(\xi)$ are defined by (\ref{sec31:eq:Wthei}) and (\ref{sec31:eq:Wphi}).

\begin{lem}\label{sec34:lem:W}
	For every $(c,\gamma)\in K$, $W^{c,\gamma}: \mathbf{Y}\rightarrow\mathbf{X}$ is continuous, and is a right inverse of $L^{c,\gamma}_{0}$. 
\end{lem}
\begin{proof}
	In the following, $C$ denotes a universal constant which may change from line to line. We make use of the property that $\bar{U}_{\theta}\in C^2(-1,1)\cap L^{\infty}(-1,1)$ and the fact that $0<\epsilon<1/2$. 
	For convenience let us write  $W:=W^{c,\gamma}(\xi)$ and $W_{\theta}:=W_{\theta}^{c,\gamma,1}(\xi)$ for $\xi\in \mathbf{Y}$. 
	


	We first prove $W_{\theta}: \mathbf{Y}\to\mathbf{X}$ is well-defined. 
	Using Lemma \ref{lem3_4_ab} and the fact that $\xi_{\theta}\in \mathbf{N}_1$, we have, by the expression of $W_{\theta}=W^1_{\theta}$ in (\ref{sec31:eq:Wthei}) and (\ref{sec34:eq:ea}), for any $-1<x<1$ that 
	\begin{equation}\label{sec34:eq:Wthe:bdd}
	\begin{split}
		& \quad \left|  \ln \frac{1+x}{3}  \ln \frac{1-x}{3} W^1_{\theta}(x)\right| \\
		& \le \left| \ln \frac{1+x}{3} \ln \frac{1-x}{3} \right| \|\xi_{\theta}\|_{\mathbf{N}_1}e^{-a(x)} 
		\int_{0}^{x}e^{a(s)} \left| \ln \frac{1+s}{3} \ln \frac{1+s}{3} \right|^{-2} ds  
		\le C \|\xi_{\theta}\|_{\mathbf{N}_1}. 
	\end{split}
	\end{equation}
	From the above we also have that $\lim_{x\to \pm 1}W_{\theta}(x)=0$.
	
	By (\ref{sec31:eq:diff:Wthe}), (\ref{sec31:eq:diff:a}),  (\ref{sec34:eq:Wthe:bdd}), and the property that $\bar{U}_{\theta}=2+O(1)\left(\ln \frac{1+x}{3}\right)^{-1}=-2+O(1)\left(\ln \frac{1-x}{3}\right)^{-1}$, we have that for $-1<x<1$ 
	\begin{equation}\label{sec34:eq:Wthe:bdd2}
	\begin{split}
		& \quad \left| (1-x^2) \left( \ln \frac{1+x}{3} \right)^2 \left( \ln \frac{1-x}{3} \right)^2 W'_{\theta} \right| \\ 
		& \le \left| (2x+\bar{U}_{\theta}) \left( \ln \frac{1+x}{3} \right)^2 \left( \ln \frac{1-x}{3} \right)^2 W_{\theta} \right|
		+ \left( \ln \frac{1+x}{3} \right)^2 \left( \ln \frac{1-x}{3} \right)^2 |\xi_{\theta}(x)|  \\
		& \le  C \|\xi_{\theta}\|_{\mathbf{N}_1}. 
	\end{split}
	\end{equation}
	By (\ref{sec31:eq:diff:a}), it can be seen that $|a''(x)|,|a'''(x)|\le C$ for $-\frac{1}{2}<x<\frac{1}{2}$. Then using this fact, \eqref{sec34:eq:Wthe:bdd} and \eqref{sec34:eq:Wthe:bdd2}, we have, for $-\frac{1}{2}<x<\frac{1}{2}$,
	\[
	    |W_{\theta}''(x)|=\left|a''(x)W_{\theta}(x)+a'(x)W'_{\theta}(x)+\left(\frac{\xi_{\theta}}{1-x^2}\right)'\right|\le C \|\xi_{\theta}\|_{\mathbf{N}_1},
	\]
	and 
	\[
	     |W_{\theta}'''(x)|=\left|a'''(x)W_{\theta}(x)+2a''(x)W'_{\theta}(x)+a'(x)W''_{\theta}(x)+\left(\frac{\xi_{\theta}}{1-x^2}\right)''\right|\le C \|\xi_{\theta}\|_{\mathbf{N}_1}.
	\]

	So we have shown that $W_{\theta}\in \mathbf{M}_1$, and $\|W_{\theta}\|_{\mathbf{M}_1}\le C\|\xi_{\theta}\|_{\mathbf{N}_1}$ for some constant $C$. 
	By the definition of $W_{\phi}(\xi)$ in (\ref{sec31:eq:Wphi}), (\ref{sec31:eq:diff1:Wphi}), (\ref{sec34:eq:eb}) and the fact that $\xi_{\phi}\in \mathbf{N}_2$, we have, for every $-1<x<1$, 
	\[
	\begin{split}
		& \quad (1-x^2)^\epsilon |W_{\phi}(x)|  \le  \int_{0}^{x}e^{-b(t)}\int_{0}^{t}e^{b(s)}\frac{|\xi_{\phi}(s)|}{1-s^2}dsdt\\
		& \le  C\|\xi_{\phi}\|_{\mathbf{N}_2} (1-x^2)^\epsilon \int_{0}^{x}  \left(\ln\frac{1+t}{3}\right)^{-2}\left(\ln\frac{1-t}{3}\right)^{-2} (1-t^2)^{-1} \\
		& \qquad \cdot \int_{0}^{t} \left(\ln\frac{1+s}{3}\right)^2  \left(\ln\frac{1-s}{3}\right)^2 (1-s^2)^{-1-\epsilon}dsdt\\
		& \le C\|\xi_{\phi}\|_{\mathbf{N}_2},
		\end{split}
	\]
	and 
	\begin{equation}\label{sec34:eq:Wphi:bdd:temp}
		|(1-x^2)^{1+\epsilon}W'_{\phi}(x)| \le(1-x^2)^{1+\epsilon} e^{-b(x)}\int_{0}^{x}e^{b(s)}\frac{|\xi_{\phi}(s)|}{1-s^2}ds
		 \le C\|\xi_{\phi}\|_{\mathbf{N}_2}.
	\end{equation}
	Similarly, since $|b'(x)|=\frac{|\bar{U}_{\theta}|}{1-x^2}$, using (\ref{sec31:eq:diff2:Wphi}),  (\ref{sec34:eq:Wphi:bdd:temp}) and the fact that $\xi_{\phi}\in \mathbf{N}_2$, we have
	\[
		|(1-x^2)^{2+\epsilon}W''_{\phi}(x)|
		\le C\|\xi_{\phi}\|_{\mathbf{N}_2}.
	\]
	Therefore $W(\xi)\in\mathbf{X}$ for all $\xi\in\mathbf{Y}$, and $\|W(\xi)\|_{\mathbf{X}}\le C\|\xi\|_{\mathbf{Y}}$ for some constant $C$. So $W:\mathbf{Y}\rightarrow\mathbf{X}$ is well-defined and continuous. 
	 
	By definition of $W$,  we have $l_{c,\gamma}[W_{\theta}](x)=\xi_{\theta}$. So $(l_{c,\gamma}[W_{\theta}])''(0)=\xi''_{\theta}(0)=0$,  $l_{c,\gamma}[W_{\theta}](x)+\frac{1}{2}(l_{c,\gamma}[W_{\theta}])''(0)(1-x^2)=\xi_{\theta}$.  Thus $L_0^{c,\gamma}W(\xi)=\xi$, $W$ is a right inverse of $L_0^{c,\gamma}$.
	 
\end{proof}

Let $V_{c,\gamma}^i$, $1\le i\le 4$, be vectors defined by (\ref{eq_basis}),  we have
\begin{lem}\label{sec34:lem:ker}
	\[
		\ker L^{c,\gamma}_{0}=
		\mathrm{span}\{V_{c,\gamma}^1, V_{c,\gamma}^2, V_{c,\gamma}^3,V^4_{c,\gamma}\}. 
	 \]
\end{lem}
\begin{proof}
	Let $V\in\mathbf{X}$, $L^{c,\gamma}_{0}V=0$. We know that $V$ is given by (\ref{sec31:eq:ker}) for some $d_1,d_2,d_3,d_4\in\mathbb{R}$.
	
	By Lemma \ref{lem3_4_ab}, and the expressions of $V^1,V^2$ in (\ref{eq_basis}), we have that	
		
	\begin{equation}\label{eq3_4_ker_1}
		V^1_{\theta}(x) = e^{-a(x)}=  O(1) \left| \ln \frac{1+x}{3} \right|^{-2} \left| \ln \frac{1-x}{3} \right|^{-2},
	\end{equation}
	and 
	\begin{equation}\label{eq3_4_ker_2}
		 V^2_{\theta}(x) =e^{-a(x)}\int_{0}^{x}e^{a(s)}ds		
		=  O(1) \left| \ln \frac{1+x}{3} \right|^{-2} \left| \ln \frac{1+x}{3} \right|^{-2}. 
	\end{equation}
	
By (\ref{sec31:eq:diff:a}), we also have
	\begin{equation*}
		 \left|\frac{d}{dx}V^1_{\theta}(x)\right|=\left|e^{-a(x)}a'(x)\right| = O(1) \left| \ln \frac{1+x}{3} \right|^{-2}\left| \ln \frac{1-x}{3} \right|^{-2} (1-x^2)^{-1},
	\end{equation*}
	\begin{equation*}
		 \left|\frac{d}{dx}V^2_{\theta}(x)\right|=\left| -V^2_{\theta}(x)a'(x) + 1\right| = O(1)  \left| \ln \frac{1+x}{3}\right|^{-2} \left| \ln \frac{1-x}{3} \right|^{-2} (1-x^2)^{-1}.
	\end{equation*}
	
	Next, by computation we have for $i=1,2,$
	\[
	 \frac{d^2}{dx^2}V^i_{\theta}=(V^i_{\theta})'a'(x)+V^i_{\theta}a''(x),\quad  \frac{d^3}{dx^3}V^i_{\theta}=(V^i_{\theta})''a'(x)+2(V^i_{\theta})'a''(x)+V^i_{\theta}a'''(x).
	\]
	Using the definition of $a(x)$ in (\ref{sec31:eq:ab}), there exists some constant $C$, depending on $c,\gamma$, such that
	\begin{equation*}
	   \left |\frac{d^2}{dx^2}V^i_{\theta}\right|\le C, \quad \left|\frac{d^3}{dx^3}V^i_{\theta}\right|\le C, \quad -\frac{1}{2}<x<\frac{1}{2}, \; i=1,2.
	\end{equation*}
	
	Moreover, by Lemma \ref{lem3_4_ab}, and the expressions of $V^3$ in (\ref{eq_basis}), we have
	\begin{equation}\label{eq3_4_ker_9}
		V^3_{\phi}(x)= \int_0^x \left| \ln \frac{1+t}{3} 
		\ln \frac{1-t}{3} \right|^{-2} (1-t^2)^{-1} \cdot \int_0^t \left| \ln \frac{1+s}{3} \ln \frac{1-s}{3} \right|^2 (1-s^2) ds dt=O(1),
	\end{equation}
	and 
	\begin{equation}\label{eq3_4_ker_10}
	\begin{split}
		&\left|\frac{d}{dx}V^3_{\phi}(x)\right|=e^{-b(x)}
		= O(1) \left| \ln \frac{1+x}{3} \ln \frac{1-x}{3} \right|^{-2} (1-x^2)^{-1}, \\
					&\left|\frac{d^2}{dx^2}V^3_{\phi}(x)\right|=e^{-b(x)}|b'(x)| = O(1) \left| \ln \frac{1+x}{3} \ln \frac{1-x}{3} \right|^{-2} (1-x^2)^{-2}. 
			\end{split}
	\end{equation}


	Using the above estimates and the definition of $V_{c,\gamma}^4$, it is not hard to verify that $V_{c,\gamma}^i\in \mathbf{X}$, $1\le i\le 4$.  It is clear that $\{V_{c,\gamma}^i , 1\le i\le 4\}$ are independent. So $\{V_{c,\gamma}^i , 1\le i\le 4\}$ is a basis of the kernel.
\end{proof}

\begin{cor}\label{sec34:cor:all:sol}
	For any $\xi\in\mathbf{Y}$, all solutions of $L^{c,\gamma}_{0}V=\xi$, $V\in\mathbf{X}$, are given by
	\[
		V=W^{c,\gamma}(\xi)+d_1V_{c,\gamma}^1+d_2V_{c,\gamma}^2+d_3V_{c,\gamma}^3+d_4V^4_{c,\gamma}. 
	\]
\end{cor}

Let $l_i$, $1\le i\le 4$, be the functionals on $\mathbf{X}$ defined by (\ref{sec31:eq:fcnal:l}), and $\mathbf{X}_1=\cap_{i=1}^4 \ker l_i$ . As shown in Section \ref{sec_2}, the matrix $(l_i(V^{j}_{c,\gamma}))$ is invertible, for every $(c,\gamma)\in K$. So $\mathbf{X}_i$ is a closed subspace of $\mathbf{X}$, and
\begin{equation}\label{sec34:eq:decomp:X}
	\mathbf{X} =
		\mbox{span} \{ V_{c,\gamma}^{1}, V_{c,\gamma}^{2}, V_{c,\gamma}^{3}, V^4_{c,\gamma} \} \oplus\mathbf{X}_1, 
\end{equation}
with the projection operator $P_1: \mathbf{X}\rightarrow\mathbf{X}_1$ given by (\ref{sec32:eq:proj}). 
\begin{lem}\label{sec34:lem:iso}
	The operator $ L^{c,\gamma}_{0}: \mathbf{X}_1\rightarrow\mathbf{Y}$ is an isomorphism.  
\end{lem}

\begin{proof}
	By Corollary \ref{sec34:cor:all:sol} and Lemma \ref{sec34:lem:ker}, $L^{c,\gamma}_{0}:\mathbf{X}\rightarrow\mathbf{Y}$ is surjective and $\ker L^{c,\gamma}_{0}$  is given by Lemma \ref{sec34:lem:ker}. The conclusion of the lemma then follows in view of the direct sum property (\ref{sec34:eq:decomp:X}).	
\end{proof}


\begin{lem}\label{sec34:lem:V:smooth}
	$V_{c,\gamma}^i\in C^{\infty}(K,\mathbf{X})$ for all $1\le i\le 4$ and $(c,\gamma)$ in compact subset $K$ of $I_{4,1}$. 
\end{lem}
\begin{proof}
It is clear that $V^4_{c,\gamma}\in  C^{\infty}(K,\mathbf{X})$ for all compact set $K$ in $I_{4,1}$.
    
    Let $\alpha=(\alpha_1,\alpha_2,\alpha_3)$ be a multi-index where $\alpha_i\ge 0$, $i=1,2,3$, and $j\ge 0$. 
    For convenience we denote $a(x)=a_{c,\gamma}(x)$, $b(x)=b_{c,\gamma}(x)$ and $V^i=V^i_{c,\gamma}$, $i=1,2,3$.

    Using Proposition \ref{propA_1} part (4), we have that for all $|\alpha|+j\ge 1$ and $(c,\gamma)\in K$,
		 \begin{equation}\label{sec34:eq:Vsmooth:temp}
		\pt_c^j \pt_\gamma^i a(x) = \pt_c^j \pt_\gamma^i b(x) = \int_{0}^{x}\frac{1}{1-s^2} \pt_c^j \pt_\gamma^i U^{c,\gamma}(s)ds = O(1).
	\end{equation}
	
	Using the expression of $V^i$, $1\le i\le 4$ in (\ref{eq_basis}), Lemma \ref{lem3_4_ab}, (\ref{eq3_4_ker_1}), (\ref{eq3_4_ker_2}), (\ref{sec34:eq:Vsmooth:temp}) and Proposition \ref{propA_1} (4), we have that for all $|\alpha|+j\ge 1$ and $(c,\gamma)\in K$,	

	\[
	\begin{split}
		&\left| \pt_{c}^{\alpha} \pt_\gamma^j V^1_{\theta}(x) \right| = e^{-a(x)}O(1)
		=O(1) \left| \ln \frac{1+x}{3} \ln \frac{1-x}{3} \right|^{-2},\\
		&  \left| \pt_c^\alpha \pt_\gamma^j V^2_{\theta}(x) \right| = e^{-a(x)}\left|\int_{0}^{x}e^{a(s)}ds\right|O(1)=O(1)  \left| \ln \frac{1+x}{3} \ln \frac{1-x}{3} \right|^{-2}, 
		\end{split}
	\]
	and 
	\[
	\begin{split}
	  &  \left| \frac{d}{dx}\pt_c^\alpha \pt_\gamma^j V^1_{\theta}(x) \right| = e^{-a(x)}|a'(x)|O(1)
		=O(1) \left( \ln \frac{1+x}{3} \right)^{-2} \left( \ln \frac{1-x}{3} \right)^{-2}(1-x^2)^{-1},\\
		& \left| \frac{d}{dx}\pt_c^\alpha \pt_\gamma^j V^2_{\theta}(x) \right| = |-V^2_{\theta}a'(x)+1|O(1)
		=O(1) \left( \ln \frac{1+x}{3} \right)^{-2} \left( \ln \frac{1-x}{3} \right)^{-2}(1-x^2)^{-1}.
		\end{split}
	\]
	From the above we can see that for all $|\alpha|+j\ge 1$, there exists some constant $C=C(\alpha,j,K)$, such that for $i=1,2$,
	\begin{eqnarray*}
		\left| \ln \frac{1+x}{3} \ln \frac{1-x}{3}\pt_{c}^{\alpha} \pt_\gamma^j V^i_{\theta}(x)\right|\le C,  \\
		\left| \ln \frac{1+x}{3} \ln \frac{1-x}{3} \right|^2 (1-x^2) \left| \frac{d}{dx} \pt_{c}^{\alpha} \pt_\gamma^jV^i_{\theta}(x)\right|\le C. 
	\end{eqnarray*}
	We also have that for $|\alpha|+j\ge 1$, $\pt_c^\alpha \pt_\gamma^j  V^i_{\theta}(1) = \pt_c^\alpha \pt_\gamma^j  V^i_{\theta}(-1) = 0$, $i=1,2$. 
	
	Next, using the definition of $a(x)$ in (\ref{sec31:eq:ab}), there exists some constant $C=C(K)$, such that  
	\[
	    \left| \frac{d^2}{dx^2}\pt_c^\alpha \pt_\gamma^jV^i_{\theta}\right|\le C, \quad \left| \frac{d^3}{dx^3}\pt_c^\alpha \pt_\gamma^jV^i_{\theta}\right|\le C, \quad -\frac{1}{2}<x<\frac{1}{2},\quad i=1,2.
	\]
	The above implies that  for all $|\alpha|+j \ge 1$, $\pt_c^\alpha \pt_\gamma^j  V^i_{\theta}\in \mathbf{M}_1$, $i=1,2$, so $V^1,V^2\in C^{\infty}(K, \mathbf{X})$.
	
	Using the expressions of $V^3$ in (\ref{eq_basis}), Lemma \ref{lem3_4_ab},  the estimates (\ref{eq3_4_ker_9}), (\ref{eq3_4_ker_10}), (\ref{sec34:eq:Vsmooth:temp}) and Proposition \ref{propA_1} (4), we have that for all $|\alpha|+j\ge 1$,
		
	\[
	   \left|\partial_{c}^{\alpha}\partial_{\gamma}^{j}V^3_{\phi}\right|=O(1),\quad 
	    \left|\frac{d}{dx}\partial_{c}^{\alpha}\partial_{\gamma}^{j}V^3_{\phi}(x)\right|= O(1) \left| \ln \frac{1+x}{3} \ln \frac{1-x}{3} \right|^{-2} (1-x^2)^{-1},
	\]
	\[
	    \left|\frac{d^2}{dx^2}\partial_{c}^{\alpha}\partial_{\gamma}^{j}V^3_{\phi}(x)\right|=O(1) \left| \ln \frac{1+x}{3} \ln \frac{1-x}{3} \right|^{-2} (1-x^2)^{-2}.
	\]
	Since $0<\epsilon<1/2$, there exists some $C=C(\alpha, j, K)$ such that for all  $(c,\gamma)\in K$,
	\[
	   | (1-x^2)^{\epsilon}\partial_{c}^{\alpha}\partial_{\gamma}^{j}V^3_{\phi}|\le C, \quad \left|(1-x^2)^{1+\epsilon}\frac{d}{dx}\partial_{c}^{\alpha}\partial_{\gamma}^{j}V^3_{\phi}\right|\le C, \quad \quad \left|(1-x^2)^{2+\epsilon}\frac{d^2}{dx^2}\partial_{c}^{\alpha}\partial_{\gamma}^{j}V^3_{\phi}\right|\le C.
	\]
	The above implies that for any $|\alpha|+j\ge 1$, $\partial_{c}^{\alpha}\partial_{\gamma}^{j}V^3_{\phi}\in \mathbf{M}_2$, so $V^3\in C^{\infty}(K,\mathbf{X})$. 
	
\end{proof}


Similar arguments as in Lemma \ref{sec32:lem:combine:bdd} imply the following lemma. 
\begin{lem}\label{sec34:lem:combine:bdd}
	There exists $C=C(K)>0$ such that for all $(c,\gamma)\in K\subset \subset I_{4,1}$,  $\beta=(\beta_1,\beta_2,\beta_3,\beta_4)\in \mathbb{R}^4$, and $V\in \mathbf{X}_1$, 
	\[
		\| V \|_{\mathbf{X}} + |\beta | 
		\leq C \| \sum_{i=1}^{4}\beta_i V^i_{c,\gamma}+ V \|_{\mathbf{X}}. 
	\] 
\end{lem}

\noindent 
{\bf Proof of Theorem \ref{sec34:thm:1}.}
Define a map $F: K \times \mathbb{R}^4 \times \mathbf{X}_1 \to \mathbf{Y}$ by
\[
	F(c,\gamma,\beta, V) = G(c,\gamma, \sum_{i=1}^4 \beta_i V^i_{c,\gamma} + V).
\]
By Proposition \ref{sec34:prop}, $G$ is a $C^{\infty}$ map from $K\times \mathbf{X}$ to $\mathbf{Y}$. Let $\tilde{U} = \tilde{U}(c,\gamma,\beta, V) = \sum_{i=1}^4 \beta_i V^i_{c,\gamma} + V$. Using Lemma \ref{sec34:lem:V:smooth}, we have $\tilde{U}\in C^{\infty}(K \times \mathbb{R}^4 \times \mathbf{X}_1, \mathbf{X})$. So it concludes that $F \in C^{\infty}(K \times \mathbb{R}^4 \times \mathbf{X}_1, \mathbf{Y})$. 

Next, by definition $F(c,\gamma,0, 0) = 0$ for all $(c,\gamma)\in K$. Fix some $(\bar{c}, \bar{\gamma}) \in K$, using Lemma \ref{sec34:lem:iso}, we have $F_{V}(\bar{c}, \bar{\gamma}, 0,0)=L_0^{\bar{c}, \bar{\gamma}}:\mathbf{X}_1\to \mathbf{Y}$ is an isomorphism.

Applying Theorem \ref{thm:IFT}, there exist some $\delta>0$ depending only on $K$ and a unique $V\in C^{\infty}(B_{\delta}(\bar{c}, \bar{\gamma})\times B_{\delta}(0), \mathbf{X}_1 )$, such that
\[
	F(c,\gamma,\beta, V(c,\gamma,\beta) ) = 0, \quad \forall (c,\gamma)\in B_{\delta}(\bar{c}, \bar{\gamma}), \beta \in B_{\delta}(0),
\]
and 
\[
	V(\bar{c}, \bar{\gamma},0)=0.
\]
The uniqueness part of Theorem \ref{thm:IFT} holds in the sense that there exists some $0<\bar{\delta}<\delta$,  such that $B_{\bar{\delta}}(\bar{c},\bar{\gamma},0,0)\cap F^{-1}(0) \subset  \{(c,\gamma,\beta,V(c,\gamma,\beta) ) |(c,\gamma) \in B_{\delta}(\bar{c}, \bar{\gamma}), \beta\in B_{\delta}(0)\}$.

\noindent
{\bf Claim.} There exists some $0 < \delta_1 < \frac{\bar{\delta}}{2}$, such that $V(c,\gamma,0)=0$ for every $(c,\gamma)\in B_{\delta_1}(\bar{c}, \bar{\gamma})$.

\noindent 
{\bf Proof of the Claim.}
Since $V(\bar{c}, \bar{\gamma},0) = 0$ and $V(c,\gamma,0)$ is continuous in $(c,\gamma)$, there exists some $0<\delta_1 < \frac{\bar{\delta}}{2}$, such that for all $(c,\gamma)\in B_{\delta_1}(\bar{c}, \bar{\gamma})$, $(c,\gamma,0, V(c,\gamma,0))\in B_{\bar{\delta}(\bar{c}, \bar{\gamma},0,0)}$. We know that for all $(c,\gamma)\in B_{\delta_1}(\bar{c},\bar{\gamma})$,
\[
	F(c,\gamma,0, 0) = 0, 
\]
and 
\[
	F(c,\gamma,0, V(c,\gamma,0))=0.
\]
By the above mentioned uniqueness result, $V(c,\gamma,0)=0$, for every $(c,\gamma)\in B_{\delta_1}(\bar{c}, \bar{\gamma})$.

Now we have $V\in C^{\infty}(B_{\delta_1}(\bar{c}, \bar{\gamma})\times B_{\delta_1}(0), \mathbf{X}_1(\bar{c}, \bar{\gamma}) )$, and 
\[
	F(c,\gamma,\beta,V(c,\gamma,\beta))=0, \quad \forall (c,\gamma)\in B_{\delta_1}(\bar{c}, \bar{\gamma}), \beta\in B_{\delta_1}(0).
\]
i.e. for any $(c,\gamma)\in B_{\delta_1}(\bar{c}, \bar{\gamma})$, $\beta\in B_{\delta_1}(0)$
\[
	G(c,\gamma,\sum_{i=1}^4 \beta_i V^i_{c,\gamma} + V(c,\gamma,\beta) )=0. 
\]
Take derivative of the above with respect to $\beta_i$ at $(c,\gamma,0)$, $1\le i\le 4$, we have
\[
	G_{\tilde{U}}(c,\gamma,0)(V^i_{c,\gamma}+\partial_{\beta_i}V(c,\gamma,0))=0.
\]
Since $G_{\tilde{U}}(c,\gamma,0)V^i_{c,\gamma}=0$ by Lemma \ref{sec34:lem:ker}, we have 
\[
	G_{\tilde{U}}(c,\gamma,0)\partial_{\beta_i}V(c,\gamma,0)=0.
\]
But $\partial_{\beta_i}V(c,\gamma,0)\in C^{\infty}(\mathbf{X}_1)$, so 
\[
	\partial_{\beta_i}V(c,\gamma,0)=0, \quad 1\le i\le 4.
\]
Since $K$ is compact, we can take $\delta_1$ to be a universal constant for each  $(c,\gamma)\in K$. So we have proved the existence of $V$ in Theorem \ref{sec34:thm:1}. 

Next, let $(c,\gamma)\in B_{\delta_1}(\bar{c}, \bar{\gamma})$. Let $\delta'$ be a small constant to be determined.  For any $U$ satisfies equation (\ref{sec31:eq:NSE}) with $U-U^{c,\gamma}\in \mathbf{X}$, and $\|U-U^{c,\gamma}\|_{ \mathbf{X}}\le \delta'$ there exist some $\beta_1,\beta_2\in \mathbb{R}$ and $V^*\in \mathbf{X}_1$ such that
\[
	U-U^{c,\gamma} = \sum_{i=1}^4 \beta_i V^i_{c,\gamma} + V^*.
\]
Then by Lemma \ref{sec34:lem:combine:bdd}, there exists some constant $C>0$ such that
\[
	\frac{1}{C}(|\beta| + \|V^*\|_{\mathbf{X}})
	\le \| \sum_{i=1}^4 \beta_i V^i_{c,\gamma} + V^*\|_{\mathbf{X}}\le \delta'.
\]
This gives $\|V^*\|_{\mathbf{X}}\le C\delta'$.

Choose $\delta'$ small enough such that $C\delta'<\delta_1$. We have the uniqueness of $V^*$. 
So  $V^*=V(c,\gamma,\beta)$ in (\ref{sec34:eq:U1}).
Theorem \ref{sec34:thm:1} is proved. 
\qed

\medskip

Now with Theorem \ref{sec32:thm:1}- \ref{sec34:thm:1} we can give the 

\medskip

\noindent
{\bf Proof of Theorem \ref{thm1}.}
Let $K$ be a compact subset of one of the sets $I_{k,l}$, $1\le k\le 8$ and $1\le l\le 3$,  where $I_{k,l}$ are the sets defined by (\ref{eq1_3}). 

For $(c,\gamma)\in K\cap I_{k,l}$ with $1\le k\le 4$ and $l=1$, let
\[
    (u_{\theta}(c,\gamma,\beta), u_{\phi}(c,\gamma,\beta))=\frac{1}{\sin\theta}\left(U^{c,\gamma}+\beta_1 V^3_{c,\gamma}+\beta_2 V^4_{c,\gamma}+V(c,\gamma,0, 0, \beta_1,\beta_2)\right) ,
\]
where $\beta=(\beta_1,\beta_2) \in B_{\delta}(0)$, $\delta,  V_{c,\gamma}^3, V^4_{c,\gamma}$ and $V(c,\gamma,0, 0,\beta_1,\beta_2)$ are as in Theorem \ref{sec32:thm:1}, Theorem \ref{sec33:thm:1}, Theorem \ref{sec33:thm:1'} and Theorem \ref{sec34:thm:1}.

For $(c,\gamma)\in K\cap  I_{k,l}$ with $1\le k\le 4$ and $l=2,3$, let
\[
    (u_{\theta}(c,\gamma,\beta), u_{\phi}(c,\gamma,\beta))=\frac{1}{\sin\theta}\left(U^{c,\gamma}+\beta_1 V^3_{c,\gamma}+\beta_2 V^4_{c,\gamma}+V(c,\gamma, 0, \beta_1,\beta_2)\right), 
\]
where $\beta=(\beta_1,\beta_2) \in B_{\delta}(0)$, $\delta,  V_{c,\gamma}^3, V^4_{c,\gamma}$ and $V(c,\gamma,0,\beta_1,\beta_2)$ are as in  Theorem \ref{sec32:thm:2}, Theorem \ref{sec32:thm:2'}, Theorem \ref{sec33:thm:2} and Theorem \ref{sec33:thm:2'}.

For $(c,\gamma)\in K\cap I_{k,l}$ with $5\le k\le 8$ and $1\le l\le 3$, let
\[
    (u_{\theta}(c,\gamma,\beta), u_{\phi}(c,\gamma,\beta))=\frac{1}{\sin\theta}\left(U^{c,\gamma}+\beta_1 V^3_{c,\gamma}+\beta_2 V^4_{c,\gamma}+V(c,\gamma, \beta_1,\beta_2)\right), 
\]
where $\beta=(\beta_1,\beta_2) \in B_{\delta}(0)$, $\delta, V_{c,\gamma}^3, V^4_{c,\gamma}$ and $V(c,\gamma,\beta_1,\beta_2)$ are as in Theorem \ref{sec32:thm:3}.

With  $(u_{\theta}(c,\gamma,\beta), u_{\phi}(c,\gamma,\beta))$ defined as above, $u_{r}$ defined by (\ref{eq_divefree}) and $p$ defined by (\ref{eq1_2}), the first part of Theorem \ref{thm1} follows from Theorem \ref{sec32:thm:1}- \ref{sec34:thm:1}.

For the second part of Theorem \ref{thm1}, recall $U^{c,\gamma}=\sin\theta u^{c,\gamma}$. It is not hard to check that if $(c,\gamma)\in \hat{I}$, then $U^{c,\gamma}_{\theta}(-1)>3$ or $U^{c,\gamma}_{\theta}(1)<-3$. Let $\{u^i\}$ be a sequence of solutions of (\ref{NS}) satisfying $||\sin\theta(u^i-u^{c,\gamma})||_{L^{\infty}(\mathbb{S}^2\setminus\{S,N\})}\to 0$ as $i\to \infty$. Let $U^i=\sin\theta u^i$ for all $i\in \mathbb{N}$. We have $||U_{\theta}^i-U^{\mu,\gamma}_{\theta}||_{L^{\infty}(-1,1)}\to 0$. By Theorem 1.3 of \cite{LLY1}, $U^i(\pm 1)$ must exists and is finite for every $i$. If $U^{c,\gamma}(-1)>3$, $U^{i}_{\theta}(-1)>3$ for large $i$.  If $U^{c,\gamma}(1)<-3$, $U^{i}_{\theta}(1)<-3$ for large $i$. Then by Theorem 1.4 in \cite{LLY1}, $U^i_{\phi}$ must be constants for large $i$, the theorem is then proved. 
\qed

\end{document}